\documentclass[reqno,11pt]{amsart}
\usepackage{amsmath, latexsym, amsfonts, amssymb, amsthm, amscd}
\usepackage{mathrsfs}
\usepackage{graphics,epsf,psfrag}
\usepackage[lofdepth,lotdepth]{subfig}
\usepackage{graphicx}
\usepackage{url}
\usepackage{hyperref}
\usepackage[yyyymmdd,hhmmss]{datetime}

\numberwithin{equation}{section}

\newtheorem{theorem}{Theorem}[section]
\newtheorem{lemma}[theorem]{Lemma}
\newtheorem{proposition}[theorem]{Proposition}
\newtheorem{assumption}[theorem]{Assumption}
\newtheorem{model}[theorem]{Model}

\newtheorem{remark}[theorem]{Remark}
\newtheorem{example}[theorem]{Example}
\newtheorem{definition}[theorem]{Definition}

\DeclareMathSymbol{\leqslant}{\mathalpha}{AMSa}{"36} % nicer `smaller or equal'
\DeclareMathSymbol{\geqslant}{\mathalpha}{AMSa}{"3E} % nicer `larger or equal'
\DeclareMathSymbol{\eset}{\mathalpha}{AMSb}{"3F}     % nicer `emptyset'
\renewcommand{\leq}{\;\leqslant\;}                   % redef. of < or =
\renewcommand{\geq}{\;\geqslant\;}                   % redef. of > or =
\makeatletter
\def\captionfont@{\footnotesize}
\def\captionheadfont@{\scshape}

\long\def\@makecaption#1#2{%
  \vspace{2mm}
  \setbox\@tempboxa\vbox{\color@setgroup
    \advance\hsize-6pc\noindent
    \captionfont@\captionheadfont@#1\@xp\@ifnotempty\@xp
        {\@cdr#2\@nil}{.\captionfont@\upshape\enspace#2}%
    \unskip\kern-6pc\par
    \global\setbox\@ne\lastbox\color@endgroup}%
  \ifhbox\@ne % the normal case
    \setbox\@ne\hbox{\unhbox\@ne\unskip\unskip\unpenalty\unkern}%
  \fi
  \ifdim\wd\@tempboxa=\z@ % this means caption will fit on one line
    \setbox\@ne\hbox to\columnwidth{\hss\kern-6pc\box\@ne\hss}%
  \else % tempboxa contained more than one line
    \setbox\@ne\vbox{\unvbox\@tempboxa\parskip\z@skip
        \noindent\unhbox\@ne\advance\hsize-6pc\par}%
\fi
  \ifnum\@tempcnta<64 % if the float IS a figure...
    \addvspace\abovecaptionskip
    \moveright 3pc\box\@ne
  \else % if the float IS NOT a figure...
    \moveright 3pc\box\@ne
    \nobreak
    \vskip\belowcaptionskip
  \fi
\relax
}
\makeatother
\def\writefig#1 #2 #3 {\rlap{\kern #1 truecm
\raise #2 truecm \hbox{#3}}}

\makeatletter
\let\orgdescriptionlabel\descriptionlabel
\renewcommand*{\descriptionlabel}[1]{%
  \let\orglabel\label
  \let\label\@gobble
  \phantomsection
  \edef\@currentlabel{#1}%
  \let\label\orglabel
  \orgdescriptionlabel{#1}%
}
\makeatother

\title[Interacting diffusions on inhomogeneous random graphs]
{Quenched asymptotics for interacting diffusions on inhomogeneous random graphs}

\author{Eric Lu\c{c}on}
\address{Laboratoire MAP5 (UMR CNRS 8145), Universit\'e Paris Descartes, Sorbonne Paris Cit\'e, 75270 Paris, France, \url{eric.lucon@paridescartes.fr}.
}


\date{\today}

\begin{document}

\begin{abstract}
The aim of the paper is to address the behavior in large population of diffusions interacting on a random, possibly diluted and inhomogeneous graph. This is the natural continuation of a previous work, where the homogeneous Erd\H os-R\'enyi case was considered. The class of graphs we consider includes disordered $W$-random graphs, with possibly unbounded graphons. The main result concerns a quenched convergence (that is true for almost every realization of the random graph) of the empirical measure of the system towards the solution of a nonlinear Fokker-Planck PDE with spatial extension, also appearing in different contexts, especially in neuroscience. The convergence of the spatial profile associated to the diffusions is also considered, and one proves that the limit is described in terms of a nonlinear integro-differential equation which matches the neural field equation in certain particular cases. 
%This is the natural continuation of a previous work \cite{Lucon:2018qy} where the elliptic case is considered.
%  \\[10pt]
%  2010 \textit{Mathematics Subject Classification: 35Q84, 37N25, 82C26}
%  \\[10pt]
%  \textit{Keywords: nonlinear Fokker-Planck equation, mean-field systems, excitable systems, FitzHugh-Nagumo model, Stuart-Landau model, positively invariant manifold, disorder-induced dynamics}
\end{abstract}

\maketitle
%\tableofcontents

\section{The model}
\subsection{Interacting diffusions on a graph}
For all $n\geq1$, consider the system of coupled diffusions $( \theta_{ 1, t}^{ (n)}, \ldots, \theta_{ n, t}^{ (n)})$ in $ \mathbb{ R}^{ d}$ ($d\geq1$)
\begin{equation}
\label{eq:odegene}
{\rm d}\theta_{i, t}^{ (n)}= c(\theta_{i, t}^{ (n)}){\rm d} t + \frac{\kappa_{i}^{ (n)}}{n}\sum_{j=1}^{ n} \xi_{ i, j}^{ (n)}\Gamma\left(\theta_{i, t}^{ (n)}, \theta_{j, t}^{ (n)}\right) {\rm d} t + \sigma{\rm d} B_{i, t},\ 0\leq t\leq T,\ i=1, \ldots, n\ .
\end{equation} 
The dynamics in \eqref{eq:odegene} is decomposed into three terms: a local dynamics, represented by $c(\cdot): \mathbb{ R}^{ d} \to \mathbb{ R}^{ d}$, a mean-field coupling (governed by the binary kernel $ \Gamma(\cdot, \cdot): \mathbb{ R}^{ d}\times \mathbb{ R}^{ d} \to \mathbb{ R}^{ d}$) and a noise term, in the presence of i.i.d. standard Brownian motions in $ \mathbb{ R}^{ d}$, $B_{ 1}, \ldots, B_{ n}$. Here $ \sigma$ is a constant, but possibly degenerate (even equally $0$) diffusion matrix. The time horizon $T$ is fixed (but arbitrary). 

In \eqref{eq:odegene}, the diffusions $( \theta_{ 1}^{ (n)}, \ldots, \theta_{ n}^{ (n)})$ no longer interact on the complete graph (as it is a common framework for a mean-field analysis) but through an nontrivial graph of interaction, encoded by the sequence $(\xi_{ i,j}^{ (n)})_{i, j=1, \ldots, n}$ in $\{0, 1\}^{ n^{ 2}}$. More precisely, we define the graph of interaction of \eqref{eq:odegene} as $ \mathcal{ G}^{ (n)}:=( \mathcal{ V}^{(n)},  \mathcal{ E}^{ (n)})$ with set of vertices $ \mathcal{ V}^{ (n)}=[n]:=\{1, \ldots, n\}$ and set of edges $ \mathcal{ E}^{ (n)}= \{(i,j)\in \mathcal{ V}^{ (n)}\times  \mathcal{ V}^{ (n)},\ \xi_{ i, j}^{ (n)}=1\}$. The aim of the paper is to analyse the large population behavior of \eqref{eq:odegene} for situations where the graph of interaction $ \mathcal{ G}^{ (n)}$ is possibly inhomogeneous. This paper is the natural continuation of \cite{MR3568168} where the homogeneous Erd\H os-R\'enyi case is considered.
\begin{remark}
Note that it would also be possible to include disordered coefficients $ c( \theta_{ i}, \omega_{ i})$ and $ \Gamma(\theta_{ i}, \omega_{ i}, \theta_{ j}, \omega_{ j})$ in \eqref{eq:odegene}, where $(\omega_{ i})_{ i\in \left[n\right]}$ is some i.i.d sequence independent of everything, as it is customary for Kuramoto-type models (see Section~\ref{sec:applications} below). Everything below works with this additional random environment up to an additional expectation w.r.t. this disorder, under appropriate moment conditions.
\end{remark}
\subsection{Construction of the interaction graph}
The construction of the graph $ \mathcal{ G}^{ (n)}$ goes back to the formalism of $W$-random graphs developed in \cite{LOVASZ2006933,MR2455626,MR2925382,Borgs:2014aa,borgs2018}, which has been used in particular in a series of papers \cite{MR3238494,10113716M1075831,10113717M1134007,Chiba:2016aa,MR3187677} on macroscopic limits for Kuramoto-type models (see Section~\ref{sec:applications}), in the deterministic case $ \sigma=0$. In addition to the fact that we consider here more general dynamics, the crucial point is the presence of noise in \eqref{eq:odegene} that changes considerably the analysis (in particular, the techniques used in \cite{10113717M1134007,Chiba:2016aa} for the convergence of the empirical measure when $ \sigma=0$ do not seem to be directly applicable to the case $ \sigma\neq 0$). The work that is closest to the present analysis is the recent \cite{2018arXiv180706898O} where annealed large deviations estimates are given in the case of bounded graphons.

Given a set $I$ of spatial variables, we associate to each vertex $i\in \left[n\right]$ a position variable $x_{ i}^{ (n)}\in I$ encoding some local inhomogeneity for the vertex $i$ in the graph $ \mathcal{ G}^{ (n)}$. Here, 
\begin{definition}
\label{def:I_ell}
$I$ is a closed subset of $ \mathbb{ R}^{ p}$ ($p\geq1$), endowed with a probability measure $ \ell$ with support $I$.
\end{definition}
\noindent
In many situations (see e.g. \cite{LOVASZ2006933,MR2455626}), the set of positions $I$ is taken to be $[0, 1]$, but most of the results presented below remain valid in more general cases, closer to situations where $x_{ i}^{ (n)}$ actually encodes some real spatial position of the particle $ \theta_{ i}^{ (n)}$. Spatial extensions of mean-field dynamics are particularly relevant in a context of neuroscience where one accounts for the spatial organization of neurons in the cortex (see \cite{LucSta2014,Muller:2015aa,Cabana:2015aa,CHEVALLIER2018} and references therein for further details). The way positions $(x_{ 1}^{ (n)}, \ldots, x_{ n}^{ (n)})$ are chosen in $I$ will be made precise later (see Assumptions~\ref{ass:deterministic_positions} and~\ref{ass:random_positions} below). For now, we suppose that these positions are deterministic. In the rest of the paper, we denote by $\ell_{ n}({\rm d}x)$ the empirical measure of the positions:
\begin{equation}
\label{eq:emp_measure_positions}
\ell_{ n}({\rm d}x):= \frac{ 1}{ n}\sum_{ k=1}^{ n} \delta_{ x_{ k}^{ (n)}}({\rm d}x).
\end{equation}
Then, we introduce a kernel $W_{ n}:I^{ 2}\to [0, 1]$ such that $W_{ n}(x_{ i}^{ (n)}, x_{ j}^{ (n)})\in[0, 1]$ represents the probability of the presence of the edge $ \xi_{ i, j}^{ (n)}$ in the graph $ \mathcal{ G}^{ (n)}$:
\begin{definition}
\label{def:regularity_network}
On a common probability space $( \Omega, \mathcal{ F}, \mathbb{ P})$, we give ourselves a family of random variables $ \left( \xi_{ i,j}^{ (n)}\right)_{ i, j\in \left[n\right]; n\geq1}$ on $ \Omega$, such that, under $ \mathbb{ P}$, for each $n\geq1$, $(\xi_{ i,j}^{ (n)})_{ i, j\in \left[n\right]}$ is a collection of independent Bernoulli random variables with parameter $W_{ n}(x_{ i}^{ (n)}, x_{ j}^{ (n)})$.
\end{definition}
In \eqref{eq:odegene}, the parameter $ \kappa_{i}^{ (n)}>0$ is a dilution parameter that compensates for the possible local sparsity of the graph $ \mathcal{ G}^{ (n)}$ around vertex $i$: vertices with fewer neighbors will have a larger dilution parameter. Note that each $ \kappa_{i}^{ (n)}$ may actually depend on the whole sequence of positions in the graph $ \mathcal{ G}^{ (n)}$: $ \kappa_{i}^{ (n)}= \kappa_{i}^{ (n)}(x_{ 1}^{ (n)}, \ldots, x_{ n}^{ (n)})$. In the following $ \Xi:=\left( \xi_{ i,j}^{ (n)}\right)_{ i, j\in \left[n\right]; n\geq1}$ and $\mathcal{ X}:=\left(x_{ i}^{ (n)}\right)_{  i\in \left[n\right]; n\geq1}$ stand for the whole sequence of connections and positions. For fixed $n\geq1$, we also write $ \underline{ x}:=\left(x_{ i}^{ (n)}\right)_{  i\in \left[n\right]}$ and $ \underline{ \xi}:= \left(\xi_{ i, j}^{ (n)}\right)_{ i, j\in \left[n\right]}$. In absence of ambiguity, we write $x_{ i}$ instead of $x_{ i}^{ (n)}$, $ \xi_{ i, j}$ instead of $ \xi_{ i,j}^{ (n)}$ and $ \theta_{ i}$ instead of $ \theta_{ i}^{ (n)}$. The notations $ \mathbf{ P}\left(\cdot\right)$ and $ \mathbf{ E} \left[\cdot\right]$ stand for the probability and expectation w.r.t. the randomness in the Brownian motions and initial condition in \eqref{eq:odegene}. We use both notations $ x\cdot y$ or $ \left\langle x\, ,\, y\right\rangle$ for the scalar product of $x, y\in \mathbb{ R}^{ d}$, and $ \left\vert x \right\vert$ denotes the Euclidean norm of $x$. $ \sigma^{ \dagger}$ stands for the transpose of the matrix $ \sigma$. The notation $ \left\langle \mu\, ,\, f\right\rangle:= \int f {\rm d}\mu$ is also used for the usual duality between a measure and some test function. 
\subsection{The macroscopic kernel}
In order to obtain a macroscopic limit as $n\to\infty$ for \eqref{eq:odegene}, we require some averaging for the probability field $W_{ n}(\cdot, \cdot)$: we assume the existence of a nonnegative measurable function $W:I^{ 2} \to [0, +\infty)$ so that the probability field $(W_{ n})_{ n\geq1}$, correctly renormalized by the dilution parameters $ \kappa_{i}^{ (n)}$, converges along the sequence $ \mathcal{ X}$ as $n\to \infty$ to the macroscopic kernel $W$ (anticipating on the definitions of the Section~\ref{sec:general_prop_chaos} below, what we rigorously mean is that $ \delta_{ n}( \underline{ x})$ in \eqref{hyp:Delta_n_1} goes to $0$ as $n \to \infty$). This assumption encodes some notion of graph convergence (in the sense of \cite{LOVASZ2006933,Borgs:2014aa}) that is discussed in Section~\ref{sec:comment_graph_convergence} below.
\begin{remark}
Without loss of generality, we suppose that one particle does not interact with itself, that is $ \xi_{ i,i}^{ (n)}=0$ for all $i\in \left[n\right]$. In the limit as $n\to\infty$, this boils down to the assumption that the macroscopic kernel $W$ is zero on the diagonal. We make theses assumptions throughout this work without further notice.
\end{remark}

\subsection{The McKean-Vlasov process and the nonlinear Fokker-Planck equation}
The natural limit of the particle system \eqref{eq:odegene} is then described by the nonlinear process $\bar \theta^{x}$ (at position $x\in I$) solution to
\begin{equation}
\label{eq:theta_nonlin}
\qquad {\rm d}\bar\theta^{x}_{t}= c(\bar \theta^{x}_{ t}){\rm d} t +   \int W(x, y)\Gamma\left(\bar\theta_{ t}^{x}, \tilde\theta\right) \nu_{ t}^{ y}({\rm d} \tilde\theta)\ell({\rm d} y) {\rm d} t + \sigma{\rm d} B_{t},\ 0\leq t\leq T.
\end{equation}
where, for fixed $(x, t)$, $ \nu^{x}_{ t}({\rm d}\theta)$ is the law of $ \bar\theta_{ t}^{x}$. It is standard to see that the joint law \[\nu( {\rm d} \theta, {\rm d} x)= \nu^{x}({\rm d} \theta)\ell({\rm d} x)\] of $( \bar \theta^{x}, x)$ solves the nonlinear Fokker-Planck equation
\begin{align}
\left\langle \nu_{t}\, ,\, \varphi\right\rangle &= \left\langle \nu_{0}\, ,\, \varphi\right\rangle + \int_{ 0}^{t} \left\langle \nu_{s}\, ,\, \frac{ 1}{ 2}\nabla_{ \theta}\left( \sigma \sigma^{ \dagger} \nabla_{ \theta}\varphi \right) + \nabla_{ \theta} \varphi(\cdot)\cdot c(\cdot)\right\rangle{\rm d} s \nonumber\\ &+ \int_{ 0}^{t} \left\langle \nu_{ s}({\rm d} \theta, {\rm d} x)\, ,\, \nabla_{ \theta} \varphi(\theta, x) \cdot\int W(x, y)\Gamma(\theta, \tilde \theta) \nu_{ s}({\rm d} \tilde \theta, {\rm d} y)\right\rangle{\rm d} s, \label{eq:McKeanVlasov}
\end{align}
where $ \varphi$ is a regular test function. Writing formally $ \nu_{ t}({\rm d}\theta, {\rm d}x)= q_{ t}(\theta, x) {\rm d}\theta \ell({\rm d}x)$, \eqref{eq:McKeanVlasov} is the weak formulation of 
\begin{equation}
\label{eq:FP}
\partial_{ t} q_{ t}= \frac{1}{ 2} \nabla_{ \theta} \left( \sigma \sigma^{ \dagger} \nabla_{ \theta} q_{ t}\right) - \nabla_{ \theta} \left(q_{ t} \left(c(\cdot)+\int \Gamma(\cdot, \theta^{ \prime}) W(\cdot, y) q_{ t}( \theta^{ \prime}, y) {\rm d} \theta^{ \prime}\ell({\rm d} y)\right)\right).
\end{equation}
The precise meaning we give to \eqref{eq:theta_nonlin} and \eqref{eq:McKeanVlasov} is given in Section~\ref{sec:well_posedness_intro} below. \eqref{eq:theta_nonlin} and \eqref{eq:FP} are spatially-extended versions of standard McKean-Vlasov models that are natural large population limits of mean-field particle systems such as \eqref{eq:odegene}. A recent interest in models with spatial extension similar to \eqref{eq:FP} has been shown in a neuroscience context (see e.g. \cite{LucSta2014,Cabana:2015aa,2018arXiv180706898O,Muller:2015aa} and references therein).
\section{Main assumptions and results}
\label{sec:main_results}
\subsection{General assumptions}
\label{sec:general_assumptions_nonlin}
\subsubsection{Assumption on the kernel $W$}
For any $r\geq1$, for any $x\in I$ such that $W(x, \cdot)\in L^{r}(I, \ell)$, denote by
\begin{equation}
\label{hyp:def_Wr}
\mathcal{ W}_{ r}(x):=\int W(x, y)^{ r} \ell({\rm d} y),
\end{equation}
We require the minimal assumption that
\begin{equation}
\label{hyp:bound_W_L1}
\left\Vert \mathcal{ W}_{ 2} \right\Vert_{ \infty}:= \sup_{ z\in I} \mathcal{ W}_{ 2}(z)<\infty.
\end{equation}
This implies in particular that $\left\Vert \mathcal{ W}_{ 1} \right\Vert_{ \infty}<\infty$: in the limit $n\to \infty$, the degree of each node $x\in I$ in the macroscopic graph $W$ remains uniformly bounded (\cite{2018arXiv180709989D}). Suppose also
\begin{equation}
\label{eq:W_nondegenerate}
\inf_{ x\in I} \mathcal{ W}(x)>0.
\end{equation}
\begin{remark}
\label{rem:W_nondegenerate}
A closer look to the proofs below shows that \eqref{eq:W_nondegenerate} can be discarded if one assumes more integrability on $W$ (for example, \eqref{eq:W_nondegenerate} is not needed if $W$ is bounded). 
\end{remark}
\begin{remark}
An important remark is that we do not suppose any symmetry of the kernels $W$ and $W_{ n}$, nor that we suppose that $W$ and $W_{ n}$ are simple functions of the distance $x-y$ (this is a natural hypothesis if one thinks of applications in neuroscience, as the mutual influence between neuron $i$ on neuron $j$ need not be symmetric). There are also some interesting examples where $W$ is not symmetric, even if $ \mathcal{ G}^{ (n)}$ might be (see Section~\ref{sec:convergence_graphs_examples}). Note that the proof of Theorem~\ref{theo:identification_non_compact} below requires to consider asymmetric kernels (see \eqref{eq:tilde_WM}).
\end{remark}

\subsubsection{Assumptions on the coefficients $ \Gamma$ and $c$}
We suppose that $ (\theta, \tilde{ \theta}) \mapsto \Gamma( \theta, \tilde{ \theta})$ and $ \theta \mapsto c(\theta)$ are twice differentiable on $ \mathbb{ R}^{ d}$ with continuous derivatives and that $ \Gamma$ is Lipschitz continuous with sublinear behavior: there exists a constant $L_{ \Gamma}>0$  such that
\begin{align}
\left\vert \Gamma( \theta_{ 1}, \theta_{ 2}) - \Gamma( \tilde{ \theta}_{ 1},  \tilde{ \theta}_{ 2}) \right\vert  &\leq L_{  \Gamma} \left(\left\vert \theta_{ 1} - \tilde{ \theta}_{ 1} \right\vert + \left\vert \theta_{ 2} - \tilde{ \theta}_{ 2} \right\vert \right),\ \theta_{ 1}, \tilde{ \theta}_{ 1}, \theta_{ 2}, \tilde{ \theta}_{ 2}\in \mathbb{ R}^{ d} \label{hyp:Gamma_Lip_1},\\
\left\vert \Gamma(\theta, \tilde{ \theta}) \right\vert &\leq L_{ \Gamma} \left(1+ \left\vert \theta \right\vert + \left\vert \tilde{ \theta} \right\vert\right),\ \theta, \tilde{ \theta}\in \mathbb{ R}^{ d}.\label{hyp:Gamma_bound}
\end{align}
We require that $c(\cdot)$ is one-sided Lipschitz: there exists a constant $L_{ c}>0$ such that
\begin{equation}
\label{hyp:c_onesidedLip}
\left\langle \theta - \tilde{ \theta}\, ,\, c(\theta) - c(\tilde{ \theta})\right\rangle\leq L_{ c} \left\vert \theta- \tilde{ \theta} \right\vert^{ 2},\ \theta, \tilde{ \theta}\in \mathbb{ R}^{ d}.
\end{equation}
We also suppose some polynomial control on $c(\cdot)$: there exists $k\geq 2$ such that
\begin{equation}
\label{hyp:polynomial_control_c}
\sup_{ \theta\in \mathbb{ R}^{ d}} \frac{ \left\vert c(\theta) \right\vert}{ 1+ \left\vert \theta \right\vert^{ k}}< \infty.
\end{equation}
Unless specified otherwise, we only assume \eqref{hyp:Gamma_Lip_1}, \eqref{hyp:Gamma_bound}, \eqref{hyp:c_onesidedLip} and \eqref{hyp:polynomial_control_c}. Nonetheless, for some of the results of the paper, we may restrict for simplicity to a generic subset of these assumptions
\begin{model}[Polynomial interactions]
\label{mod:polynomial}
A particular case of the previous assumptions is to require that $c$ is polynomial of degree smaller than $k$ satisfying \eqref{hyp:c_onesidedLip} and that $ \Gamma$ is either bounded or linear: $ \Gamma(\theta, \tilde{\theta})= \Gamma\cdot \left(\theta- \tilde{\theta}\right)$ for some (possibly degenerate) matrix $ \Gamma$.
\end{model}
We give in Section~\ref{sec:applications} below several dynamics satisfying the present hypotheses, a significant application to have in mind being FitzHugh-Nagumo oscillators with electrical synapses (see Section~\ref{sec:applications}, item (2)). 
\subsubsection{Assumption on the initial condition}
We assume that \eqref{eq:McKeanVlasov} is endowed with an initial condition $ \nu_{ 0}({\rm d}\theta, {\rm d}x)$ of the form
\begin{equation}
\label{hyp:form_initial_condition_nu0}
\nu_{ 0}({\rm d}\theta, {\rm d}x)= \nu_{ 0}^{ x}({\rm d}\theta) \ell({\rm d}x)
\end{equation}
with some uniform a priori control on its moments: for $k$ given by \eqref{hyp:polynomial_control_c}, suppose that
\begin{equation}
\label{hyp:second_moment_nu0}
\sup_{x\in I} \int \left\vert \theta \right\vert^{ 2k} \nu_{ 0}^{ x}({\rm d} \theta)<+\infty.
\end{equation}
If $ (S, d)$ is a Polish space, let $w_{ 1}(\cdot, \cdot)$ be the usual Wasserstein distance \cite{MR2459454} on $S$: for any probability measures $ \nu_{ 1}, \nu_{ 2}$ on $S$,
\begin{equation}
\label{eq:wasserstein}
w_{ 1}(\nu_{ 1}, \nu_{ 2}):= \inf_{ \pi} \left\lbrace \int d \left(\theta_{ 1}, \theta_{ 2}\right) \pi({\rm d} \theta_{ 1}, {\rm d} \theta_{ 2})\right\rbrace,
\end{equation}
where the infimum is taken on all couplings $ \pi$ on $ S\times S$ with marginals $ \nu_{ 1}$ and $ \nu_{ 2}$. In \eqref{hyp:nu_0_Lip} below, we take $S= \mathbb{ R}^{ d}$, but we will also use the same definition later for $S= \mathbb{ R}^{ d}\times I$ when necessary. We assume here that there exist $L_{ 0}>0$ and $ \iota_{ 1}\in(0, 1]$ such that
\begin{equation}
\label{hyp:nu_0_Lip}
w_{ 1} \left(\nu_{ 0}^{ x}, \nu_{ 0}^{ y}\right) \leq L_{ 0} \left\vert x-y \right\vert^{ \iota_{ 1}},\ x,y\in I.
\end{equation}
We suppose finally that the initial condition of the particle system \eqref{eq:odegene} is such that $( \theta_{ 1, 0}^{ (n)}, \ldots, \theta_{ n, 0}^{ (n)})$ are independent, with respective law  $\theta_{ k, 0}^{ (n)}\sim \nu_{ 0}^{x_{ k}}({\rm d}\theta)$, for $k\in \left[n\right]$.
\subsection{Well-posedness on the nonlinear Fokker-Planck equation and a priori estimates} 
\label{sec:well_posedness_intro}
Consider $ \mathcal{ M}$ the set of probability measures $ \nu$ on $ \mathcal{ C}([0, T], \mathbb{ R}^{ d})\times I$ with marginals on $I$ equal to $\ell$. Since $ \mathcal{ C}([0, T], \mathbb{ R}^{ d})$ is Polish, it follows from the disintegration Theorem \cite{MR1932358} that any $ \nu\in \mathcal{ M}$ may be written as $ \nu({\rm d}\theta, {\rm d}x)= \nu^{ x}({\rm d}\theta) \ell({\rm d}x)$. We endow $ \mathcal{ M}$ with the following Wasserstein-type metric \cite{SznitSflour}
\begin{equation}
\label{eq:wasserstein_ptfixe}
 \delta_{ T}( \nu, \mu):= \sup_{ x\in I} \inf_{ \pi} \left\lbrace \sup_{ s\leq T} \int \left\vert \vartheta_{1, s}^{ x}- \vartheta_{2, s}^{ x} \right\vert^{ 2k} \pi ( {\rm d}\vartheta_{ 1}, {\rm d} \vartheta_{ 2})\right\rbrace^{ \frac{ 1}{ 2k}},
\end{equation} where the infimum is taken over all couplings $ \pi$ under which $ \vartheta_{1}^{ x}\sim \nu^{ x}$ and $ \vartheta_{ 2}^{ x}\sim \mu^{ x}$, for $\ell$-almost every $x\in I$.
\begin{proposition}
\label{prop:PDE_wellposed}
Under the hypotheses of Section~\ref{sec:general_assumptions_nonlin}, there exists a unique weak solution to \eqref{eq:McKeanVlasov} in $ \mathcal{ M}$ with initial condition $ \nu_{ 0}({\rm d}\theta, {\rm d}x)= \nu_{ 0}^{x}({\rm d}\theta) \ell({\rm d}x)$. This solution $ \nu$ is such that for $\ell$-almost every $x\in I$, $ \nu^{ x}({\rm d}\theta)$ is the law of the nonlinear process $( \theta_{ t}^{ x})_{ t\in[0, T]}$ given by \eqref{eq:theta_nonlin}.
\end{proposition}
\begin{remark}
\label{rem:apriori_nonlin}
A byproduct of Proposition~\ref{prop:PDE_wellposed} is the following: under the hypotheses of Proposition~\ref{prop:PDE_wellposed}, there exists a constant $C_{ 0}$, only depending on $ \Gamma, c, W, \sigma, T$ and $ \nu_{ 0}$, such that
\begin{equation}
\label{eq:apriori_bound_nonlin}
\sup_{ x\in I} \mathbf{ E} \left[\sup_{ s\leq T}\left\vert \bar \theta_{ s}^{x} \right\vert^{ 2k}\right] \leq C_{ 0}.
\end{equation}
Similarly, it is standard to prove, under the same hypotheses, a similar estimate for the particle system \eqref{eq:odegene}:
\begin{equation}
\label{eq:apriori_bound_odegene}
\sup_{ i\in \left[n\right]} \mathbf{ E} \left[\sup_{ s\leq T}\left\vert \bar \theta_{ i, s}^{(n)} \right\vert^{ 2k}\right] \leq C_{ 0}.
\end{equation}
\end{remark}
The proof of Proposition~\ref{prop:PDE_wellposed} is standard and relies on a fixed-point argument \cite{SznitSflour} on the McKean-Vlasov diffusion \eqref{eq:theta_nonlin}. Existence and uniqueness in \eqref{eq:theta_nonlin} provides existence of a solution to \eqref{eq:McKeanVlasov}. Uniqueness in \eqref{eq:McKeanVlasov} comes from a propagator method. Similar well-posedness results for spatially-extended McKean-Vlasov processes may be found in \cite{LucSta2014,Muller:2015aa,2018arXiv180706898O}. Proposition~\ref{prop:PDE_wellposed}, as well as some further regularity estimates concerning $ \nu$, is proven in Appendix~\ref{sec:wellposed_PDE}.
\subsection{A general propagation of chaos estimate}
\label{sec:general_prop_chaos}
In this paragraph, we fix a sequence of positions $(x_{ i})_{ i\in \left[n\right]}$ (that is supposed to be deterministic in Theorem~\ref{theo:conv_general} below), a probability field $W_{ n}(x_{ i}, x_{ j})$ and a kernel $W$. We are interested in the approximation of the microscopic system \eqref{eq:odegene} by its mean field limit \eqref{eq:theta_nonlin}: let $(\bar \theta_{ 1}^{x_{ 1}}, \ldots, \bar \theta_{ n}^{x_{ n}})$, $n$ independent copies of the nonlinear process driven by the same Brownian motions $(B_{ 1}, \ldots, B_{ n})$, with the same positions $x_{ i}$ and initial conditions as in \eqref{eq:odegene}. For simplicity, we write $ \bar \theta_{ i, s}$ in place of $ \bar \theta_{ i, s}^{x_{ i}}$.  
\subsubsection{Assumptions on the graph $ \mathcal{ G}^{ (n)}$}
We suppose some uniformity in the dilution parameters $( \kappa_{i}^{ (n)})_{ i\in \left[n\right]}$, namely the existence of $ \kappa_{ n}\geq1$ and $w_{ n}\in(0, 1]$ such that
\begin{align}
\kappa_{ \infty}^{ (n)}( \underline{ x}) := \max_{ i\in \left[n\right]} \left( \kappa_{i}^{ (n)}( \underline{ x})\right)& \leq \kappa_{ n}, \label{hyp:alphan_VS_an}\\
\max_{ i, j\in \left[n\right]} \left( W_{ n}(x_{ i}, x_{ j})\right)&\leq w_{ n},\label{hyp:def_w_n_infty}
\end{align}
satisfying, as $n\to \infty$,
\begin{align}
\frac{ 1}{ \kappa_{ n}} &\leq w_{ n} \leq 1\label{hyp:compare_wn_alphan},\\
\kappa_{ n}^{ 2} w_{ n} &= o \left( \frac{ n}{ \log(n)}\right), \text{ as } n\to\infty.\label{hyp:alpha_n_infty}
\end{align}
\begin{remark}
Assumptions \eqref{hyp:def_w_n_infty} and \eqref{hyp:compare_wn_alphan} are mostly technical, since it is always possible to take $w_{ n}=1$. Nonetheless, there are simple cases where it is natural to take $w_{ n}\to 0$ as $n\to\infty$: consider $W_{ n}(x_{ i}, x_{ j})\equiv w_{ n}= \rho_{ n}$ for some $ \rho_{ n} \xrightarrow[ n\to\infty]{}0$ (this corresponds to a uniform diluted Erd\H os-R\'enyi graph $ \mathcal{ G}^{ (n)}$, see \cite{MR3568168}). In this case, it is natural to renormalize the sum in \eqref{eq:odegene} by the mean degree of each vertex (equal to $n \rho_{ n}$), so that we take $ \kappa_{i}^{ (n)}\equiv \kappa_{ n}= \frac{ 1}{ \rho_{ n}}$ for all $i\in \left[n\right]$. Then, \eqref{hyp:alpha_n_infty} boils down to the condition $ \kappa_{ n}=_{ n\to\infty} o \left( \frac{ n}{ \log(n)}\right)$, which is exactly the condition found in \cite{MR3568168}, Eq. (1.12) in the Erd\H os-R\'enyi case. A general extension of this simple case is considered in Section~\ref{sec:dense_bounded_graphons}.
\end{remark}
We define
\begin{equation}
\delta_{ n}( \underline{ x}):=\sup_{ i\in \left[n\right]} \left(\frac{1}{ n} \sum_{ k=1}^{ n} \left\vert \kappa_{i}^{ (n)}W_{ n}(x_{ i}, x_{ k}) - W(x_{ i}, x_{ k})\right\vert\right).\label{hyp:Delta_n_1}
\end{equation}
The main assumption of this paragraph is
\begin{equation}
\delta_{ n}( \underline{ x})\xrightarrow[ n\to \infty]{} 0.\label{hyp:Delta1_to_0}
\end{equation}
The convergence \eqref{hyp:Delta1_to_0} of the microscopic probability field $W_{ n}$ (properly renormalized by $ \kappa^{ (n)}$) to the macroscopic kernel $W$ encodes some notion of convergence of the underlying graph $ \mathcal{ G}^{ (n)}$ as $ n\to \infty$. Further comments on this point are made in \S~\ref{sec:comment_graph_convergence}. In the remaining of the paper, we adopt the following definition
\begin{definition}
\label{def:convergence_graph}
We say that $( \mathcal{ G}^{ (n)}, \kappa^{ (n)})$ converges to $W$ as $n\to\infty$ if Assumptions \eqref{hyp:alphan_VS_an}, \eqref{hyp:def_w_n_infty}, \eqref{hyp:compare_wn_alphan}, \eqref{hyp:alpha_n_infty} and \eqref{hyp:Delta1_to_0} are true.
\end{definition}
General examples of converging graphs are given in Section~\ref{sec:convergence_graphs_examples}. 
\subsubsection{Assumptions on the kernel $W$} 
The second set of assumptions concerns regularity estimates on the limiting kernel $W$. The following notations are used throughout the paper: 
\begin{align}
[\Gamma]_{u}(\theta, x)&:= \int \Gamma(\theta, \tilde\theta) \nu_{ u}^{x}({\rm d} \tilde\theta),\ \theta\in \mathbb{ R}^{ d}, x\in I, u\geq0.\label{eq:G}\\
\Upsilon_{ t}(x, y, z)&:= \int_{ 0}^{t} \int \left\langle [ \Gamma]_{ u}(\theta, y)\, ,\, [ \Gamma]_{ u}(\theta, z)\right\rangle \nu_{ u}^{ x}({\rm d}\theta){\rm d}u,\ x,y,z\in I, t\geq0.\label{eq:upsilon_t}
\end{align}
A priori controls on $ \left[ \Gamma\right]$ and $ \Upsilon$ are given in Lemma~\ref{lem:regul_Upsilon} below. Define (recall the definition of $\ell_{ n}({\rm d}x)$ in \eqref{eq:emp_measure_positions}), for $i\in \left[n\right]$
\begin{equation}
\label{hyp:epsilon_i_123}
\begin{split}
\epsilon_{ n, T}^{ (1, i)}( \underline{ x})&:= \int W(x_{ i}, y) W(x_{ i}, z) \Upsilon_{ T}(x_{ i}, y, z) \left\lbrace\ell_{ n}({\rm d}y)\ell_{ n}({\rm d}z)- \ell({\rm d} y)\ell({\rm d} z)\right\rbrace ,\\
\epsilon_{ n, T}^{ (2, i)}( \underline{ x})&:= \int W(x_{ i}, y) W(x_{ i}, z) \Upsilon_{ T}(x_{ i}, y, z) \left\lbrace \ell_{ n}({\rm d}y) - \ell({\rm d}y)\right\rbrace\ell({\rm d}z),\\
\epsilon_{ n, T}^{ (3, i)}( \underline{ x})&:= \int W(x_{ i}, y)W(x_{ i}, z) \Upsilon_{ T}(x_{ i}, y, z)\ell({\rm d}y) \left\lbrace \ell_{ n}({\rm d}z)-\ell({\rm d} z)\right\rbrace.
\end{split}
\end{equation}
We require that
\begin{equation}
\epsilon_{ n, T}^{ (m)}( \underline{ x}):=\sup_{ i\in \left[n\right]} \left\vert \epsilon_{ n, T}^{ (m, i)}( \underline{ x}) \right\vert\xrightarrow[ n\to \infty]{} 0, \text{ for }m=1,2,3,\label{hyp:Delta3_to_0}
\end{equation}
as well as
\begin{equation}
\sup_{ n\geq1} \sup_{ i\in \left[n\right]} \int W(x_{ i}, y) \ell_{ n}({\rm d}y)=\sup_{ n\geq1} \sup_{ i\in \left[n\right]} \frac{ 1}{ n} \sum_{ k=1}^{ n} W(x_{ i}, x_{ k})<+\infty.\label{hyp:bound_W_L1_sum}
\end{equation}
Assumptions \eqref{hyp:Delta3_to_0} and \eqref{hyp:bound_W_L1_sum} capture a notion of regularity of the macroscopic kernel $W$: \eqref{hyp:bound_W_L1_sum} is the discrete counterpart of \eqref{hyp:bound_W_L1} (when $r=1$) and (forgetting about the factor $\Upsilon_{ T}$ in \eqref{hyp:epsilon_i_123}) \eqref{hyp:Delta3_to_0} essentially says that various empirical means in $W$ converge to their expectation. Hence, we adopt the following definition
\begin{definition}
\label{def:graphon_regular}
We say that the kernel $W$ is regular along the sequence of positions $\mathcal{ X}$ if \eqref{hyp:Delta3_to_0} and \eqref{hyp:bound_W_L1_sum} are true.
\end{definition}
We are now in position to state the first main result of the paper:
\begin{theorem}
\label{theo:conv_general}
Fix $T\geq0$. Suppose that the hypotheses of Section~\ref{sec:general_assumptions_nonlin} are true and that $( \mathcal{G}^{ (n)}, \kappa^{ (n)}, \mathcal{ X}, W)$ satisfy Definitions~\ref{def:convergence_graph} and ~\ref{def:graphon_regular}. In this case,
\begin{equation}
\label{eq:conv_general}
\sup_{ i\in \left[n\right]}\mathbf{ E}\left[ \sup_{ s\leq T} \left\vert \theta_{ i, s} - \bar \theta_{ i, s} \right\vert^{ 2}\right] \to 0 \text{ as }n\to \infty,
\end{equation}
for almost every realization of the connectivity sequence $ \Xi$ given by Definition~\ref{def:regularity_network}.
\end{theorem}
Theorem~\ref{theo:conv_general} is proven in Section~\ref{sec:proof_conv_general}. We detail in Section~\ref{sec:convergence_graphs_examples} generic examples of graphs that are regular and convergent in the sense of Definitions~\ref{def:convergence_graph} and~\ref{def:graphon_regular}. The rest of the present section is organized as follows: a byproduct of Theorem~\ref{theo:conv_general} concerns the convergence of the empirical measure of \eqref{eq:odegene} (Section~\ref{sec:conv_empirical_measure_intro}). A different approach to the macroscopic description of \eqref{eq:odegene} is given in Section~\ref{sec:convergence_profile_intro} and links between the two approaches are provided in Section~\ref{sec:identification_intro}. A discussion on the notion of graph convergence encoded by assumption \eqref{hyp:Delta1_to_0} is given in Section~\ref{sec:comment_graph_convergence}. Comments on applications and links with existing literature are given in Section~\ref{sec:applications}.

\subsection{Convergence of the empirical measure}
\label{sec:conv_empirical_measure_intro}
It is standard (see e.g. \cite{MR3568168} for a similar result in the Erd\H os-R\'enyi case) to derive from Theorem~\ref{theo:conv_general} the convergence of the empirical measure of the system \eqref{eq:odegene}
\begin{equation}
\label{eq:emp_measure}
\nu_{ n, t}( {\rm d} \theta, {\rm d} x) = \frac{ 1}{ n} \sum_{ i=1}^{ n} \delta_{ (\theta_{ i, t}^{ (n)}, x_{ i}^{ (n)})}({\rm d}\theta, {\rm d}x),\ t\geq0.
\end{equation}
to the solution $ \nu$ to the nonlinear Fokker-Planck equation \eqref{eq:McKeanVlasov}. 
\subsubsection{Assumptions}
In addition to the hypotheses of Section~\ref{sec:general_assumptions_nonlin}, we assume that there exist $ L_{ W}>0$ and $ \iota_{ 2}\in(0, 1]$ such that
\begin{equation}
\label{hyp:Int_W_Lip}
\delta \mathcal{ W}(x, y):= \int \left\vert W(x, z)  - W(y,z)\right\vert{\rm d} z \leq L_{ W} \left\vert x-y \right\vert^{ \iota_{ 2}},\ x,y\in I.
\end{equation} 
When $I$ is not bounded, we suppose $ \iota_{ 1}= \iota_{ 2}$ (recall \eqref{hyp:nu_0_Lip}) and denote in any case $ \iota:= \iota_{ 1}\wedge \iota_{ 2}\in (0, 1]$. Let $ \mathcal{ H}_{ \iota}$ be the set of $\iota$-H\"older functions on $ I$
\begin{equation}
\mathcal{ H_{ \iota}}:= \left\lbrace \varphi: I \to \mathbb{ R}, \left\Vert \varphi \right\Vert_{ \mathcal{ H}_{ \iota}} < \infty\right\rbrace
\end{equation}
where $\left\Vert \varphi \right\Vert_{ \mathcal{ H}_{ \iota}}:= \sup_{ x\in I} \left\vert \varphi(x) \right\vert + \sup_{ x\neq y} \frac{ \left\vert \varphi(x)- \varphi(y) \right\vert}{ \left\vert x-y \right\vert^{ \iota} }$.
Denote by $d_{ \mathcal{ H}_{ \iota}}(\cdot, \cdot)$ the distance given by, for every probability measures $ \mu$ and $ \nu$ on $ I$,
\begin{equation}
d_{ \mathcal{ H}_{ \iota}}( \mu, \nu) = \sup_{ \varphi\in \mathcal{ H}_{ \iota}, \left\Vert \varphi \right\Vert_{ \mathcal{ H}_{ \iota}}\leq 1} \left\vert \left\langle \mu- \nu\, ,\, \varphi\right\rangle \right\vert.
\end{equation}
We suppose that the empirical measure of the positions $ \ell_{ n}$ satisfies
\begin{equation}
\label{hyp:conv_empirical_measure_positions}
d_{ \mathcal{ H}_{ \iota}}(\ell_{ n}, \ell) \xrightarrow[ n\to\infty]{} 0.
\end{equation}
\subsubsection{Result}
Under these assumptions, we state the convergence of the empirical measure \eqref{eq:emp_measure} in the following simple way:
\begin{theorem}
\label{theo:conv_empirical_measure}
Under the hypotheses of Theorem~\ref{theo:conv_general} and the hypotheses made in Section~\ref{sec:conv_empirical_measure_intro} above, for any continuous function $ \varphi: \mathbb{ R}^{ d} \times I \to \mathbb{ R}$ such that for some $C_{ \varphi}>0$, for any $ \theta, \tilde{ \theta}\in \mathbb{ R}^{ d}$, $x,y\in I$, $ \left\vert \varphi(\theta, x) - \varphi(\tilde{ \theta}, y)\right\vert \leq C_{ \varphi} \left(\left\vert \theta- \tilde{ \theta} \right\vert+ \left\vert x-y \right\vert^{ \iota}\right)$ and $ \sup_{ x\in I}\left\vert \varphi(\theta, x) \right\vert \leq C_{ \varphi} (1+ \left\vert \theta \right\vert)$, the following convergence holds:
\begin{equation}
\sup_{ t\in[0, T]} \mathbf{ E} \left[\left\vert \left\langle \nu_{ n, t} - \nu_{ t}\, ,\, \varphi\right\rangle \right\vert^{ 2}\right] \xrightarrow[ n\to\infty]{}0,
\end{equation}
for almost every realization of the connectivity sequence $ \Xi$ given by Definition~\ref{def:regularity_network}.
\end{theorem}
Theorem~\ref{theo:conv_empirical_measure} is proven in Section~\ref{sec:proof_conv_empirical_measure}. 
\subsection{The nonlinear spatial profile}
\label{sec:convergence_profile_intro}
In \cite{MR3238494,MR3187677,10113716M1075831,10113717M1134007}, a different approach to the large population behavior of \eqref{eq:odegene} is considered. The point of view here is to consider the deterministic macroscopic spatial profile $ ( \psi(\cdot, t))_{ t\in [0, T]}$ that, in our context, solves the following nonlinear integro-differential equation
\begin{equation}
\label{eq:heat_equation}
\partial_{ t}\psi(x, t)= c( \psi(x, t)) + \int_{I} \Gamma( \psi(x, t) , \psi(y, t)) W(x, y)\ell({\rm d}y).
\end{equation}
In the context of \cite{MR3187677}, \eqref{eq:heat_equation} is referred to as the nonlinear heat equation on the graph $W$. For FitzHugh-Nagumo dynamics with linear interaction, \eqref{eq:heat_equation} corresponds to the reaction-diffusion equation addressed in the recent work \cite{2018arXiv180401263C}. We consider here weak solutions to \eqref{eq:heat_equation} in the sense of the following definition: if $ \mathcal{ C}([0, T], L^{ k}(I, \ell))$ is the set of continuous functions with values in $L^{ k}(I, \ell)$ where $k\geq2$ is given in \eqref{hyp:second_moment_nu0},
\begin{definition}
\label{def:weak_sol_heat_eq}
We say that $ \psi(\cdot, t)_{ t\in[0, T]} \in \mathcal{ C}([0, T], L^{ k}(I,\ell))$ is a weak solution to \eqref{eq:heat_equation} if for all regular test functions $J:I\to \mathbb{ R}^{ d}$, for all $t\in[0,T]$, we have
\begin{multline}
\label{eq:heat_equation_weak}
\int_{I} \left\langle \psi(x, t)\, ,\, J(x)\right\rangle \ell({\rm d}x)= \int_{I} \left\langle \psi(x, 0)\, ,\, J(x)\right\rangle \ell({\rm d}x) + \int_{ 0}^{t} \int_{I} \left\langle c( \psi(x, s))\, ,\, J(x)\right\rangle \ell({\rm d}x) {\rm d}s \\+ \int_{ 0}^{t}\int_{I^{ 2}} \left\langle \Gamma( \psi(x, s) , \psi(y, s))\, ,\, J(x)\right\rangle W(x, y)\ell({\rm d}y) \ell({\rm d}x){\rm d}s.
\end{multline}
\end{definition}
We first state a uniqueness result for \eqref{eq:heat_equation_weak}:
\begin{proposition}
\label{prop:uniqueness_weak_solution}
Under the assumptions of Section~\ref{sec:general_assumptions_nonlin}, for any $ \psi_{ 0} \in L^{ k}(I, \ell)$, there is at most one weak solution in $\mathcal{ C}([0, T], L^{ k}(I, \ell))$ to \eqref{eq:heat_equation_weak} with initial condition $ \psi_{ 0}$.
\end{proposition}
\subsubsection{Assumptions}
For the rest of Section~\ref{sec:convergence_profile_intro}, we restrict ourselves to the case $I:=[0,1]$, endowed with its Lebesgue measure $\ell({\rm d}x):= {\rm d}x$ and where $x_{ i}^{ (n)}:= \frac{ i}{ n}$ for $i\in \left[n\right]$ (see Section~\ref{sec:regularity_kernel} for further details). Following the approach of \cite{MR3187677}, it is possible to consider the spatial field:
\begin{equation}
\label{eq:spatial_profile}
\theta_{ n}(x, t):= \theta_{ \lfloor  nx +1\rfloor, t}^{ (n)}=\sum_{ i=1}^{ n} \theta_{ i, t}^{ (n)}\textbf{ 1}_{ [ x_{ i-1}^{ (n)}, x_{ i}^{ (n)})}(x), \ x\in I,\ t\geq0.
\end{equation}
We restrict here for simplicity to Model~\ref{mod:polynomial}. We suppose that the hypotheses of Section~\ref{sec:general_assumptions_nonlin} hold and that 
$( \mathcal{G}^{ (n)}, \kappa^{ (n)})$ converges to $W$ in the sense of Definition~\ref{def:convergence_graph}. We require the regularity of $(W, \mathcal{ X})$ in the sense of Definition~\ref{def:graphon_regular} together with the following supplementary condition:
\begin{align}
\sum_{ i,j=1}^{ n}\int_{\frac{ i-1}{ n}}^{\frac{ i}{ n}}\int_{\frac{ j-1}{ n}}^{\frac{ j}{ n}} \left\vert W \left(\frac{ i}{ n}, \frac{ j}{ n}\right) - W(x, y) \right\vert^{ 2} {\rm d}x{\rm d}y&\to 0, \label{hyp:Int_Wn_W_2}
\end{align}
\subsubsection{Result} The convergence is the following:
\begin{theorem}
\label{theo:conv_profile}
Suppose that the previous assumptions hold. Then, for almost every realization of the connectivity sequence $ \Xi$ given in Definition~\ref{def:regularity_network}, the spatial field $( \theta_{ n})$ given in \eqref{eq:spatial_profile} converges weakly in $ \mathcal{ C}([0,T], L^{ k})$ to $ \psi(\cdot, t)_{ t\in[0,T]}$, unique solution in $ \mathcal{ C}([0,T], L^{ k})$ to \eqref{eq:heat_equation_weak} with initial condition
\begin{equation}
\label{hyp:initial_condition_psi}
\psi_{ 0}(x) := \int_{ \mathbb{ R}^{ d}} \theta \nu_{ 0}^{ x}({\rm d} \theta).
\end{equation}
\end{theorem}
The present result can be seen as a generalization of \cite{MR3238494,MR3187677,10113716M1075831,10113717M1134007}, where the case of Kuramoto-type interaction (namely $ \Gamma$ of the form $ \Gamma( \theta- \tilde{ \theta})$ with $ \Gamma(\cdot)$ and $c(\cdot)$ Lipschitz and bounded) in absence of noise ($ \sigma=0$) is considered. Theorem~\ref{theo:conv_profile} is proven in Section~\ref{sec:proof_conv_spatial_profile}.

\subsection{Identification}
\label{sec:identification_intro}
A direct consequence of Theorems~\ref{theo:conv_empirical_measure} and~\ref{theo:conv_profile} is  the identification between the spatial profile $ \psi(\cdot, t)$, weak solution to \eqref{eq:heat_equation_weak} in terms of the expected value of the solution $ \nu_{ t}$ of \eqref{eq:McKeanVlasov} (see \eqref{eq:identification_psi_nu} and \eqref{eq:identification_psi_nu_noncompact} below). This identification is straightforward in the case $I=[0, 1]$, $\ell({\rm d}x)= {\rm d}x$, as it based on both convergence of processes $(\nu_{ n, t})_{ n\geq1}$ and $ (\theta_{ n}(\cdot, t))_{ n\geq1}$:
\begin{theorem}[Identification in the compact case]
\label{theo:identification}
Suppose $I=[0,1]$ and that the hypotheses of Theorems~\ref{theo:conv_empirical_measure} and~\ref{theo:conv_profile} hold. Let $ \left(\nu_{ t}\right)_{ t\in[0,T]}$ be the unique solution to \eqref{eq:McKeanVlasov} with initial condition $ \nu_{ 0}$ and $ \psi(\cdot, t)_{ t\in[0, T]}$ be the unique solution to \eqref{eq:heat_equation_weak} with initial condition $ \psi_{ 0}(x):= \int \theta \nu_{ 0}^{ x}({\rm d}\theta)$. Then, for all $t\in [0, T]$, all test regular functions $J$ on $ [0,1]$
\begin{equation}
\label{eq:identification_psi_nu}
\int_{ [0,1]} \left\langle \psi(x, t)\, ,\, J(x)\right\rangle {\rm d}x= \int_{ [0,1]} \left\langle \int_{ \mathbb{ R}^{ d}} \theta \nu_{ t}^{ x}({\rm d}\theta)\, ,\, J(x)\right\rangle {\rm d}x.
\end{equation}
\end{theorem}
Theorem~\ref{theo:identification} is proven in Section~\ref{sec:proof_conv_spatial_profile}. When $I\subset \mathbb{ R}^{ p}$ is not compact, there is no natural construction of the spatial profile $(\theta_{ n}(\cdot, t))$ as in \eqref{eq:spatial_profile}. However, by a simple truncation argument, it is still possible to get the following identification result (which may have an interest of its own, independently of the context of random graphs):
\begin{theorem}[Identification when $I= \mathbb{ R}^{ p}$]
\label{theo:identification_non_compact}
Suppose that $I= \mathbb{ R}^{ p}$ is endowed with a probability measure $\ell({\rm d}x)=\ell(x) {\rm d}x$ that is absolutely continuous w.r.t. the Lebesgue measure on $ \mathbb{ R}^{ p}$. Suppose that $\ell$ is $ \mathcal{ C}^{ 1}$ on $ \mathbb{ R}^{ d}$ and fix a kernel $W(x, y)$ that is $ \mathcal{ C}^{ 1}$ on $ \mathbb{ R}^{ p}\times \mathbb{ R}^{ p}$. Suppose that $c, \Gamma, \nu_{ 0}$ and $W$ satisfy Model~\ref{mod:polynomial} and the assumptions of Section~\ref{sec:general_assumptions_nonlin}. Then, \eqref{eq:heat_equation_weak} has a unique weak solution $ \psi(\cdot, t)_{ t\in[0, T]}$ with initial condition $ \psi_{ 0}(x):= \int \theta \nu_{ 0}^{ x}({\rm d}\theta)$. If $ \left(\nu_{ t}\right)_{ t\in[0,T]}$ is the unique solution to \eqref{eq:McKeanVlasov} with initial condition $ \nu_{ 0}$, then, for all $t\in [0, T]$, all test regular functions $J$ with compact support on $ \mathbb{ R}^{ p}$,
\begin{equation}
\label{eq:identification_psi_nu_noncompact}
\int_{ \mathbb{ R}^{ p}} \left\langle \psi(x, t)\, ,\, J(x)\right\rangle \ell({\rm d}x)= \int_{ \mathbb{ R}^{ p}} \left\langle \int_{ \mathbb{ R}^{ d}} \theta \nu_{ t}^{ x}({\rm d}\theta)\, ,\, J(x)\right\rangle \ell({\rm d}x).
\end{equation}
\end{theorem}
Theorem~\ref{theo:identification_non_compact} is proven in Section~\ref{sec:identification_noncompact_case}. In the case of FitzHugh-Nagumo oscillators, \eqref{eq:identification_psi_nu_noncompact} (i.e. the identification of the expected value of \eqref{eq:FP} as a solution of \eqref{eq:heat_equation}) can be seen as a weak formulation of a recent work \cite{2018arXiv180401263C} where a similar issue is addressed, using PDE techniques. Although we consider here a more general class of model, the present identification is weaker, as it is only valid in the sense of distributions (in particular, the spatial regularity of $(\psi(\cdot, t))$ is not addressed here). Another significant difference with \cite{2018arXiv180401263C} is that we crucially need here to have a probability measure $\ell({\rm d}x)$ on the spatial variable $x$, whereas \cite{2018arXiv180401263C} adresses directly the case where $\ell$ is Lebesgue on $ \mathbb{ R}^{ p}$.
\subsection{A comment on graph convergence}
\label{sec:comment_graph_convergence}
The aim of this paragraph is to question the notion of graph convergence given by Definition~\ref{def:convergence_graph}. The point we want to raise here is that this notion of convergence does not really concern so much the unlabeled and (possibly) undirected graph $ \mathcal{ G}^{(n)}$ constructed in Definition~\ref{def:regularity_network}, but is rather a notion of convergence of a directed and weighted graph $\bar{\mathcal{ G}}^{ (n)}$ that is coupled to $ \mathcal{ G}^{(n)}$ (with weights that depend on the dilution sequence $ \kappa^{ (n)}$). The reason is that, even though the original graph $ \mathcal{ G}^{ (n)}$ might be symmetric, the renormalization $ \kappa_{ i}^{ (n)}$ for each node $i$ in \eqref{eq:odegene} induces an asymmetry in the interaction between $i$ and any of its neighbor $j$. To be more specific:
\begin{definition}
\label{def:bar_Gn}
Let $ \bar{ \mathcal{ G}}^{ (n)}$ be the directed and weighted graph (with vertex set $[n]$) constructed from $ \mathcal{ G}^{ (n)}$ in the following way: for any $i\neq j\in [n]$, both edges $i\to j$ and $j\to i$ are present in $\bar{\mathcal{ G}}^{ (n)}$ if and only if the undirected edge $\{i, j\}$ is present in $ \mathcal{ G}^{ (n)}$. Then, attribute the weight $ \kappa_{i}^{ (n)}$ (resp. $ \kappa_{j}^{ (n)}$) to $i\to j$ (resp. $j\to i$) in $ \bar{ \mathcal{ G}}^{ (n)}$.
\end{definition}
We suppose here again for simplicity that $I:=[0,1]$, endowed with its Lebesgue measure $\ell({\rm d}x):= {\rm d}x$ and where $x_{ i}^{ (n)}:= \frac{ i}{ n}$ for $i\in \left[n\right]$. We assume in this paragraph that Definition~\ref{def:convergence_graph} holds as well as:
\begin{equation}
\label{hyp:Int_complete_W_converge}
 \sum_{ i,j=1}^{ n} \int_{ \frac{ i-1}{ n}}^{ \frac{ i}{ n}}\int_{ \frac{ j-1}{ n}}^{ \frac{ j}{ n}} \left\vert W(x_{ i}^{ (n)}, x_{ j}^{ (n)}) -W(x, y) \right\vert {\rm d}x {\rm d}y \xrightarrow[n\to\infty]{}0.
\end{equation}
\begin{proposition}
\label{prop:conv_graph_bar_Gn}
Let $\bar{\mathcal{ G}}^{ (n)}$ be given by Definition~\ref{def:bar_Gn}. Under the above hypotheses, we have the following convergence result:
\begin{equation}
\label{eq:limit_barGn_W}
d_{ \Box} \left( \bar{ \mathcal{ G}}^{ (n)}, W\right) \xrightarrow[ n\to\infty]{}0,
\end{equation}
where $d_{ \Box}(\cdot, \cdot)$ is the cut-off distance.
\end{proposition}
We use here the formalism of graph convergence developed in \cite{LOVASZ2006933,MR2455626,MR2825531,borgs2018,Borgs:2014aa} (and references therein). The precise definition of the cut-off distance together with the proof of Proposition~\ref{prop:conv_graph_bar_Gn} are given in Appendix~\ref{sec:link_graph_conv}. Note that one needs to slightly generalize the formalism of \cite{borgs2018,Borgs:2014aa} to the case of directed graphs and asymmetric kernels $W$ (but this is of minor difficulty).

An important point concerning Proposition~\ref{prop:conv_graph_bar_Gn} is that a lot of the structure of the microscopic graph $ \mathcal{ G}^{ (n)}$ is lost in the limit $n\to\infty$: the macroscopic limit \eqref{eq:FP} essentially captures a dynamics that lives on the \emph{renormalized} graph $ \mathcal{ \bar{G}}^{ (n)}$, which may be significantly different to $ \mathcal{ G}^{ (n)}$. To be more specific, one of the main contributions of the theory developed in \cite{LOVASZ2006933,MR2455626,MR2825531,borgs2018,Borgs:2014aa} is to show the existence of a large class of generic models of microscopic graphs $ \mathcal{ G}^{ (n)}$ that converge to some graphon $ \mathcal{P}(x, y)$. Here, through the renormalization by $( \kappa_{ i}^{ (n)})_{ i\in \left[n\right]}$ in \eqref{eq:odegene}, the limit $\mathcal{ P}$ of $ \mathcal{ G}^{ (n)}$  is in general different from the limit $W$ of $\bar{\mathcal{ G}}^{ (n)}$ given by \eqref{eq:limit_barGn_W}. This is typically true when $ \mathcal{ G}^{ (n)}$ has vertices with diverging degree as $n\to\infty$, see e.g. Section~\ref{sec:convergence_graphs_examples}, (Examples~\ref{ex:unbounded_graphon_ex1} and~\ref{ex:unbounded_graphon_ex2}) and Remark~\ref{rem:diff_Gn_sameW}, where we have two different $ \mathcal{ G}^{ (n)}$, converging to different $ \mathcal{ P}$, such that their renormalized graphs $ \bar{ \mathcal{ G}}^{ (n)}$ converge to the same $W$. 

In particular, even though the graph of interaction $ \mathcal{ G}^{ (n)}$ might be of power-law type, the renormalized graph $ \bar {\mathcal{ G}}^{ (n)}$ and its macroscopic counterpart $W$ that we consider in this paper are \emph{never} of power-law type: a crucial assumption that is constantly used in this work is  \eqref{hyp:bound_W_L1}, i.e. the degree of each macroscopic node remains uniformly of order $1$. This uniformity in degrees crucially depends on the choice of the dilution coefficients $(\kappa_{ i}^{ (n)})$. To illustrate this, consider the graph $ \mathcal{ G}^{ (n)}$ defined in Example~\ref{ex:unbounded_graphon_ex2} where, instead of \eqref{hyp:dilution_alpha_ni_cP}, we choose now
\begin{equation}
\label{hyp:renorm_uniform}
\kappa_{i}^{ (n)} = \frac{ 1}{ \rho_{ n}},\ i\in \left[n\right],
\end{equation}
that is, the same uniform dilution as for bounded kernels  \eqref{eq:kappan_uniform_P} (example already considered in \cite{10113716M1075831}, \S~6.2). The graph $( \mathcal{ G}^{(n)}, \kappa^{ (n)})$ remains convergent in the sense of Definition~\ref{def:convergence_graph} to
\begin{equation}
\label{hyp:W_equal_cP}
W(x, y):= \mathcal{ P}(x, y)= (1- \alpha)^{ 2}x^{ - \alpha} y^{ - \alpha}.
\end{equation}
Indeed, choosing for simplicity $ 0< \alpha< \frac{ 1}{ 6}$ and $ \rho_{ n}= n^{ - \delta}$ with $ 2 \alpha< \delta <\frac{ 1}{ 2} - \alpha$  in \eqref{eq:generic_Wn},
\begin{align*}
\delta_{ n}(\underline{ x})&=\sup_{ i\in \left[n\right]}\frac{ 1}{ n} \sum_{ k=1}^{ n} \left\vert \min \left( n^{ \delta}, (1- \alpha)^{ 2} (x_{ i}x_{ j})^{ - \alpha}\right) - (1- \alpha)^{ 2} (x_{ i}x_{ j})^{ - \alpha} \right\vert,
\end{align*}
is equally $0$ for all $n$, since $ \delta> 2 \alpha$. But now, the uniform renormalization \eqref{hyp:renorm_uniform} (well adapted to vertices with low degree, with position away from $0$) is no longer sufficient to compensate for vertices with high degree with position close to $0$ and the boundedness assumption \eqref{hyp:bound_W_L1} is no longer satisfied for \eqref{hyp:W_equal_cP}: macroscopic nodes $x\in [0,1]$ have diverging degrees as $ x\to 0$. At the level of generality considered in this work (but even for Kuramoto-type interaction), it is unclear if the convergence results (Theorems~\ref{theo:conv_general} or Theorem~\ref{theo:conv_profile}) remain true when assumption \eqref{hyp:bound_W_L1} is discarded.
\subsection{Applications and links with the existing literature}
\label{sec:applications}
\subsubsection{Applications} The kind of applications we have in mind are:
\begin{enumerate}
\item The Kuramoto model and its variants: take $d=1$ and $ \Gamma(\theta, \tilde{ \theta})=\sin \left(\tilde{ \theta}- \theta\right)$. Examples of local dynamics are usually $c(\cdot)\equiv \omega$ or $c(\theta)=1+ a \sin(\theta)$. In this context, \eqref{eq:McKeanVlasov} gives rise to a family of Kuramoto models with spatial extension already considered in the literature ($P$-nearest neighbor model \cite{PhysRevLett.106.234102}, long-range interactions \cite{PhysRevE.85.066201}). The question of characterizing the synchronized states, as for the original Kuramoto model, is still an ongoing question (see \cite{doi:10.1142/S0218127406014551,1742-5468-2014-8-R08001} and references therein). In the case of Kuramoto type ODEs (that is when $ \sigma=0$), a series of papers (see \cite{MR3238494,10113716M1075831,10113717M1134007,Chiba:2016aa,MR3187677} and references therein) have addressed similar issues to the ones addressed here. In addition to the fact that we consider more general hypotheses (e.g. possibly unbounded and asymmetric $ \Gamma$ and $c$ being non-Lipschitz), the main difficulty of the present analysis is that noise is present in \eqref{eq:odegene}. In particular, the fixed-point argument \cite{Neunzert1984} used in \cite{10113717M1134007,Chiba:2016aa} for the convergence of the empirical measure in the deterministic case does not seem to generalize easily to the case $ \sigma\neq 0$. 
\item FitzHugh-Nagumo oscillators: this corresponds to $d=2$, $ \theta= (V, w)$, $ c(V, w):= (V-\frac{V^3}{3} - w,  \frac{1}{ \tau}(V+a-bw))$ (for appropriate parameters $a, b, \tau$) and $ \Gamma(\theta, \tilde{ \theta}):= ( V- \tilde{ V}, 0)$. Here, $V$ stands for the potential of one neuron and $w$ its recovery variable. We refer to \cite{MR3392551} and references therein for more details on this particular model and its applications to neuroscience. Once again, \eqref{eq:McKeanVlasov} gives a spatially-extended version of a Fokker-Planck PDE, already analyzed in the context of neurons interacting through a deterministic spatial kernel (\cite{MR2998591,LucSta2014,LucSta2016}). The present work gives a new interpretation of such spatially-extended PDEs in terms of the mean-field limit of diffusions on random graphs.
\end{enumerate}
\subsubsection{Long-time behavior}
Theorems~\ref{theo:conv_general} and~\ref{theo:conv_empirical_measure} are the natural extensions of \cite{MR3568168} that concerns the case of homogeneous Erd\H os-R\'enyi graphs. One should also mention at this point the recent work \cite{2018arXiv180710921C} which addresses quenched propagation of chaos and large deviations results on homogeneous graphs. The result that is closest to this work is the recent \cite{2018arXiv180706898O} where a similar result of convergence is addressed in the case of bounded and Lipschitz coefficients $ c$, $ \Gamma$. The analysis in \cite{2018arXiv180706898O} restricts to bounded kernels $W$ and random positions (that is a particular case of Section~\ref{sec:dense_bounded_graphons} below). Note also that the convergence of \cite{2018arXiv180706898O} is annealed in both disorder (connections and positions), whereas the present analysis is quenched. Contrary to \cite{2018arXiv180706898O}, we do not address large deviation estimates here. 

This work comes with all the comments and restrictions raised in \cite{MR3568168}: the convergence results are only valid on bounded time intervals $[0, T]$, where $T$ is independent of $n$ (although that, with a little more work, it would certainly be possible to extend the result up to times $T$ which grows logarithmically in $n$, as in \cite{MR3568168}, Corollary~1.2). In any case, the behavior of the empirical measure \eqref{eq:emp_measure} on larger time scales remains unclear. The difficulty of the long-time analysis of \eqref{eq:odegene} (already present for complete mean-field interactions, see \cite{Bertini:2013aa,MR3689966} for results in this direction for Kuramoto-type models) is here even more present for general graphs, since the interaction in \eqref{eq:odegene} cannot be written as a closed expression of the empirical measure \eqref{eq:emp_measure}. In particular, a trajectorial Central Limit Theorem associated to Theorem~\ref{theo:conv_empirical_measure} remains open (in this direction, see \cite{BHAMIDI2018} for an annealed fluctuation theorem in the Erd\H os-R\'enyi case).

One point raised in \cite{MR3568168} concerned the necessity of the independence of the initial condition of \eqref{eq:odegene} $(\theta_{ 1, 0}, \ldots, \theta_{ n, 0})$ with respect to the graph $ \mathcal{ G}^{ (n)}$ (in \cite{MR3568168}, the $(\theta_{ 1, 0}, \ldots, \theta_{ n, 0})$ are identically distributed; see also \cite{2018arXiv180710921C} where a convergence result is proven in the Erd\H os-R\'enyi case assuming only the convergence of the empirical measure of the initial condition). Here, we note that the present framework allows for a slight connection between the initial condition and the graph: the law $ \nu_{ 0}^{ x_{ k}}$ of $ \theta_{ k, 0}$ depends on its position $x_{ k}$ which encodes for the way the graph $ \mathcal{ G}^{ (n)}$ is built. A simple illustration is when the graph is made of two complete disconnected components (one concerning the particles with positions in $[0, \frac{ 1}{ 2}]$, with initial law $ \nu_{ 0}^{ x}=\mu_{ 1}$ and one concerning particles with positions in $[ \frac{ 1}{ 2}, 1]$, with initial law $ \nu_{ 0}^{ x}=\mu_{ 2}$). In this case, the behavior of the system is governed by \eqref{eq:FP}, with macroscopic kernel $W= \mathbf{ 1}_{ [0, \frac{ 1}{ 2}]^{ 2}} + \mathbf{ 1}_{ [ \frac{ 1}{ 2}, 1]^{ 2}}$.
\subsubsection{Neural field equation and traveling waves}
In the particular case where $d=1$, $ c(\theta) = - \alpha \theta$ (for some $ \alpha>0$) and $ \Gamma(\theta, \tilde{ \theta})= f( \tilde{ \theta})$ (typically $f$ is a sigmoid function),  \eqref{eq:heat_equation} becomes:
\begin{equation}
\label{eq:neural_field_equation}
\partial_{ t}\psi(x, t)= - \alpha \psi(x, t) + \int_{I} f(\psi(y, t)) W(x, y)\ell({\rm d}y).
\end{equation}
Equation \eqref{eq:neural_field_equation} is nothing else than the neural field equation, introduced by Wilson and Cowan \cite{Wilson:1972aa} and Amari \cite{Amari:1977:DPF:2731211.2731248} in order to describe the macroscopic activity of a population of neurons with spatial extension. Eq. \eqref{eq:neural_field_equation} has been the subject of an extensive literature (see \cite{MR3136844,MR2871421} and references therein; see in particular the recent \cite{CHEVALLIER2018} showing that \eqref{eq:neural_field_equation} is a proper limit for spatially-extended Hawkes processes).  An important issue here is the existence and stability of traveling waves \cite{doi:10.1137/18M1165797,101137130918721}. The point we want to raise here is the possibility of studying such traveling waves through the analysis of the corresponding McKean-Vlasov PDE \eqref{eq:McKeanVlasov} (whose dynamics is, to our knowledge, much less studied than \eqref{eq:neural_field_equation}, see \cite{maclaurin2017mean}), through the identification \eqref{eq:identification_psi_nu}. An interesting and open question concerns the possiblility of extending this identification beyond finite time scales (as for the Kuramoto model). In this context, it is reasonable to expect that the effect of thermal noise will persist on larger time intervals, resulting in stochastic neural field equations \cite{MR3367676,10113715M102856X,10113713095094X,10113715M1033927}. 
\section{Examples}
\label{sec:convergence_graphs_examples}
The point of this section is twofold: to describe generic models $( \mathcal{ X}, W)$ that are regular in the sense of Definition~\ref{def:graphon_regular} (Section~\ref{sec:regularity_kernel}) and to give examples of graphs $( \mathcal{ G}^{ (n)}, \kappa^{ (n)})$ that are convergent in the sense of Definition~\ref{def:convergence_graph} (Sections~\ref{sec:generic_class_convergent_graphs},~\ref{sec:dense_bounded_graphons} and~\ref{sec:convergent_graph_singular}). These examples are directly inspired by the formalism of \emph{$W$-random graphs}, introduced in  \cite{LOVASZ2006933,MR2455626,MR2925382,Borgs:2014aa,borgs2018} and used in \cite{MR3238494,10113716M1075831,10113717M1134007,Chiba:2016aa,MR3187677} (and references therein) in the context of Section~\ref{sec:convergence_profile_intro}. In this framework, a usual setting is to consider the compact $I=[0,1]$. This set-up is particularly well adapted to the choice of deterministic regular positions (see Assumptions~\ref{ass:deterministic_positions} below). Nonetheless, as already mentioned, it also makes sense to consider a general state space $I$ endowed with a general probability measure $\ell$, where each $x_{ i}$ actually encodes for a real position \cite{Cabana:2015aa,CHEVALLIER2018}.
\subsection{Two classes of regular models}
\label{sec:regularity_kernel}
We describe in this paragraph two generic classes of positions $ \mathcal{ X}$ and macroscopic kernels $W$ and provide simple conditions in both models ensuring that $W$ is regular in the sense of Definition~\ref{def:graphon_regular}.
\subsubsection{Deterministic positions}
A first set of hypotheses corresponds to deterministic positions \cite{MR3238494,Cabana:2015aa,CHEVALLIER2018,Muller:2015aa}:
\begin{assumption}[Deterministic positions]
\label{ass:deterministic_positions} 
We suppose that $I:=[0, 1]$, endowed with its Lebesgue measure $\ell({\rm d}x):= {\rm d}x$.
For all $n\geq1$, the sequence $ \underline{ x}=(x_{ 1}^{ (n)}, \ldots, x_{ n}^{ (n)})$ is deterministic, regularly positioned on $I$:
\begin{equation}
\label{hyp:x_regular}
x_{ i}^{ (n)}:= \frac{ i}{ n},\ n\geq1, i\in \left[n\right].
\end{equation}
We set $x_{ 0}^{ (n)}:=0$ and $x_{ n+1}^{ (n)}:= 1$ for notational convenience.
\end{assumption}
A sufficient condition for the regularity of $W$ is to require \eqref{hyp:Int_W_Lip} and
\begin{equation}
\label{hyp:delta_n_W}
s_{ n}(W)( \underline{ x}):=\sup_{ i\in \left[n\right]} \left(\sum_{ k=1}^{ n}\int_{ x_{ k-1}}^{x_{ k}}  \left\vert W(x_{ i}, x_{ k}) -  W(x_{ i}, y) \right\vert  {\rm d} y\right) \xrightarrow[ n\to\infty]{} 0.
\end{equation}
\begin{proposition}
\label{prop:control_det_Delta3}
Under Assumption~\ref{ass:deterministic_positions}, under the hypotheses of Section~\ref{sec:general_assumptions_nonlin} and \eqref{hyp:Int_W_Lip}, \eqref{hyp:delta_n_W},
the kernel $W$ is regular along the sequence $\mathcal{ X}$ in the sense of Definition~\ref{def:graphon_regular}.
\end{proposition}
Proposition~\ref{prop:control_det_Delta3} is proven in Section~\ref{sec:proof_regularity_graphon_det}.
\subsubsection{Random positions}
Another general framework concerns the case of random positions \cite{2018arXiv180706898O,MR3187677}:
\begin{assumption}[Random positions]
\label{ass:random_positions}
Let $I$ be a closed subset of $ \mathbb{ R}^{ p}$ that is the support of a probability measure $\ell({\rm d}x)$. The sequence \begin{equation}
\mathcal{ X}=(x_{ 1}, x_{ 2}, \ldots)
\end{equation} is the realization of i.i.d. random variables with law $\ell({\rm d}x)$ on $I$.
\end{assumption}
For any $p\geq1$ denote by $ \left\Vert W \right\Vert_{ L^{ p}}= \left\Vert W \right\Vert_{ L^{ p}(I^{ 2}, \ell\otimes\ell)}$ the usual $L^{ p}$-norm of $W$ on $I^{ 2}$:
\begin{equation}
\left\Vert W \right\Vert_{ L^{ p}}:= \left(\int_{ I^{ 2}} W(x, y)^{ p} \ell({\rm d}x)\ell({\rm d}y)\right)^{ \frac{ 1}{ p}}.
\end{equation}
\begin{proposition}
\label{prop:control_An6}
If Assumption~\ref{ass:random_positions} is true, under the hypotheses of Section~\ref{sec:general_assumptions_nonlin} and if there exists $ \chi>9$ such that
\begin{equation}
\label{hyp:norm_W_chi}
\left\Vert W \right\Vert_{ L^{ \chi}}< \infty,
\end{equation}
then the kernel $W$ is regular along the sequence $\mathcal{ X}$ for $\ell$-almost every realization of the sequence $ \mathcal{ X}$, in the sense of Definition~\ref{def:graphon_regular}. 
\end{proposition}
Proposition~\ref{prop:control_An6} is proven in Section~\ref{sec:proof_graphon_regularity_rand}.

\subsection{A class of convergent graphs}
\label{sec:generic_class_convergent_graphs}
We give now examples of microscopic graphs $ \mathcal{ G}^{ (n)}$ constructed as in Definition~\ref{def:regularity_network}, that satisfy the results of Section~\ref{sec:main_results}, in both situations of deterministic (Assumption~\ref{ass:deterministic_positions}) and random positions (Assumption~\ref{ass:random_positions}). The present examples fall into the framework of $W$-random graphs (with possibly unbounded graphons), that is, when the probability field $W_{ n}(x, y)$ is directly constructed upon a predetermined determinisitic kernel $ \mathcal{ P}(x, y)$. A general framework may be given by (see \cite{Borgs:2014aa,10113716M1075831} for similar definitions)
\begin{definition}[Generic random graph with graphon $ \mathcal{ P}$]
For fixed $(I, \ell)$, $n\geq1$ and a given positive measurable kernel $(x, y) \mapsto \mathcal{ P}(x, y)$ on $I^{ 2}$, we define
\begin{equation}
\label{eq:generic_Wn}
W_{ n}(x, y) := \rho_{ n} \min \left(\frac{ 1}{ \rho_{ n}}, \mathcal{ P}(x, y)\right),
\end{equation}
where $\rho_{ n}\in[0, 1]$.
\end{definition}
One important aim of \cite{LOVASZ2006933,MR2455626,MR2925382,Borgs:2014aa,borgs2018} (and references therein) is precisely to prove that, under various hypotheses, $ \mathcal{ G}^{ (n)}$ converges to $ \mathcal{ P}$. In this context, one generally distinguishes between \emph{bounded graphons} $ \mathcal{ P}$ \cite{LOVASZ2006933,MR2455626,MR2925382} and \emph{unbounded graphons} (a typical hypothesis being that $ \mathcal{ P}\in L^{ p}(I^{ 2})$ for some $p\geq1$ see \cite{Borgs:2014aa,borgs2018} and references therein). Note that when the graphon $ \mathcal{ P}$ is bounded (and up to the change $ \rho_{ n} \leftrightarrow \frac{ \rho_{ n}}{ \left\Vert \mathcal{ P} \right\Vert_{ \infty}} $, one can always suppose that $ \left\Vert \mathcal{ P} \right\Vert_{ \infty}=1$), \eqref{eq:generic_Wn} boils down to
\begin{equation}
\label{eq:generic_Wn_P_bounded}
W_{ n}(x, y) = \rho_{ n} \mathcal{ P}(x, y),\ x,y\in I.
\end{equation}
When $ \rho_{ n}=1$, we are dealing with dense graphs, whereas in the case $ \rho_{ n} \xrightarrow[ n\to\infty]{}0$, we consider diluted graphs. A very simple particular case of \eqref{eq:generic_Wn_P_bounded} corresponds to $ \mathcal{ P}(x, y)\equiv 1$ which boils to a (possibly diluted) homogeneous Erd\H os-R\'enyi random graph, already studied in \cite{MR3568168,2018arXiv180710921C}. Thus, one has to think of \eqref{eq:generic_Wn_P_bounded} as an inhomogeneous version of the Erd\H os-R\'enyi case. When $ \mathcal{ P}$ is not bounded, one usually assumes in \eqref{eq:generic_Wn} that $ \rho_{ n}\to 0$ and $ n \rho_{ n}\to \infty$, as $n\to\infty$.

Here the distinction is not really on the boundedness of $ \mathcal{ P}$ in \eqref{eq:generic_Wn}, but rather between $\sup_{ x} \int \mathcal{ P}(x, y) \ell({\rm d}y)<+\infty$ and $\sup_{ x} \int \mathcal{ P}(x, y) \ell({\rm d}y)=+\infty$. In the first case, there is uniform control on the asymptotic degree of each node in the graph $ \mathcal{ G}^{ (n)}$ whereas in the second, $ \mathcal{ G}^{ (n)}$ has nodes with diverging degree as $n\to\infty$. We treat these two cases in Section~\ref{sec:dense_bounded_graphons} and~\ref{sec:convergent_graph_singular} below.
\subsection{Convergent graphs: the case of graphons with uniformly bounded degrees}
\label{sec:dense_bounded_graphons}
We assume in this paragraph that $W_{ n}$ is given by \eqref{eq:generic_Wn} for
\begin{equation}
\label{eq:P_int_unif_bounded}
\sup_{ x\in I} \int \mathcal{ P}(x, y) \ell({\rm d}y)<+\infty.
\end{equation}
Here, we adopt a renormalization that is uniform on the nodes $i\in \left[n\right]$: set
\begin{equation}
\label{eq:kappan_uniform_P}
\kappa_{i}^{ (n)}= \kappa_{ n}:=\frac{ 1}{ \rho_{ n}},\ i=1, \ldots, n \text{ and }w_{ n}= \rho_{ n},
\end{equation}
satisfying
\begin{equation}
\label{eq:cond_dilution_ER_P}
\kappa_{ n}=\frac{ 1}{ \rho_{ n}} = o \left( \frac{ n}{ \log(n)}\right), \text{ as } n\to\infty.
\end{equation}
In this case, the appropriate limit for $(\mathcal{ G}^{ (n)}, \kappa_{ n})$ is simply given by $W:= \mathcal{ P}$ itself. The verification of the hypotheses of Section~\ref{sec:general_assumptions_nonlin} require that 
\begin{equation}
\label{hyp:int_P_bounded}
\sup_{ z\in I} \int \mathcal{ P}(z, y)^{ 2}\ell({\rm d}y)<+\infty \text{ and } \inf_{ z\in I} \int \mathcal{ P}(z, y) \ell({\rm d}y)>0.
\end{equation} We address now the question of the convergence of $(\mathcal{ G}^{ (n)}, \kappa_{ n})$ to $ \mathcal{ P}$ (Definition~\ref{def:convergence_graph}) as well as the regularity of the model (Definition~\ref{def:graphon_regular}) in both cases of deterministic (Assumption~\ref{ass:deterministic_positions}) and random positions (Assumption~\ref{ass:random_positions}).
\subsubsection{Convergence of the graph $( \mathcal{ G}^{ (n)}, \kappa^{ (n)})$:}
Note that \eqref{hyp:alphan_VS_an}, \eqref{hyp:def_w_n_infty} and \eqref{hyp:compare_wn_alphan} are trivially verified here, as well as \eqref{hyp:alpha_n_infty}, by \eqref{eq:cond_dilution_ER_P}. This dilution condition was already noticed in \cite{MR3568168} in the Erd\H os-R\'enyi case $ \mathcal{ P}\equiv 1$: the microscopic graphs that are relevant for the present work have an averaged degree larger than $\log(n)$. The only point that we need to check is \eqref{hyp:Delta1_to_0}. Note first that when $ \mathcal{ P}$ is bounded, \eqref{hyp:Delta1_to_0} is immediately verified, since we have $ \delta_{ n}( \underline{ x})\equiv0$ for all $n$. This can be slightly generalized into the following sufficient condition:
\begin{proposition}
\label{prop:crit_graph_P_bounded}
Suppose that the sequence of positions $\mathcal{ X}$ and the kernel $ \mathcal{ P}$ are such that for some $ \epsilon\in\left[0, \frac{ 1}{ 2}\right)$ and $C>0$, for all $n\geq1$ sufficiently large,
\begin{equation}
\label{hyp:control_cP}
\sup_{ i,j\in \left[n\right]} \mathcal{ P}(x_{ i}, x_{ j})\leq C n^{ \epsilon},
\end{equation}
Choose $ \delta$ such that $ \epsilon< \delta < \frac{ 1}{ 2}$ and define
\begin{equation}
\label{hyp:rho_n_P}
\rho_{ n}:= n^{ - \delta}.
\end{equation}
Then, $( \mathcal{ G}^{ (n)}, \kappa^{ (n)})$ given by \eqref{eq:generic_Wn} and \eqref{eq:kappan_uniform_P} (with $ \rho_{ n}$ given by \eqref{hyp:rho_n_P}) converges to $ \mathcal{ P}$ in the sense of Definition~\ref{def:convergence_graph}.
\end{proposition}
\begin{proof}[Proof of Proposition~\ref{prop:crit_graph_P_bounded}]
It suffices to note that by \eqref{hyp:control_cP} and \eqref{hyp:rho_n_P}, we have, for $n$ sufficiently large, $ \min \left(\frac{ 1}{ \rho_{ n}}, \mathcal{ P}(x_{ i}, x_{ k})\right) = \mathcal{ P}(x_{ i}, x_{ k})$ for all $i, k\in \left[n\right]$. Hence, $ \delta_{ n}(\underline{ x})=0$ for such $n$. 
\end{proof}
\subsubsection{Regular kernels for deterministic positions (Assumption~\ref{ass:deterministic_positions}):} if $ \mathcal{ P}$ is bounded, having in mind Proposition~\ref{prop:control_det_Delta3}, a simple sufficient condition is $W= \mathcal{ P}$ $ \beta$-H\"older for some $ \beta\in (0, 1]$: for such $ \mathcal{ P}$,  \eqref{hyp:Int_W_Lip} holds for $ \iota_{ 2}= \beta$ and \eqref{hyp:delta_n_W} is straightforward. The supplementary regularity condition \eqref{hyp:Int_Wn_W_2} required for Theorem~\ref{theo:conv_profile} is also valid. Interesting examples include $ \mathcal{ P}(x, y)= 1- \max (x, y)$ or $ \mathcal{ P}(x,y)= 1- xy$ which are encountered in the context of dense inhomogeneous graphs (see \cite{borgs2018,MR2825531} for many interesting examples). However, note that the hypotheses \eqref{hyp:Int_W_Lip}, \eqref{hyp:delta_n_W} and \eqref{hyp:Int_Wn_W_2} are sufficiently general to capture some interesting cases where $ \mathcal{ P}$ is not continuous: the $P$-nearest neighbor model \cite{LucSta2014} corresponds to $ \rho_{ n}=1$ and
$\mathcal{ P}(x, y)= \mathbf{ 1}_{ \left\vert x-y \right\vert\leq R},\ x,y\in I$, for some $R\in(0, 1]$. It is immediate to see that \eqref{hyp:Int_W_Lip} is true for $ \iota_{ 2}= \frac{ 1}{ 2}$, that $s_{ n}(W) = O \left(n^{ -1}\right)$ in \eqref{hyp:delta_n_W} and that \eqref{hyp:Int_Wn_W_2} holds. Another interesting case of unbounded kernel $ \mathcal{ P}$ (which still satisfies \eqref{eq:P_int_unif_bounded}) is $ \mathcal{ P}(x, y):= \frac{ 1}{ \left\vert x-y \right\vert^{ \alpha}}$ on $I=[0,1]$ (already considered in \cite{LucSta2014}). This enters into the present framework for $0< \alpha< \frac{ 1}{ 2}$. 
\medskip

In order to compare with Section~\ref{sec:convergent_graph_singular}, we end this paragraph with the following example:
\begin{example}[\cite{10113716M1075831}]
\label{ex:unbounded_graphon_ex1}
Let $( \mathcal{ G}^{ (n)}, \kappa^{ (n)})$ be given by \eqref{eq:generic_Wn} where
\begin{equation}
\label{eq:cP_y_alpha}
\mathcal{ P}(x, y) := (1- \alpha)y^{ - \alpha},\ \alpha\in \left[0, \frac{ 1}{ 2}\right),\ x, y \in I=[0,1],
\end{equation}
and $ \kappa_{i}^{ (n)}$ is given by \eqref{eq:kappan_uniform_P} and $ \rho_{ n}= n^{ - \delta}$ for some $ \alpha< \delta < \frac{ 1}{ 2}$. Then, $( \mathcal{ G}^{ (n)}, \kappa_{ n})$ is regular and convergent to $ \mathcal{ P}$.
\end{example}
Indeed a rough bound gives that $\sup_{ i, j\in \left[n\right]} \mathcal{ P}(x_{ i}, x_{ j}) \leq (1- \alpha) n^{\alpha}$, so that \eqref{hyp:control_cP} holds for $ \epsilon= \alpha$. Moreover, 
\begin{align*}
s_{ n}(W)(\underline{ x})
%&= (1- \alpha) \sum_{ k=1}^{ n} \int_{ x_{ k-1}}^{ x_{ k}}  (y^{ - \alpha} - x_{ k}^{ - \alpha}) {\rm d} y,\\
&= (1- \alpha) \left(\int_{0}^{1}y^{ - \alpha}{\rm d}y  - \frac{ 1}{ n^{ 1- \alpha}}\sum_{ k=1}^{ n} k^{ - \alpha}\right)= O \left(\frac{ 1}{ n^{ 1-\alpha}} \right).
\end{align*}
which proves \eqref{hyp:delta_n_W} since $ \alpha< 1/2$. In a same way, \eqref{hyp:Int_Wn_W_2} is true. Inequality \eqref{hyp:Int_W_Lip} (for $ \iota_{ 2}=1$) is trivial, so that Definition~\ref{def:graphon_regular} is verified.
\subsubsection{Regular kernels for random positions (Assumption~\ref{ass:random_positions}):} in the case of random positions, in addition to \eqref{hyp:int_P_bounded}, we need to verify Proposition~\ref{prop:control_An6} and Proposition~\ref{prop:crit_graph_P_bounded}. It turns out that condition \eqref{hyp:norm_W_chi} is sufficient for both: indeed, fix $ \mathcal{ P}\in L^{ \chi}(I^{ 2})$ with $ \chi>9$ and let $ \epsilon\in \left(\frac{ 3}{ \chi}, \frac{ 1}{ 2}\right)$. For any $i, j\in [n]$, 
\begin{align*}
\mathbb{ P} \left(\left\vert \mathcal{ P}(x_{ i}, x_{ j})\right\vert > n^{ \epsilon}\right)&\leq \frac{ \mathbb{ E} \left[ \left\vert \mathcal{ P}(x_{ i}, x_{ j}) \right\vert^{ \chi}\right]}{ n^{ \epsilon \chi}}= \frac{ \left\Vert \mathcal{ P} \right\Vert_{ \chi}^{ \chi}}{ n^{ \epsilon \chi}}
\end{align*}
A rough union bound on $i, j\in [n]$ gives
\begin{align*}
\mathbb{ P} \left(\sup_{ i,j\in \left[n\right]}\left\vert \mathcal{ P}(x_{ i}, x_{ j})\right\vert > n^{ \epsilon}\right)&\leq \frac{ \left\Vert \mathcal{ P} \right\Vert_{ \chi}^{ \chi}}{ n^{ \epsilon \chi-2}}.
\end{align*}
Since $ \epsilon \chi -2>1$, by Borel-Cantelli Lemma, we have almost surely that $\sup_{ i,j\in \left[n\right]}\left\vert \mathcal{ P}(x_{ i}, x_{ j})\right\vert \leq n^{ \epsilon}$ for $n$ sufficiently large. In particular, any bounded measurable weight $ \mathcal{ P}$ in $L^{ \infty}([0, 1]^{ 2})$ satisfy the hypotheses. Among interesting examples which have not been addressed so far, one can highlight the case of kernels with values in $\{0, 1\}$ (with $ \rho_{ n}=1$). This case corresponds to deterministic graphs (see \cite{MR3238494}, \S~4).

\subsection{Convergent graphs: the case of graphons with diverging degrees}
\label{sec:convergent_graph_singular}
We consider here $W_{ n}$ given by \eqref{eq:generic_Wn} for $ \mathcal{ P}$ satisfying 
\begin{equation}
\label{hyp:Int_P_unbounded}
\forall x\in I, \int \mathcal{ P}(x, y)^{ 2} \ell({\rm d}y)<+\infty \text{ and } \mathcal{ P}_{ \ast}:=\inf_{ z\in I} \int \mathcal{ P}(z, y) \ell({\rm d}y)>0.
\end{equation}
The point of the paragraph is to discuss the consequences of having possibly
\begin{equation}
\sup_{ x\in I}\int \mathcal{ P}(x, y) \ell({\rm d}y)=+\infty.
\end{equation}
Here, the uniform renormalization \eqref{eq:kappan_uniform_P} is no longer adapted: we consider instead
\begin{equation}
\label{hyp:dilution_alpha_ni_cP}
\kappa_{i}^{ (n)} = \frac{ n}{\rho_{ n} \sum_{ j=1}^{ n} \min \left(\frac{ 1}{ \rho_{ n}}, \mathcal{ P}(x_{ i}, x_{ j})\right)},\ i\in \left[n\right].
\end{equation}
This corresponds to renormalizing the interaction in \eqref{eq:odegene} by the averaged degree $\sum_{ j} W_{ n}(x_{ i}, x_{ j})$ of each vertex. Here, the correct choice for the macroscopic interaction kernel $W$ is (\cite{10113716M1075831})
\begin{equation}
\label{eq:def_W_VS_cP}
W(x, y):= \frac{ \mathcal{ P}(x, y)}{\int \mathcal{ P}(x, z) \ell({\rm d} z)},\ x, y\in I.
\end{equation}
\begin{remark}
By construction, $\sup_{ x\in I} \int W(x, y) \ell({\rm d}y)=1$. Verifying the technical second moment \eqref{hyp:bound_W_L1} requires to have $ \int \frac{ \mathcal{ P}(x, y)^{ 2}}{ \left(\int \mathcal{ P}(x, z) \ell({\rm d}z)\right)^{ 2}} \ell({\rm d}x)<\infty$, which is in particular immediate for Example~\ref{ex:unbounded_graphon_ex2} below.
Of course, it is possible to apply the renormalization \eqref{hyp:dilution_alpha_ni_cP} to the previous case of Section~\ref{sec:dense_bounded_graphons}. However, when $ \mathcal{ P}$ is bounded, the two renormalizations \eqref{eq:kappan_uniform_P} and \eqref{hyp:dilution_alpha_ni_cP} lead to slightly different macroscopic models: in \eqref{eq:def_W_VS_cP}, $W$ is renormalized by the  factor $\int \mathcal{ P}(x, z) \ell({\rm d} z)$, which is not natural in the bounded case. 
\end{remark}
\subsubsection{On the convergence of the graph $( \mathcal{ G}^{ (n)}, \kappa^{ (n)})$:} the following result is the counterpart of Proposition~\ref{prop:crit_graph_P_bounded}:
\begin{proposition}
\label{prop:conv_delta_n_cP}
Suppose that the sequence of positions $\mathcal{ X}$ and the kernel $ \mathcal{ P}$ are such that for some $ \epsilon\in\left[0, \frac{ 1}{ 2}\right)$ and $C>0$, for all $n\geq1$ sufficiently large, estimate \eqref{hyp:control_cP} holds. Choose also $ \delta$ such that $ \epsilon< \delta < \frac{ 1}{ 2}$ and define
$\rho_{ n}:= n^{ - \delta}$ as in \eqref{hyp:rho_n_P}. Suppose
\begin{equation}
\label{hyp:uniform_conv_Sn}
\sup_{ i=1, \ldots, n} \left\vert  \frac{ 1}{ n} \sum_{ j=1}^{ n} \mathcal{ P}(x_{ i}, x_{ j}) - \int \mathcal{ P}(x_{ i}, z) \ell({\rm d}z)\right\vert\to 0,\text{ as }n\to\infty.
\end{equation}
Then, $( \mathcal{ G}^{ (n)}, \kappa^{ (n)})$ given by \eqref{eq:generic_Wn} and \eqref{hyp:dilution_alpha_ni_cP} converges to $W$ given by \eqref{eq:def_W_VS_cP} in the sense of Definition~\ref{def:convergence_graph}, for the choice of $ \kappa_{ n}:= \frac{ 2n^{ \delta}}{ \mathcal{ P}_{ \ast}}$ and $w_{ n}:=1$.
\end{proposition}
\begin{proof}[Proof of Proposition~\ref{prop:conv_delta_n_cP}]
By \eqref{hyp:control_cP} and \eqref{hyp:rho_n_P}, we have, for $n$ sufficiently large, $ \min \left(\frac{ 1}{ \rho_{ n}}, \mathcal{ P}(x_{ i}, x_{ k})\right) = \mathcal{ P}(x_{ i}, x_{ k})$ for all $i, k\in \left[n\right]$. Hence, for such $n$, for $i\in \left[n\right]$, $\kappa_{i}^{ (n)} = \frac{ n^{\delta}}{ \frac{ 1}{ n}\sum_{ j=1}^{ n} \mathcal{ P}(x_{ i}, x_{ j})}$. Using \eqref{hyp:uniform_conv_Sn}, we have $\inf_{i\in \left[n\right]} \frac{ 1}{ n} \sum_{ j=1}^{ n} \mathcal{ P}(x_{ i}, x_{ j})\geq \frac{ 1}{ 2}\int \mathcal{ P}(x_{ i}, y) \ell({\rm d}y)\geq \frac{ \mathcal{ P}_{ \ast}}{ 2}$, for $n$ sufficiently large. This proves \eqref{hyp:alphan_VS_an} for $\kappa_{ n}:= \frac{ 2n^{ \delta}}{ \mathcal{ P}_{ \ast}}$.  Assumption \eqref{hyp:compare_wn_alphan} is trivial for $w_{ n}=1$ and the dilution condition \eqref{hyp:alpha_n_infty} holds since $ \delta< \frac{ 1}{ 2}$. It remains to check \eqref{hyp:Delta1_to_0}: denote by
\begin{equation*}
S_{ n}(x):=\frac{ 1}{ n}\sum_{ j=1}^{ n} \mathcal{ P}(x, x_{ j}),\ \text{ and } S(x):=  \int \mathcal{ P}(x, z) \ell({\rm d} z),\ x\in I.
\end{equation*}
With these notations, $\frac{ 1}{ n} \sum_{ k=1}^{ n}\left\vert \kappa_{i}^{ (n)}W_{ n}(x_{ i}, x_{ k}) - W(x_{ i}, x_{ k})\right\vert = \frac{ \left\vert S_{ n}(x_{ i})  - S(x_{ i})\right\vert}{S(x_{ i})}$ and the result follows immediately from \eqref{hyp:Int_P_unbounded} and \eqref{hyp:uniform_conv_Sn}.
\end{proof}
\subsubsection{Regular kernels for random positions (Assumption~\ref{ass:random_positions}):} as for Section~\ref{sec:dense_bounded_graphons}, regularity and convergence holds under sufficient integrability of the kernel $ \mathcal{ P}$:
\begin{proposition}
\label{prop:conv_singular_random}
Suppose that Assumption~\ref{ass:random_positions} holds. For any $ \mathcal{ P}\in L^{ \chi}(I^{ 2})$ with $ \chi>9$ which verifies \eqref{hyp:Int_P_unbounded}, the following is true: for almost realizations of the sequence $ \mathcal{ X}$, $( \mathcal{ G}^{ (n)}, \kappa^{ (n)})$ defined by \eqref{eq:generic_Wn} and \eqref{hyp:dilution_alpha_ni_cP} (with $ \rho_{ n}= n^{ - \delta}$, $ \frac{ 3}{ \chi}< \delta < \frac{ 1}{ 2}$) converges to $W$ given by \eqref{eq:def_W_VS_cP} in the sense of Definition~\ref{def:convergence_graph}, (for the choice of $\kappa_{ n}= \frac{ 2n^{ \delta}}{ \mathcal{ P}_{ \ast}}$ and $w_{ n}=1$) and Definition~\ref{def:graphon_regular} is verified.
\end{proposition}
\begin{proof}[Proof of Proposition~\ref{prop:conv_singular_random}]
Once again, we apply Proposition~\ref{prop:control_An6} and Proposition~\ref{prop:conv_delta_n_cP} together with a Borel-Cantelli argument. Let $ \epsilon\in \left(\frac{ 3}{ \chi}, \frac{ 1}{ 2}\right)$. The same reasoning as before shows that, since $ \epsilon \chi -2>1$, we have almost surely that $\sup_{ i,j\in \left[n\right]}\left\vert \mathcal{ P}(x_{ i}, x_{ j})\right\vert \leq n^{ \epsilon}$ for $n$ sufficiently large. Secondly, let 
\begin{equation}
\bar{ \mathcal{ P}}(x, y) := \mathcal{ P}(x, y) -\int \mathcal{ P}(x, z) \ell({\rm d}z).
\end{equation}
Compute
\begin{equation}
\label{eq:mcP_moment_6}
\mathbb{ E} \left[ \left(\frac{ 1}{ n}\sum_{ j=1}^{ n} \bar{\mathcal{ P}}(x_{ i}, x_{ j})\right)^{ 6}\right]= \frac{ 1}{ n^{ 6}} \sum_{ j_{ 1}, \ldots, j_{ 6}=1}^{ n} \mathbb{ E} \left[\prod_{ l=1}^{ 6}\bar{\mathcal{ P}}(x_{ i}, x_{ j_{ l}})\right].
\end{equation}
Among the sum above, consider the case where one index (for example $j_{ 1}$) is such that $j_{ 1}\notin \left\lbrace j_{ 2},\ldots, j_{ 6}\right\rbrace$. In this case, conditioning w.r.t. $(x_{ i}, x_{ j}, j\neq j_{ 1})$ within the previous expectation gives $\mathbb{ E} \left[\prod_{ l=1}^{ 6}\bar{\mathcal{ P}}(x_{ i}, x_{ j_{ l}})\right]=0$, by independence of the $(x_{ k})_{ k\in \left[n\right]}$ and by definition of $\bar{ \mathcal{ P}}$. Hence, the nontrivial contributions to \eqref{eq:mcP_moment_6} are necessarily of the form $\mathbb{ E} \left[\prod_{ l=1}^{ 3}\bar{\mathcal{ P}}(x_{ i}, x_{ u_{ l}})^{ 2}\right]$ for any $(u_{ 1}, u_{ 2}, u_{ 3})\in [n]^{ 3}$. This means that there exists a constant $C>0$ (independent of $i\in \left[n\right]$) such that
\begin{equation}
\mathbb{ E} \left[ \left(\frac{ 1}{ n}\sum_{ j=1}^{ n} \bar{\mathcal{ P}}(x_{ i}, x_{ j})\right)^{ 6}\right]\leq \frac{ C}{ n^{ 3}}.
\end{equation}
By a union bound on $i\in \left[n\right]$ and Markov inequality, we obtain, for all $ p\geq1$
\begin{equation*}
\mathbb{ P} \left( \sup_{ i\in \left[n\right]}\left\vert \frac{ 1}{ n}\sum_{ j=1}^{ n} \bar{\mathcal{ P}}(x_{ i}, x_{ j}) \right\vert\geq \frac{ 1}{ p}\right) \leq \sum_{ i=1}^{ n}\mathbb{ P} \left(\left\vert \frac{ 1}{ n}\sum_{ j=1}^{ n} \bar{\mathcal{ P}}(x_{ i}, x_{ j}) \right\vert\geq \frac{ 1}{ p}\right)\leq \frac{ Cp^{ 6}}{ n^{ 2}}.
\end{equation*}
An application of Borel-Cantelli Lemma shows that \eqref{hyp:uniform_conv_Sn} holds almost surely.
\end{proof}
We finish this section with an example in the deterministic case:
\begin{example}[\cite{10113716M1075831}, Ex.~2.1]
\label{ex:unbounded_graphon_ex2}
Consider here the model $( \mathcal{ G}^{ (n)}, \kappa^{ (n)})$ given by \eqref{eq:generic_Wn} and \eqref{hyp:dilution_alpha_ni_cP} where
\begin{equation}
\label{eq:cP_xy_alpha}
\mathcal{ P}(x, y) := (1- \alpha)^{ 2} x^{ - \alpha} y^{ - \alpha},\ \alpha\in \left[0, \frac{ 1}{ 2} \right),\ x, y \in I=[0, 1].
\end{equation}
\end{example}
\begin{proposition}
\label{prop:ex_cP_xy}
Suppose that Assumption~\ref{ass:deterministic_positions} holds. Let $ \alpha\in \left[0, \frac{ 1}{ 2}\right)$ and $ \mathcal{ P}$ defined by \eqref{eq:cP_xy_alpha}. There exists $ \delta(\alpha)< 1/2$ such that for all $ \delta(\alpha)< \delta< \frac{ 1}{ 2}$, the renormalized graph $( \mathcal{ G}^{ (n)}, \kappa^{ (n)})$ given by \eqref{eq:generic_Wn} and  \eqref{hyp:dilution_alpha_ni_cP} with $\rho_{ n}:= n^{ - \delta}$ converges to $W(x, y)= (1- \alpha)y^{ - \alpha}$, for the choice of $\kappa_{ n}:= \frac{ n^{ \delta}}{(1- \alpha)^{ 2}}$ and $w_{ n}:=1$ and $W$ is regular in the sense of Definition~\ref{def:graphon_regular}.
\end{proposition}
\begin{remark}
\label{rem:diff_Gn_sameW}
The macroscopic limits in Example~\ref{ex:unbounded_graphon_ex1} and Example~\ref{ex:unbounded_graphon_ex2} are the same, although the underlying graphs $ \mathcal{ G}^{ (n)}$ have really different structures. In Example~\ref{ex:unbounded_graphon_ex1}, $ \mathcal{ G}^{ (n)}$ is more or less homogeneous whereas Example~\ref{ex:unbounded_graphon_ex2} is much more hub-like: nodes with positions close to $0$ are connected to the whole population with probability close to $1$. We see here the effect of the renormalization \eqref{hyp:dilution_alpha_ni_cP}: it compensates for the hubs in the graph $ \mathcal{ G}^{ (n)}$ so that, even though the graphs $ \mathcal{ G}^{ (n)}$ might be different, the renormalized graphs $ \bar{\mathcal{ G}}^{ (n)}$ are actually quite similar.
\end{remark}
\begin{proof}[Proof of Proposition~\ref{prop:ex_cP_xy}] 
Recall the following simple asymptotics: for $ \alpha\in(0, 1)$, there exist some sequence $ \epsilon_{ n}\to_{ n\to\infty}0$ and some constant $C(\alpha)\neq0$ such that
\begin{equation}
\label{eq:asympt_Riemann}
\sum_{ k=1}^{ n} k^{ - \alpha} = \frac{ n^{ 1- \alpha}}{ 1- \alpha} + C(\alpha) + \epsilon_{ n},\ n\geq1.
\end{equation}
We first verify Definition~\ref{def:convergence_graph}. First, the easy case where $ \alpha\in \left[0, \frac{ 1}{ 4}\right)$ can be treated via Proposition~\ref{prop:conv_delta_n_cP}: the following rough bound $\sup_{ i, j\in \left[n\right]} \mathcal{ P}(x_{ i}, x_{ j}) \leq (1- \alpha)^{ 2} n^{ 2 \alpha}$ holds so that \eqref{hyp:control_cP} is true for $ \epsilon= 2 \alpha< \frac{ 1}{ 2}$. Morever,
\begin{align*}
\left\vert  \frac{ 1}{ n} \sum_{ j=1}^{ n} \mathcal{ P}(x_{ i}, x_{ j}) - \int \mathcal{ P}(x_{ i}, z) {\rm d}z\right\vert &= (1- \alpha) x_{ i}^{ - \alpha} \left\vert  \frac{1- \alpha}{ n^{ 1- \alpha}} \sum_{ j=1}^{ n}  j^{ - \alpha} -1\right\vert,
\end{align*}
which, since $ x_{ i}^{ - \alpha}\leq n^{ \alpha}$, is of order $n^{ -(1- 2 \alpha)}$, uniformly in $i\in \left[n\right]$. Thus, Proposition~\ref{prop:conv_delta_n_cP} is true, for any $ \delta$ such that $2 \alpha< \delta < \frac{ 1}{ 2}$. 

The case $ \alpha\in \left[ \frac{ 1}{ 4}, \frac{ 1}{ 2}\right)$ is more technical and cannot be dealt via Proposition~\ref{prop:conv_delta_n_cP} directly. Choose any $ \delta$ such that $ \alpha< \delta < \frac{ 1}{ 2} \leq 2 \alpha <1$. Since $x_{ i}\leq 1$, we have
\begin{align*}
\kappa_{i}^{ (n)} &= \frac{ n^{ 1+ \delta}}{\sum_{ j=1}^{ n} \min \left( n^{ \delta}, (1- \alpha)^{ 2} \frac{ 1}{ x_{ i}^{ \alpha} x_{ j}^{ \alpha}}\right)} \leq \frac{ n^{ \delta}}{(1- \alpha)^{ 2}},
\end{align*}
so that \eqref{hyp:alphan_VS_an} and \eqref{hyp:compare_wn_alphan} are true for $w_{ n}=1$. The choice of $ \delta< \frac{ 1}{ 2}$ ensures that \eqref{hyp:alpha_n_infty} is verified. For all $i\in \left[n\right]$, we have $ \delta_{ n}( \underline{ x}) = \sup_{ i\in \left[n\right]} \delta_{ n, i}( \underline{ x})$ with 
\begin{equation*}
\delta_{ n, i}( \underline{ x}):=\frac{ 1}{ n} \sum_{ k=1}^{ n} \left\vert \kappa_{i}^{ (n)} W_{ n}(x_{ i}, x_{ k}) - W(x_{ i}, x_{ k}) \right\vert.
\end{equation*}
Using the notations
\begin{align}
j_{ n}^{ i}&:= \left\lfloor (1- \alpha)^{ 2/ \alpha} \frac{ n^{ 2- \delta/ \alpha}}{ i} \right\rfloor  \label{hyp:j_n_i}
\end{align}
and
\begin{align*}
D_{ n, i}&:= \sum_{ j=1}^{ n} \min \left( n^{ \delta}, (1- \alpha)^{ 2} \frac{ 1}{ x_{ i}^{ \alpha} x_{ j}^{ \alpha}}\right)= j_{ n}^{ i} n^{ \delta} +(1- \alpha)^{ 2}  \frac{ n^{\alpha}}{ x_{ i}^{ \alpha}} \sum_{ j=j_{ n}^{ i}+1}^{n}  \frac{ 1}{ j^{ \alpha}},
\end{align*}
we obtain
\begin{align}
\delta_{ n, i}( \underline{ x})
%=\frac{ 1}{ n} \sum_{ k=1}^{ n} \left\vert \frac{ n}{D_{ n, i}} \min \left(n^{ \delta}, (1- \alpha)^{ 2} \frac{ n^{ 2 \alpha}}{ i^{ \alpha} k^{ \alpha}}\right) - (1- \alpha) \frac{ n^{ \alpha}}{ k^{ \alpha}} \right\vert, \nonumber\\
&=\frac{ 1}{ n} \sum_{ k=1}^{ j_{ n}^{ i}} \left\vert \frac{ n^{ 1+ \delta}}{D_{ n, i}} - (1- \alpha) \frac{ n^{ \alpha}}{ k^{ \alpha}} \right\vert + \left\vert \frac{(1- \alpha) n}{D_{ n, i} x_{ i}^{ \alpha}} - 1 \right\vert  \left(\frac{ 1- \alpha}{ n^{ 1- \alpha}} \sum_{ k=j_{ n}^{ i}+1}^{n} \frac{ 1}{ k^{ \alpha}}\right), \nonumber\\
&\leq \frac{ j_{ n}^{ i}n^{\delta}}{D_{ n, i}} + \frac{ 1- \alpha}{ n^{ 1- \alpha}} \sum_{ k=1}^{ j_{ n}^{ i}} \frac{ 1}{ k^{ \alpha}} + \left\vert \frac{(1- \alpha) n}{D_{ n, i} x_{ i}^{ \alpha}} - 1 \right\vert  \left(\frac{ 1- \alpha}{ n^{ 1- \alpha}} \sum_{ k=j_{ n}^{ i}+1}^{n} \frac{ 1}{ k^{ \alpha}}\right).\label{eq:sum_Delta_ni}
\end{align}
By \eqref{eq:asympt_Riemann}, we have for any $ \beta\in(0, 1]$
\begin{equation}
\label{eq:estim_Riem_2}
\frac{ 1- \alpha}{ n^{ 1- \alpha}} \sum_{ k=1}^{  \left\lfloor n^{ \beta}\right\rfloor} \frac{ 1}{ k^{ \alpha}} \begin{cases}
\to_{ n\to\infty} 1& \text{ if } \beta=1,\\ O \left( \frac{ 1}{ n^{ (1- \alpha)(1- \beta)}}\right)\to_{ n\to\infty} 0& \text{ if } \beta\in(0,1).
\end{cases}
\end{equation}
First observe that for all $i=1, \ldots, n$ $j_{ n}^{ i} \leq  \left\lfloor n^{ 2- \delta/ \alpha}\right\rfloor $. Since $ \delta>\alpha$, $2- \delta/ \alpha<1$ and by \eqref{eq:estim_Riem_2}, the second term in \eqref{eq:sum_Delta_ni} is such that
\begin{equation}
\label{eq:sum_jni_to_0}
\sup_{ i\in \left[n\right]}  \left(\frac{ 1- \alpha}{ n^{ 1- \alpha}} \sum_{ k=1}^{ j_{ n}^{ i}} \frac{ 1}{ k^{ \alpha}}\right)  \to 0, \text{ as }n\to\infty.
\end{equation}
Moreover we have the existence of $n_{ 0}\geq1$ such that for all $n\geq n_{ 0}$
\begin{equation}
\inf_{ i\in \left[n\right]} \left( \frac{ 1- \alpha}{ n^{ 1- \alpha}} \sum_{ k=j_{ n}^{ i}+1}^{ n} \frac{ 1}{ k^{ \alpha}}\right)\geq \frac{ 1}{ 2}.\label{eq:sum_jn_n_12}
\end{equation}
Let us now concentrate on the first term of \eqref{eq:sum_Delta_ni}: for all $i\in \left[n\right]$,
\begin{align*}
\frac{ j_{ n}^{ i} n^{\delta}}{D_{ n, i}}&= \frac{ 1}{ 1 +(1- \alpha)^{ 2}  \frac{ n^{2\alpha}}{i^{ \alpha}j_{ n}^{ i} n^{ \delta}} \sum_{ j=j_{ n}^{ i}+1}^{n}  \frac{ 1}{ j^{ \alpha}}},\\
%&\leq \frac{ 1}{ 1 +(1- \alpha)^{ 2- 2/ \alpha} n^{2\alpha - \delta -2 + \delta/ \alpha}i^{ 1- \alpha} \sum_{ j=j_{ n}^{ i}+1}^{n}  \frac{ 1}{ j^{ \alpha}}}, \text{ by } \eqref{hyp:j_n_i},\\
&\leq \frac{1}{ 1 +(1- \alpha)^{ 1- 2/ \alpha} n^{\alpha - \delta -1 + \delta/ \alpha} \left( \frac{ 1- \alpha}{ n^{ 1- \alpha}}\sum_{ j=j_{ n}^{ i}+1}^{n}  \frac{ 1}{ j^{ \alpha}}\right)},\text{ (by  \eqref{hyp:j_n_i} and since $ i\geq1$)},\\
&\leq \frac{1}{ 1 + \frac{ (1- \alpha)^{ 1- 2/ \alpha}}{ 2} n^{\alpha - \delta -1 + \delta/ \alpha}}\to 0,\ \text{ by }\eqref{eq:sum_jn_n_12},
\end{align*}
since $\alpha - \delta -1 + \delta/ \alpha = \frac{ \delta- \alpha}{ \alpha}(1- \alpha)>0$. This proves that the first term of \eqref{eq:sum_Delta_ni} converges to $0$ uniformly in $i\in \left[n\right]$. The third term of \eqref{eq:sum_Delta_ni} $\frac{ 1- \alpha}{ n^{ 1- \alpha}} \sum_{ k=j_{ n}^{ i}+1}^{n} \frac{ 1}{ k^{ \alpha}}$ is smaller than $\frac{ 1- \alpha}{ n^{ 1- \alpha}} \sum_{ k=1}^{n} \frac{ 1}{ k^{ \alpha}}$ which converges (to $1$) and hence, bounded. It remains to control $\left\vert \frac{(1- \alpha) n}{D_{ n, i} x_{ i}^{ \alpha}} - 1 \right\vert $. We can write
\begin{align*}
\left\vert \frac{(1- \alpha) n}{D_{ n, i} x_{ i}^{ \alpha}} - 1 \right\vert &= \frac{(1- \alpha) n}{D_{ n, i} x_{ i}^{ \alpha}} \left\vert  - \frac{j_{ n}^{ i} i^{ \alpha}}{ (1- \alpha)n^{ 1+ \alpha- \delta}} +1-  \frac{ 1- \alpha}{n^{ 1- \alpha}} \sum_{ j=j_{ n}^{ i}+1}^{n}  \frac{ 1}{ j^{ \alpha}} \right\vert,\\
%&\leq \frac{(1- \alpha) n}{D_{ n, i} x_{ i}^{ \alpha}} \left(\frac{j_{ n}^{ i} i^{ \alpha}}{ (1- \alpha)n^{ 1+ \alpha- \delta}} +  \frac{ 1- \alpha}{n^{ 1- \alpha}} \sum_{ j=1}^{j_{ n}^{ i}}  \frac{ 1}{ j^{ \alpha}} +\left\vert 1 - \frac{ 1- \alpha}{n^{ 1- \alpha}} \sum_{ j=1}^{n}  \frac{ 1}{ j^{ \alpha}} \right\vert \right),\\
&\leq \frac{(1- \alpha) n}{D_{ n, i} x_{ i}^{ \alpha}} \left(\frac{(1- \alpha)^{ 2/ \alpha-1}}{n^{ -1+ \alpha- \delta + \delta/ \alpha}} +  \frac{ 1- \alpha}{n^{ 1- \alpha}} \sum_{ j=1}^{j_{ n}^{ i}}  \frac{ 1}{ j^{ \alpha}} +\left\vert 1 - \frac{ 1- \alpha}{n^{ 1- \alpha}} \sum_{ j=1}^{n}  \frac{ 1}{ j^{ \alpha}} \right\vert \right).\end{align*}
The terms within the brackets converge to $0$, uniformly in $i\in \left[n\right]$, (recall \eqref{eq:estim_Riem_2} and \eqref{eq:sum_jni_to_0}) and we have
\begin{align*}
 \frac{(1- \alpha) n}{D_{ n, i} x_{ i}^{ \alpha}}&=  \frac{1- \alpha}{j_{ n}^{ i} n^{ \delta-1}x_{ i}^{ \alpha} +(1- \alpha)^{ 2} n^{\alpha-1} \sum_{ j=j_{ n}^{ i}+1}^{n}  \frac{ 1}{ j^{ \alpha}}}\leq \frac{1}{\frac{ 1- \alpha}{ n^{ 1- \alpha}} \sum_{ j=j_{ n}^{ i}+1}^{n}  \frac{ 1}{ j^{ \alpha}}}\leq 2,
\end{align*}
at least for large $n$, once again by \eqref{eq:sum_jn_n_12}. This proves \eqref{hyp:Delta1_to_0} in the case $ \alpha\in \left[ \frac{ 1}{ 4}, \frac{ 1}{ 2}\right)$. The proof of the regularity of $W$ has already been done in Section~\ref{sec:dense_bounded_graphons}. This concludes the proof of Proposition~\ref{prop:ex_cP_xy}.
\end{proof}

\section{Proofs for the propagation of chaos results}
\subsection{Proof of Theorem~\ref{theo:conv_general}}
\label{sec:proof_conv_general}
For the moment, the sequence of positions $\mathcal{ X}$ and the associated connectivity sequence $ \Xi$ are fixed. For all $n\geq1$, for fixed $ \underline{ x}$ and $ \underline{ \xi}$, we introduce the following quantities
\begin{align}
b_{ n}( \underline{ \xi})&:= \sup_{ i\in \left[n\right]} \left\vert \frac{1}{ n} \sum_{ k=1}^{ n} \kappa_{i}^{ (n)}\xi_{ i, k}^{ (n)} \right\vert, \label{hyp:bn_xi}\\
d_{ n, t}( \underline{ \xi}, \underline{ x})&:= \sup_{ i\in \left[n\right]}\int_{ 0}^{t}\mathbf{ E} \left[\left\vert \frac{ 1}{n} \sum_{ k=1}^{ n} \kappa_{i}^{ (n)}\left(\xi_{ i, k}^{ (n)} - W_{ n}(x_{ i}, x_{ k})\right) [\Gamma]_{u}(\bar\theta_{ i, u}, x_{ k}) \right\vert^{ 2}\right] {\rm d} u. \label{hyp:Delta_n_2}
\end{align}
\begin{proposition}
\label{prop:control_one}
Suppose that the hypotheses of Section~\ref{sec:general_assumptions_nonlin} are true. For $n\geq1$ and $T>0$, let $( \theta_{ i, t})_{ i\in \left[n\right]}:=( \theta_{ i, t}^{ (n)})_{ i\in \left[n\right]}$ be the solution of \eqref{eq:odegene} and $( \bar \theta_{i, t})_{ i\in \left[n\right]}:= ( \bar \theta_{t}^{ x_{ i}})_{ i\in \left[n\right]}$ independent copies of \eqref{eq:theta_nonlin} with the same initial conditions and Brownian motions as \eqref{eq:odegene}. There exists a constant $C>0$ (independent of $n$ and $T$), such that for any fixed choice of connectivities and positions $( \underline{ \xi}, \underline{ x})$,
\begin{multline}
\label{cln:control_one}
\sup_{ i\in \left[n\right]}\mathbf{ E}\left[ \sup_{ s\leq T} \left\vert \theta_{ i, s} - \bar \theta_{ i, s} \right\vert^{ 2}\right]\leq C\left(T + e^{ CT}\right)e^{ C b_{ n}( \underline{ \xi})^{ 2}T}\\ \times\left(b_{ n}( \underline{ \xi})\frac{ \kappa_{\infty}^{(n)}( \underline{ x})}{ n} + \delta_{ n}( \underline{ x})^{ 2}+ d_{ n, T}( \underline{ \xi}, \underline{ x}) + \sum_{ m=1}^{ 3} \epsilon_{ n, T}^{ (m)}( \underline{ x})\right),
\end{multline}
where we recall the definitions of $ \kappa_{ \infty}^{ (n)}( \underline{ x})$ in \eqref{hyp:alphan_VS_an}, of $ \delta_{ n}( \underline{ x})$ in \eqref{hyp:Delta_n_1} and of $ \epsilon_{ n, T}^{ (m)}( \underline{ x})$ in \eqref{hyp:Delta3_to_0}.
\end{proposition}
\begin{proof}[Proof of Proposition~\ref{prop:control_one}]
 For all $s\leq t$, $i\in \left[n\right]$, by the one-sided Lipschitz-continuity \eqref{hyp:c_onesidedLip} of $c$,
\begin{multline*}
\left\vert \theta_{ i, s} - \bar \theta_{ i, s} \right\vert^{ 2} =  2 \int_{ 0}^{s} \left\langle \theta_{ i, u} - \bar \theta_{ i, u}\, ,\, c( \theta_{ i, u}) - c( \bar \theta_{ i, u})\right\rangle {\rm d} u \\ 
+ 2\int_{ 0}^{s}  \Bigg\langle \theta_{ i, u} - \bar \theta_{ i, u}\, ,\, \Bigg(\frac{ 1}{n} \sum_{ k=1}^{ n} \kappa_{i}^{ (n)}\xi_{ i, k} \Gamma( \theta_{ i, u}, \theta_{ k, u}) - \int W(x_{ i}, \tilde x)\Gamma(\bar \theta_{ i, u}, \tilde\theta) \nu_{ u}({\rm d} \tilde\theta, {\rm d} \tilde x)\Bigg)\Bigg\rangle{\rm d} u,\\
\leq (2 L_{ c} +1)\int_{ 0}^{s}  \left\vert \theta_{ i, u} - \bar \theta_{ i, u} \right\vert^{ 2}{\rm d} u 
+ \int_{ 0}^{s} \left\vert \frac{ 1}{n} \sum_{ k=1}^{ n} \kappa_{i}^{ (n)}\xi_{ i, k} \Gamma( \theta_{ i, u}, \theta_{ k, u}) - \int W(x_{ i}, \tilde x)\Gamma(\bar \theta_{ i, u}, \tilde\theta) \nu_{ u}({\rm d} \tilde\theta, {\rm d} \tilde x) \right\vert^{ 2}  {\rm d} u.
\end{multline*}
Taking the supremum in $s\leq t$ and the expectation w.r.t. Brownian motions and initial conditions,
\begin{multline}
\mathbf{ E} \left[ \sup_{ s\leq t}\left\vert \theta_{ i, s} - \bar \theta_{ i, s} \right\vert^{ 2}\right] \leq  (2 L_{ c}+1)\int_{ 0}^{t}\mathbf{ E} \left[ \sup_{ v\leq u}\left\vert  \theta_{ i, v} - \bar \theta_{ i, v}\right\vert^{ 2}\right] {\rm d} u \\ + \int_{0}^{t} \mathbf{ E} \left[\left\vert \frac{ 1}{n} \sum_{ k=1}^{ n} \kappa_{i}^{(n)}\xi_{ i, k} \Gamma( \theta_{ i, u}, \theta_{ k, u}) - \int W(x_{ i}, \tilde x)\Gamma(\bar \theta_{ i, u}, \tilde\theta) \nu_{ u}({\rm d} \tilde\theta, {\rm d} \tilde x) \right \vert^{ 2}\right]{\rm d} u. \label{eq:diff_thetai} 
\end{multline}
It remains to control the last term in \eqref{eq:diff_thetai}, which can be bounded from above by $ \int_{ 0}^{t}\left(6\sum_{ k=1}^{ 6}A_{ n, i, u}^{ (k)}\right){\rm d} u$, where, for $u\leq t$,
\begin{align}
A_{ n, i, u}^{ (1)}&:= \mathbf{ E} \left[\left\vert \frac{ 1}{ n} \sum_{ k=1}^{ n} \kappa_{i}^{ (n)}\xi_{ i, k}\left(\Gamma( \theta_{ i, u}, \theta_{ k, u})- \Gamma( \bar\theta_{ i, u}, \theta_{ k, u})\right)\right\vert^{ 2}\right],\\
A_{ n, i, u}^{ (2)}&:= \mathbf{ E} \left[\left\vert \frac{1}{n} \sum_{ k=1}^{ n} \kappa_{i}^{ (n)}\xi_{ i, k} \left(\Gamma( \bar\theta_{ i, u}, \theta_{ k, u}) - \Gamma( \bar\theta_{ i, u}, \bar\theta_{ k, u}) \right) \right\vert^{ 2}\right],\\
A_{ n, i, u}^{ (3)}&:= \mathbf{ E} \left[\left\vert \frac{ 1}{ n} \sum_{ k=1}^{ n} \kappa_{i}^{ (n)}\xi_{ i, k} \left(\Gamma( \bar\theta_{ i, u}, \bar\theta_{ k, u}) -  \left[ \Gamma\right]_{ u}(\bar \theta_{ i, u}, x_{ k})\right) \right\vert^{ 2}\right],\label{eq:An3}\\
A_{ n, i, u}^{ (4)}&:= \mathbf{ E} \left[\left\vert \frac{ 1}{n} \sum_{ k=1}^{ n} \kappa_{i}^{ (n)}\left(\xi_{ i, k} - W_{ n}(x_{ i}, x_{ k})\right) \left[ \Gamma\right]_{ u}(\bar \theta_{ i, u}, x_{ k}) \right\vert^{ 2}\right],\label{eq:An4}\\
A_{ n, i, u}^{ (5)}&:= \mathbf{ E} \left[\left\vert \frac{1}{ n} \sum_{ k=1}^{ n} \left(\kappa_{i}^{ (n)}W_{ n}(x_{ i}, x_{ k}) - W(x_{ i}, x_{ k})\right) \left[ \Gamma\right]_{ u}(\bar \theta_{ i, u}, x_{ k}) \right\vert^{ 2}\right],\\
A_{ n, i, u}^{ (6)}&:= \mathbf{ E} \left[\left\vert \frac{ 1}{ n} \sum_{ k=1}^{ n}W(x_{ i}, x_{ k}) \left[ \Gamma\right]_{ u}(\bar \theta_{ i, u}, x_{ k}) - \int W(x_{ i}, \tilde x)  \left[ \Gamma\right]_{ u}(\bar \theta_{ i, u}, \tilde{ x})\ell({\rm d} \tilde{ x}) \right\vert^{ 2}\right],\label{eq:An6}
\end{align}
where we recall the definition of $\left[ \Gamma\right]_{ u}$ in \eqref{eq:G}. Note as this point that, $d_{ n, t}$ defined in \eqref{hyp:Delta_n_2} is such that $ d_{ n, t}=  \sup_{ i\in \left[n\right]}\int_{ 0}^{t} A_{ n, i, u}^{ (4)} {\rm d} u$. Among the other terms, $A_{ n}^{ (1)}$ and $A_{ n}^{ (2)}$ capture the approximation of the particle system $( \theta_{ i})_{ i}$ by its mean-field limit $( \bar\theta_{ i})_{ i}$ and $A_{ n}^{ (3)}$ relates the empirical measure of the mean-field particle system to its deterministic limit $ \nu$.  The convergence of $(W_{ n}(x_{ i}, x_{ j}))_{ i, j}$ to the macroscopic kernel $(W(x_{ i}, x_{ j}))_{ i, j}$ is controlled by $A_{ n}^{ (5)}$. By the Lipschitz continuity of $ \Gamma$ \eqref{hyp:Gamma_Lip_1}, we have (recall the definition of $b_{ n}$ in \eqref{hyp:bn_xi}),
\begin{align*}
A_{ n, i, u}^{ (1)} &\leq L_{ \Gamma}^{ 2} b_{ n}( \underline{ \xi})^{ 2}\sup_{ r\in \left[n\right]}\mathbf{ E} \left[\sup_{ v\in[0, u]}\left\vert \theta_{ r, v} - \bar \theta_{ r, v} \right\vert^{ 2} \right].
\end{align*}
In a same way, by \eqref{hyp:Gamma_Lip_1},
\begin{align*}
A_{ n, i, u}^{ (2)} &\leq L_{ \Gamma}^{ 2}  \mathbf{ E}\left[\frac{1}{n} \sum_{ k=1}^{ n} \kappa_{i}^{ (n)} \xi_{ i, k} \left\vert \theta_{ k, u}- \bar \theta_{ k, u} \right\vert\right]^{ 2}= \frac{L_{ \Gamma}^{ 2}}{n^{ 2}} \sum_{ k, l=1}^{ n} \left(\kappa_{i}^{ (n)}\right)^{ 2}\xi_{ i, k}\xi_{ i, l} \mathbf{ E}\left[\left\vert \theta_{ k, u}- \bar \theta_{ k, u} \right\vert \left\vert \theta_{ l, u}- \bar \theta_{ l, u} \right\vert\right], \\
&\leq \frac{L_{ \Gamma}^{ 2}}{2 n^{ 2}} \sum_{ k, l=1}^{ n} \left(\kappa_{i}^{ (n)}\right)^{ 2}\xi_{ i, k}\xi_{ i, l} \mathbf{ E}\left[\left\vert \theta_{ k, u}- \bar \theta_{ k, u} \right\vert^{ 2} + \left\vert \theta_{ l, u}- \bar \theta_{ l, u} \right\vert^{ 2}\right]\leq L_{ \Gamma}^{ 2}  b_{ n}( \underline{ \xi})^{ 2} \sup_{ r\in \left[n\right]}\mathbf{ E} \left[\sup_{ v\in[0, u]}\left\vert \theta_{ r, v} - \bar \theta_{ r, v} \right\vert^{ 2}\right].
\end{align*} 
Concerning the term $A^{ (3)}$, denote by (recall \eqref{eq:G})
\begin{equation}
\label{eq:Delta}
\delta \Gamma( \bar\theta_{ i,u}, \bar \theta_{ k, u}):= \left(\Gamma( \bar\theta_{ i, u}, \bar\theta_{ k, u}) -  \left[ \Gamma\right]_{ u}(\bar \theta_{ i, u}, x_{ k})\right)= \left(\Gamma( \bar\theta_{ i, u}, \bar\theta_{ k, u}) - \int \Gamma( \bar\theta_{ i, u}, \tilde\theta) \nu_{ u}^{x_{ k}}({\rm d} \tilde\theta)\right).
\end{equation} 
One has in particular, using \eqref{hyp:Gamma_bound} and Remark~\ref{rem:apriori_nonlin},
\begin{align}
\mathbf{ E}\left[\left\vert \delta \Gamma( \bar\theta_{ i,u}, \bar \theta_{ k, u}) \right\vert^{ 2}\right] &\leq 2 \mathbf{ E} \left[\left\vert \Gamma( \bar\theta_{ i, u}, \bar\theta_{ k, u}) \right\vert^{ 2}\right] +2 \mathbf{ E}  \left[\int \left\vert \Gamma( \bar\theta_{ i, u}, \tilde\theta) \right\vert \nu_{ u}^{x_{ k}}({\rm d} \tilde\theta)\right]^{ 2}, \nonumber\\
&\leq 2 L_{ \Gamma}^{ 2} \mathbf{ E}\left[1+ \left\vert \bar\theta_{ i, u} \right\vert + \left\vert \bar \theta_{ k, u} \right\vert\right]^{ 2} + 2 L_{ \Gamma}^{ 2} \mathbf{ E}\left[1+ \left\vert \bar\theta_{ i, u} \right\vert + \mathbf{ E}\left\vert \bar \theta_{ k, u} \right\vert\right]^{ 2},\nonumber\\
&\leq 12 L_{ \Gamma}^{ 2} \left(1+\mathbf{ E}\left[\left\vert \bar\theta_{ i, u} \right\vert^{ 2}\right] + \mathbf{ E} \left[\left\vert \bar \theta_{ k, u} \right\vert^{ 2}\right]\right),\nonumber\\
&\leq 12 L_{ \Gamma}^{ 2}\left(1+ 2\sup_{ r\in \left[n\right]}\mathbf{ E} \left[\sup_{ v\leq u}\left\vert \bar\theta_{ r, v} \right\vert^{ 2}\right]\right)\leq 12 L_{ \Gamma}^{ 2} \left(1+ 2C_{ 0}\right),\label{eq:Esp_Delta}
\end{align}
for $C_{ 0}$ given by \eqref{eq:apriori_bound_nonlin}. Thus,
\begin{align*}
A_{ n, i, u}^{ (3)} &\leq  \mathbf{ E} \left[\left\vert \frac{ 1}{ n} \sum_{ k=1}^{ n} \kappa_{i}^{ (n)}\xi_{ i, k}\delta \Gamma( \bar\theta_{ i,u}, \bar \theta_{ k, u}) \right\vert^{ 2}\right],\\
&=\frac{  \left(\kappa_{i}^{ (n)}\right)^{ 2}}{n^{ 2}} \sum_{ k=1}^{ n} \xi_{ i, k} \mathbf{ E} \left[ \left\vert \delta \Gamma( \bar\theta_{ i,u}, \bar \theta_{ k, u}) \right\vert^{ 2}\right]+ \frac{ \left(\kappa_{i}^{ (n)}\right)^{ 2}}{ n^{ 2}} \sum_{ k\neq l}^{ n} \xi_{ i, k} \xi_{ i, l} \mathbf{ E} \left[ \left\langle \delta \Gamma( \bar\theta_{ i,u}, \bar \theta_{ k, u})\, ,\, \delta \Gamma( \bar\theta_{ i,u}, \bar \theta_{ l, u})\right\rangle\right].
\end{align*}
Since we have supposed that $ \xi_{ i,i}=0$, one can suppose that $ k\neq i$ and $l\neq i$ in the second sum and conditioning by $ \mathcal{ F}_{ k}:=\sigma( \bar \theta_{ r, u}, r\neq k)$ gives
\begin{align*}
\mathbf{ E} \left[ \left\langle  \delta \Gamma( \bar\theta_{ i, u}, \bar \theta_{ k, u})\, ,\, \delta \Gamma( \bar\theta_{ i, u}, \bar \theta_{ l, u})\right\rangle\right] &= \mathbf{ E} \left[ \mathbf{ E} \left[ \left\langle \delta \Gamma( \bar\theta_{ i, u}, \bar \theta_{ k, u})\, ,\, \delta \Gamma( \bar\theta_{ i, u}, \bar \theta_{ l, u})\right\rangle\vert \mathcal{ F}_{ k}\right]\right],\\
&=\mathbf{ E} \left[ \left\langle \delta \Gamma( \bar\theta_{ i, u}, \bar \theta_{ l, u})\, ,\, \mathbf{ E} \left[\delta \Gamma( \bar\theta_{ i, u}, \bar \theta_{ k, u})\vert \mathcal{ F}_{ k}\right]\right\rangle\right]=0,
\end{align*} by definition of $ \delta \Gamma(\bar \theta_{ i, u}, \bar \theta_{ k, u})$. Consequently, by \eqref{eq:Esp_Delta} (recall the definition of $ \kappa_{\infty}^{ (n)}( \underline{ x})$ in \eqref{hyp:alphan_VS_an}), 
\begin{align*}
A_{ n, i, u} ^{ (3)} &\leq 12 L_{ \Gamma}^{ 2} \left(1+ 2C_{ 0}\right) \left(\frac{ \left(\kappa_{i}^{ (n)}\right)^{ 2}}{ n^{ 2}} \sum_{ k=1}^{ n} \xi_{ i, k} \right)\leq 12 L_{ \Gamma}^{ 2} \left(1+ 2C_{ 0}\right) b_{ n}( \underline{ \xi}) \frac{ \kappa_{\infty}^{(n)}( \underline{ x})}{ n}.
\end{align*}
Concerning the term $A_{ n}^{ (5)}$, using \eqref{eq:bound_Gamma_u} and \eqref{eq:apriori_bound_nonlin} (recall the definition of $ \delta_{ n}( \underline{ x})$ in \eqref{hyp:Delta_n_1}), we have
\begin{align*}
A_{ n, i, u}^{ (5)} &\leq  3 L_{ \Gamma}^{ 2}\left(1+ 2C_{ 0}\right) \left(\frac{1}{ n} \sum_{ k=1}^{ n} \left\vert \kappa_{i}^{ (n)}W_{ n}(x_{ i}, x_{ k}) - W(x_{ i}, x_{ k})\right\vert\right)^{ 2}=3 L_{ \Gamma}^{ 2}\left(1+ 2C_{ 0}\right) \delta_{ n}( \underline{ x})^{ 2}.
\end{align*}
Finally, let us control the last term $A_{ n}^{ (6)}$: using the shortcut
\begin{equation}
\overline{[ \Gamma W]}_{ u}(\theta, x, y)=W(x, y) [ \Gamma]_{ u}( \theta, y) - \int W(x, z) [ \Gamma]_{ u}(\theta, z)\ell({\rm d} z)
\end{equation}
we have (recall the definition of $\Upsilon_{ t}$ in \eqref{eq:upsilon_t}) 
\begin{align*}
\int_{ 0}^{ t}&A_{ n, u}^{ (6)} {\rm d}u= \int_{ 0}^{t}\int\mathbf{ E} \left[\left\langle \overline{[ \Gamma W]}_{ u}(\bar\theta_{ i, u}, x_{ i}, y)\, ,\, \overline{[ \Gamma W]}_{ u}(\bar\theta_{ i, u}, x_{ i}, z)\right\rangle\right]\ell_{ n}({\rm d}y)\ell_{ n}({\rm d}z) {\rm d}u\\
&= \int W(x_{ i}, y)W(x_{ i}, z) \Upsilon_{ t}(x_{ i}, y, z)\ell_{ n}({\rm d}y)\ell_{ n}({\rm d}z)+\int W(x_{ i}, y) W(x_{ i}, z) \Upsilon_{ t}(x_{ i}, y, z) \ell({\rm d} y)\ell({\rm d} z)\\
&- \int W(x_{ i}, y)W(x_{ i}, z) \Upsilon_{ t}(x_{ i}, y, z) \ell_{ n}({\rm d}y)\ell({\rm d} z)- \int W(x_{ i}, z) W(x_{ i}, y) \Upsilon_{ t}(x_{ i}, y, z) \ell({\rm d} y)\ell_{ n}({\rm d}z).
\end{align*} 
So that (recall the definition of the $\epsilon_{ n, T}^{ (m, i)}( \underline{ x})$ in \eqref{hyp:epsilon_i_123}),
\begin{align*}
\sup_{ i\in \left[n\right]} \int_{ 0}^{t}A_{ n, u}^{ (6)} {\rm d}u\leq \sum_{ m=1}^{ 3}  \epsilon_{ n, t}^{ (m)}( \underline{ x}).
\end{align*}
Define now
\begin{equation}
f_{ t}:= \sup_{ i\in \left[n\right]}\mathbf{ E} \left[ \sup_{ s\leq t}\left\vert \theta_{ i, s} - \bar \theta_{ i, s} \right\vert^{ 2}\right], t\leq T.
\end{equation}
Taking the supremum on $ i\in \left[n\right]$ and gathering all the previous estimates in \eqref{eq:diff_thetai} gives, for some constant $C=C(\Gamma, c, W, \sigma, \nu_{ 0})>0$,
\begin{align*}
f_{ t} &\leq  C b_{ n}( \underline{ \xi})^{ 2}\int_{ 0}^{t} f_{ u}{\rm d} u + C\left(t + e^{ Ct}\right) \left(b_{ n}( \underline{ \xi})\frac{ \kappa_{\infty}^{(n)}( \underline{ x})}{ n} + \delta_{ n}( \underline{ x})^{ 2}\right) + d_{ n, t}( \underline{ \xi}, \underline{ x}) + \sum_{ m=1}^{ 3}\epsilon_{ n, t}^{ (m)}(\underline{ x}),\ t\in[0, T].
\end{align*}
An application of Gr\"onwall's Lemma gives the conclusion. This proves Proposition~\ref{prop:control_one}.
\end{proof}
At this point of the proof, in \eqref{cln:control_one}, $ \delta_{ n}( \underline{ x})$ and $\sum_{ m=1}^{ 3} \epsilon_{ n, T}^{ (m)}( \underline{ x})$ go to $0$ as $n\to\infty$, by hypothesis. The point now is to prove that the two remaining terms in \eqref{cln:control_one} (that depend on the realization of the connectivity sequence $ \underline{ \xi}$) are such that, first, $b_{ n}( \underline{ \xi})$ is bounded and second, that $d_{ n, T}( \underline{ \xi}, \underline{ x})$ goes to $0$ as $n\to\infty$, almost surely. This is the purpose of Proposition~\ref{prop:control_Delta_2} below. The following concentration estimate may be found in \cite{Dembo1998}, Corollary 2.4.7:
\begin{lemma}
\label{lem:concentration_dembo}
Fix $n\geq 1$ and $(Y_{ l})_{ l=1, \ldots, n}$ real valued random variables defined on a probability space $( \Omega, \mathcal{ F}, \mathbb{ P})$. Suppose that there exists $v>0$ such that, almost surely, for all $l=1, \ldots, n-1$, $Y_{ l}\leq 1$, $ \mathbb{ E} \left[Y_{ l+1}\vert Y_{ l}\right]= 0$ and $ \mathbb{ E} \left[Y_{l+1}^{ 2}\vert Y_{ l}\right] \leq v$. Then for all $x\geq0$,
\begin{equation}
\mathbb{ P} \left(n^{ -1} (Y_{ 1}+ \ldots + Y_{ n}) \geq x\right) \leq \exp \left(-n H\left( \frac{ x+v}{ 1+v} \Big\vert \frac{ v}{ 1+v}\right)\right),
\end{equation}
where $H(p\vert q)= p \log(p/q) + (1-p) \log((1-p)/(1-q))$, for $p, q\in [0, 1]$.
\end{lemma}
Using this result, one can prove the following
\begin{lemma}
\label{lem:large_deviations}
Fix $n\geq1$, $ \kappa_{ n}>0$, $(p_{ 1}, \ldots, p_{ n})$ in $[0, 1]$ and a sequence $(v_{ 1}, \ldots, v_{ n})$ such that $ \left\vert v_{ l} \right\vert \leq 1$ for all $l\in \left[n\right]$. Suppose that there exist $\kappa_{ n}>0$ and $w_{ n}\in(0, 1]$, $n\geq1$ satisfying \eqref{hyp:compare_wn_alphan} and \eqref{hyp:alpha_n_infty}  such that $0< \kappa_{ n}< \kappa_{ n}$, $p_{ l}\leq w_{ n}$ for all $l\in \left[n\right]$. If $ (U_{ 1}, \ldots, U_{ n})$ are independent random variable with $ U_{ l}\sim \mathcal{ B}(p_{ l})$ for all $l\in \left[n\right]$, we have the following estimate
\begin{equation}
\label{eq:LD_gen}
\mathbb{ P} \left(\left\vert \frac{ \kappa_{ n}}{ n} \sum_{ l=1}^{ n} ( U_{ l}- p_{ l}) v_{ l} \right\vert > \varepsilon_{ n}\right) \leq 2\exp \left(-16\log(n) B \left( 4 \sqrt{ 2}\left(\frac{\log(n)}{ n w_{ n}}\right)^{ 1/2} \right)\right),
\end{equation}
where
\begin{equation}
\label{eq:func_B}
B(u):= u^{ -2} \left[(1+u)\log(1+u)-u\right]
\end{equation}
for the choice of
\begin{equation}
\label{eq:xve_n}
\varepsilon_{ n}^{ 2}:= 32 \frac{ \kappa_{ n}^{ 2} w_{ n}}{ n} \log(n),
\end{equation}
which goes to $0$ as $n\to\infty$, by \eqref{hyp:alpha_n_infty}.
\end{lemma}
\begin{proof}[Proof of Lemma~\ref{lem:large_deviations}]
 Fix $n\geq1$, $ \kappa_{ n}>0$, $(p_{ l})_{ l\in \left[n\right]}$, $(v_{ l})_{ l\in \left[n\right]}$ and $( U_{ l})_{ l\in \left[n\right]}$ previously defined. Let $Y_{l}:=( U_{ l}- p_{ l})v_{ l}$. $(Y_{1}, \ldots, Y_{n})$ are independent random variables such that for all $l\in \left[n\right]$, $ \mathbb{ E}[Y_{l}]=0$ and $ \mathbb{ E} \left[Y_{l}^{ 2}\right]\leq w_{ n}$. Then, for all $ \varepsilon>0$, by Lemma~\ref{lem:concentration_dembo},
\begin{align*}
\mathbb{ P} \left( \frac{ \kappa_{ n}}{ n} \sum_{ l=1}^{ n} ( U_{ l}- p_{ l}) v_{ l} > \varepsilon\right)&\leq \exp \left(-n H \left( \frac{ \varepsilon \kappa_{ n}^{ -1}+ w_{ n}}{ 1+ w_{ n}}\Big \vert \frac{ w_{ n}}{ 1+ w_{ n}}\right)\right)\leq \exp \left(-\frac{ n \varepsilon^{ 2}}{ 2w_{ n} \kappa_{ n}^{ 2}} B \left( \frac{\varepsilon }{ w_{ n} \kappa_{ n}}\right)\right),
\end{align*}
where we used the inequality (\cite{Dembo1998}, Exercise~2.4.21)
\begin{equation}
\label{aux:ineq_H}
H \left( \frac{ x+v}{ 1+v} \Big\vert \frac{ v}{ 1+v}\right) \geq \frac{ x^{ 2}}{ 2v} B \left( \frac{ x}{ v}\right),\ x,v>0.
\end{equation}
Since $ u \mapsto u^{ 2}B(u)$ is nondecreasing, $\exp \left(-\frac{ n \varepsilon^{ 2}}{ 2w_{ n} \kappa_{ n}^{ 2}} B \left( \frac{\varepsilon }{ w_{ n} \kappa_{ n}}\right)\right)\leq \exp \left(-\frac{ n \varepsilon^{ 2}}{ 2w_{ n} \kappa_{ n}^{ 2}} B \left( \frac{\varepsilon }{ w_{ n} \kappa_{ n}}\right)\right)$, so that, for the choice of $ \varepsilon= \varepsilon_{ n}$ defined by \eqref{eq:xve_n}, we have
\begin{align*}
\mathbb{ P} \left(\frac{ \kappa_{ n}}{ n} \sum_{ l=1}^{ n} ( U_{ l}- p_{ l}) v_{ l} > \varepsilon_{ n}\right) &\leq \exp \left(- 16 \log(n) B \left( 4 \sqrt{ 2}\left(\frac{\log(n)}{ n w_{ n}}\right)^{ 1/2} \right)\right)
\end{align*}
Doing the same for the sequence $(- v_{ l})_{ l\in \left[n\right]}$, one obtains \eqref{eq:LD_gen}. This proves Lemma~\ref{lem:large_deviations}.\end{proof}
Theorem~\ref{theo:conv_general} is an immediate consequence of the following result:
\begin{proposition}
\label{prop:control_Delta_2}
Suppose that the hypotheses of Section~\ref{sec:general_assumptions_nonlin} and Section~\ref{sec:general_prop_chaos} (namely   \eqref{hyp:compare_wn_alphan} and \eqref{hyp:alpha_n_infty}) are true. Suppose that the sequence of positions $ \mathcal{ X}$ is such that $ \delta_{ n}( \underline{ x})$ goes to $0$ as $n\to \infty$.
There exist a deterministic sequence $( \varepsilon_{ n})_{ n\geq1}$ with $ \varepsilon_{ n}\to0$ as $n\to\infty$, a constant $C>0$ such that for all $T>0$, there is an event $ \mathcal{ O}\in \mathcal{ F}$ with $ \mathbb{ P}( \mathcal{ O})=1$ such that the following is true: for every $ \omega\in \mathcal{ O}$, there exists $n_{ 0}<+\infty$, such that for all $n\geq n_{ 0}$,
\begin{align}
b_{ n}( \underline{ \xi}(\omega)) & \leq 2 + \sup_{ i\in \left[n\right]} \frac{ 1}{ n} \sum_{ k=1}^{ n} W(x_{ i}, x_{ k}),\label{eq:control_sum_xi}\\
d_{ n, T}( \underline{ \xi}(\omega), \underline{ x})&\leq \varepsilon_{ n} C \left(T + e^{ CT}\right) \left( 1 + \sup_{ i\in \left[n\right]} \frac{ 1}{ n} \sum_{ k=1}^{ n} W(x_{ i}, x_{ k})\right).\label{eq:control_int_An4}
\end{align}
\end{proposition}
\begin{proof}[Proof of Proposition~\ref{prop:control_Delta_2}]
Introduce the following notations
\begin{align}
\bar \xi_{ i,k}=\bar \xi_{ i,k}&:= \xi_{ i, k} - W_{ n}(x_{ i}, x_{ k}),\label{eq:xibar}\\
\gamma_{ k,l, u}^{ (i)} &=\mathbf{ E} \left[[\Gamma]_{ u}(\bar\theta_{ i, u}, x_{ k})[\Gamma]_{ u}(\bar\theta_{ i, u}, x_{ l})\right],\ k,l\in \left[n\right],
\end{align}
so that we can rewrite \eqref{eq:An4} as
\begin{align}
A_{ n, i, u}^{ (4)}&= \frac{\left(\kappa_{i}^{ (n)}\right)^{ 2}}{ n^{ 2}} \sum_{ k,l=1}^{ n} \bar \xi_{ i, k}\bar \xi_{ i, l}\gamma_{k, l, u}^{ (i)}.\label{aux:Usik}
\end{align}
Using \eqref{eq:bound_Gamma_u}, for some constant $C>0$ independent of $i, k, l, u$,
\begin{align}
\left\vert \gamma_{ k,l, u}^{ (i)} \right\vert &\leq\mathbf{ E} \left[ \left\vert [\Gamma]_{ u}(\bar\theta_{ i, u}, x_{ k})[\Gamma]_{ u}(\bar\theta_{ i, u}, x_{ l}) \right\vert\right]\leq \frac{ 1}{ 2}\mathbf{ E} \left[ \left\vert [\Gamma]_{ u}(\bar\theta_{ i, u}, x_{ k})\right\vert^{ 2}\right] + \frac{ 1}{ 2}\mathbf{ E} \left[ \left\vert [\Gamma]_{ u}(\bar\theta_{ i, u}, x_{ l}) \right\vert^{ 2}\right],\nonumber\\
&\leq C \left(1+ \mathbf{ E}\left[ \left\vert \bar \theta_{ i, u} \right\vert^{ 2}\right]\right)\leq C \left(1+ \sup_{ r\in \left[n\right]}\mathbf{ E} \left[\left\vert \bar\theta_{ r,u} \right\vert^{ 2}\right] \right)\leq C \left(1+ C_{ 0}\right),\label{aux:bound_gamma_ikn}
\end{align} 
using \eqref{eq:apriori_bound_nonlin}. Setting 
\begin{align}
G_{ k, l, t}^{ (i)}&:=\int_{ 0}^{t} \gamma_{ k, l, u}^{ (i)}{\rm d}u, \text{ and }\left\Vert G_{ t} \right\Vert_{ \infty}:= \sup_{ k,l,i} \left\vert G_{ k, l, t}^{ (i)} \right\vert,
\end{align} 
one obtains from \eqref{aux:bound_gamma_ikn} that 
\begin{equation}
\label{eq:bound_G_t}
\left\Vert G_{ t} \right\Vert_{ \infty}\leq C(t + e^{ Ct})
\end{equation}
for some appropriate constant $C>0$. Let us define now
\begin{equation}
X_{ i, k, t}^{ (n)}:= \frac{ \kappa_{i}^{ (n)}}{ n} \sum_{ l=1}^{ n} \bar \xi_{ i, l} \frac{ G_{ k,l,t}^{ (i)}}{ \left\Vert G_{ t} \right\Vert_{ \infty}}
\end{equation}
we can write
\begin{align*}
\int_{ 0}^{t}A_{ n, i, u}^{ (4)} {\rm d}u&=\frac{ \kappa_{i}^{ (n)} \left\Vert G_{ t} \right\Vert_{ \infty}}{ n} \sum_{ k=1}^{ n}  (\xi_{ i, k}- W_{ n}(x_{ i}, x_{ k})) X_{i, k, t}^{ (n)},\\
&=\frac{ \kappa_{i}^{ (n)} \left\Vert G_{ t} \right\Vert_{ \infty}}{ n} \sum_{ k=1}^{ n} (1- W_{ n}(x_{ i}, x_{ k}))\xi_{ i, k} X_{ i, k, t}^{ (n)} + \frac{ \kappa_{i}^{ (n)} \left\Vert G_{ t} \right\Vert_{ \infty}}{ n} \sum_{ k=1}^{ n} (-W_{ n}(x_{ i}, x_{ k}))(1-\xi_{ i, k}) X_{ i, k, t}^{ (n)}.
\end{align*}
Consequently, almost surely, for all $t\geq0$, the following inequality holds
\begin{equation}
\label{eq:bound_An4}
d_{ n, t}\leq \left\Vert G_{ t} \right\Vert_{ \infty}\left(\sup_{ i, k\in \left[n\right]}\left\vert X_{ i, k, t}^{ (n)} \right\vert\right) \sup_{ i\in \left[n\right]}\left(\frac{ \kappa_{i}^{ (n)}}{ n} \sum_{ k=1}^{ n} \xi_{ i, k} + \frac{ \kappa_{i}^{ (n)}}{ n} \sum_{ k=1}^{ n} W_{ n}(x_{ i}, x_{ k})\right).
\end{equation}
We first derive a uniform bound on $ \left(X_{ i, k, t}^{ (n)}\right)_{ i, k\in \left[n\right]}$. For fixed $i,k\in \left[n\right]$, apply Lemma~\ref{lem:large_deviations} for the choice of $ U_{ l}= \xi_{ i, l}$,  $ \kappa_{ n}= \kappa_{i}^{ (n)}$, $p_{ l}= W_{ n}(x_{ i}, x_{ l})$ and $v_{ l}= \frac{ G_{ k,l,t}^{ (i)}}{ \left\Vert G_{ t} \right\Vert_{ \infty}}$: inequality \eqref{eq:LD_gen} together with a simple union bound gives
\begin{align*}
\mathbb{ P} \left( \sup_{ i,k\in \left[n\right]} \left\vert X_{ i, k, t}^{ (n)} \right\vert > \varepsilon_{ n}\right) &\leq 2n^{ 2}\exp \left(-16\log(n) B \left( 4 \sqrt{ 2}\left(\frac{\log(n)}{ n w_{ n}}\right)^{ 1/2} \right)\right).
\end{align*}
Note that under the assumptions \eqref{hyp:compare_wn_alphan} and \eqref{hyp:alpha_n_infty} on $ \kappa_{ n}$ and $w_{ n}$, we have \[\frac{ \log(n)}{ n w_{ n}} \leq  \frac{ \log(n) \kappa_{ n}^{ 2}w_{ n}}{ n} \to 0.\] Since $B(u)\to \frac{ 1}{ 2}$ as $u\to 0$, choose a deterministic $p\geq1$ such that for all $n\geq p$, $B \left( 4 \sqrt{ 2}\left(\frac{\log(n)}{ n w_{ n}}\right)^{ 1/2} \right)\geq \frac{ 1}{ 4}$. For such an $n$,
\begin{align*}
\mathbb{ P} \left( \sup_{ i, k\in \left[n\right]} \left\vert X_{ i, k, t}^{ (n)} \right\vert > \varepsilon_{ n}\right) &\leq 2n^{ 2}\exp \left(-4\log(n)\right)= \frac{ 2}{ n^{ 2}}.
\end{align*}
Hence, by Borel-Cantelli Lemma, there exists $ \mathcal{ O}_{ 1} \in \mathcal{ F}$ with $ \mathbb{ P}( \mathcal{ O}_{ 1})=1$ such that, on $ \mathcal{ O}_{ 1}$, there exists $n_{ 1}<+\infty$ such that for all $n\geq n_{ 1}$,
\begin{equation}
\label{eq:bound_Xik}
\sup_{ i\in \left[n\right]}\sup_{ k\in \left[n\right]} \left\vert X_{ i, k, t}^{ (n)} \right\vert \leq \varepsilon_{ n}.
\end{equation}
Secondly, apply once again Lemma~\ref{lem:large_deviations} for the choice of $ U_{ l}= \xi_{ i, l}$, $ \kappa_{ n}= \kappa_{i}^{ (n)}$, $p_{ l}= W_{ n}(x_{ i}, x_{ l})$ and $v_{ l}\equiv1$. The same reasoning as above gives for $n\geq p$,
\begin{align*}
\mathbb{ P} \left( \sup_{ i\in \left[n\right]}\left\vert \frac{ \kappa_{i}^{ (n)}}{ n} \sum_{ k=1}^{ n} \bar\xi_{ i, k} \right\vert > \varepsilon_{ n}\right)&\leq 2n\exp \left(-4\log(n)\right)= \frac{ 2}{ n^{ 3}},
\end{align*}
so that, there exists $ \mathcal{ O}_{ 2}\in \mathcal{ F}$ such that $ \mathbb{ P}( \mathcal{ O}_{ 2})=1$, such that on $ \mathcal{ O}_{ 2}$, there exists $n_{ 2}<+\infty$, such that for all $n\geq n_{ 2}$,
\begin{equation}
\label{eq:bound_xiik}
\sup_{ i\in \left[n\right]} \left( \frac{ \kappa_{i}^{ (n)}}{ n} \sum_{ k=1}^{ n} \xi_{ i, k}\right) \leq\varepsilon_{ n} + \sup_{ i\in \left[n\right]} \frac{ \kappa_{i}^{ (n)}}{ n} \sum_{ k=1}^{ n} W_{ n}(x_{ i}, x_{ k}) \leq 1 + \sup_{ i\in \left[n\right]} \frac{ \kappa_{i}^{ (n)}}{ n} \sum_{ k=1}^{ n} W_{ n}(x_{ i}, x_{ k}).
\end{equation}
On the event $ \mathcal{ O}:= \mathcal{ O}_{ 1}\cap \mathcal{ O}_{ 2}$ (of probability $1$), the inequality \eqref{eq:bound_An4} together with \eqref{eq:bound_Xik} and \eqref{eq:bound_xiik} gives, for $n\geq \max(n_{ 1}, n_{ 2})$
\begin{equation}
d_{ n, T}\leq C\varepsilon_{ n}(T+ e^{ CT}) \left( 1 + 2 \sup_{ i\in \left[n\right]} \frac{ \kappa_{i}^{ (n)}}{ n} \sum_{ k=1}^{ n} W_{ n}(x_{ i}, x_{ k})\right).
\end{equation}
Using now the fact that $ \delta_{ n}( \underline{ x})\to 0$ as $n\to\infty$, $\sup_{ i\in \left[n\right]} \frac{ \kappa_{i}^{ (n)}}{ n} \sum_{ k=1}^{ n} W_{ n}(x_{ i}, x_{ k})$ is smaller than $1+\sup_{ i\in \left[n\right]} \frac{ 1}{ n} \sum_{ k=1}^{ n} W(x_{ i}, x_{ k})$ at least for large $n$. This concludes the proof of Proposition~\ref{prop:control_Delta_2}.
\end{proof}
\subsection{Proof of Theorem~\ref{theo:conv_empirical_measure}}
\label{sec:proof_conv_empirical_measure}
Recall that $\ell_{ n}$ below is the empirical measure of the positions \eqref{eq:emp_measure_positions}.
We first show that the convergence \eqref{hyp:conv_empirical_measure_positions} is true for both deterministic positions (Assumption~\ref{ass:deterministic_positions}) and random positions (Assumption~\ref{ass:random_positions}):
\begin{lemma}
\label{lem:conv_elln_Halpha}
In the deterministic case (Assumption~\ref{ass:deterministic_positions}) (resp. in the random case, Assumption~\ref{ass:random_positions}), the convergence \eqref{hyp:conv_empirical_measure_positions} holds (resp. almost surely).
\end{lemma}
\begin{proof}[Proof of Lemma~\ref{lem:conv_elln_Halpha}]
 In the deterministic case, it is immediate to see that for any $ \varphi\in \mathcal{ H}_{ \iota}$ with $ \left\Vert \varphi \right\Vert_{ \mathcal{ H}_{ \iota}}\leq 1$, $\left\vert \left\langle \ell_{ n}- \ell\, ,\, \varphi\right\rangle \right\vert \leq \frac{ 1}{ (1+\iota) n^{ \iota}}$ so that \eqref{hyp:conv_empirical_measure_positions} follows directly. We now focus on the random case. First note that the following holds: for all $ \varepsilon>0$, there exists a finite set $ \mathcal{ H}_{ \iota}^{ \varepsilon} \subset \mathcal{ H}_{ \iota}$ such that, for all $n\geq1$, $ d_{ \mathcal{ H}_{ \iota}}(\ell_{ n}, \ell) \leq \varepsilon + \max_{ \varphi\in \mathcal{ H}_{ \iota}^{ \varepsilon}} \left\vert \left\langle \ell_{ n}- \ell\, ,\, \varphi\right\rangle \right\vert$. The proof of this point follows closely the proof of \cite{MR1932358}, Th.11.3.3: for any $ \varepsilon>0$, take $K\subset I$ compact such that $ \ell(K)> 1- \varepsilon$. The set of functions $B:= \left\lbrace \varphi,\ \left\Vert \varphi \right\Vert_{ \mathcal{ H}_{ \iota}}\leq1\right\rbrace$, restricted to $K$, is compact, by Ascoli-Arzel\`a theorem. Thus there exists $k\geq1$ and $ \varphi_{ 1}, \ldots, \varphi_{ k}\in B$ such that for all $ \varphi\in B$, there exists $j\leq k$ such that $ \sup_{ y\in K} \left\vert \varphi(y) - \varphi_{ j}(y) \right\vert< \varepsilon$, so that $\sup_{ x\in K^{ \varepsilon}} \left\vert \varphi(x) - \varphi_{ j}(x)\right\vert < 3 \varepsilon$. Set $g(x) := \max \left(0, 1- d(x, K)/ \varepsilon\right)$. Then $g\in \mathcal{ H}_{ \iota}$ and for $n$ large enough, $\ell_{ n}(K^{ \varepsilon})\geq \int g {\rm d}\ell_{ n}> 1- 2 \varepsilon$. Thus
 \begin{align*}
 \left\vert \left\langle \ell_{ n} - \ell \, ,\, \varphi\right\rangle \right\vert &\leq  \int \left\vert \varphi - \varphi_{ j} \right\vert {\rm d}(\ell_{ n}+ \ell) + \left\vert \left\langle \ell_{ n} - \ell\, ,\, \varphi_{ j}\right\rangle \right\vert\leq 12 \varepsilon + \max_{ j=1,\ldots, k} \left\vert \left\langle \ell_{ n} - \ell\, ,\, \varphi_{ j}\right\rangle \right\vert.
 \end{align*} 
 Thus, the result follows from the fact that almost surely, for every bounded continuous function $ \varphi$ (and hence, for all $ \varphi\in \mathcal{ H}_{ \iota}$) we have $ \left\vert \left\langle \ell_{ n}- \ell\, ,\, \varphi\right\rangle \right\vert \xrightarrow[ n\to\infty]{}0$ (\cite{MR1932358}, Th.~11.4.1).
\end{proof}
\begin{proof}[Proof of Theorem~\ref{theo:conv_empirical_measure}]
For $t\in[0, T]$ and  $ \varphi: \mathbb{ R}^{ d} \times I \to \mathbb{ R}$ satisfying, for $ \theta, \tilde{ \theta}\in \mathbb{ R}^{ d}$, $x, y\in I$, $ \left\vert \varphi(\theta, x) - \varphi(\tilde{ \theta}, y)\right\vert \leq C_{ \varphi} \left(\left\vert \theta- \tilde{ \theta} \right\vert+ \left\vert x-y \right\vert\right)$ and $\sup_{ x\in I} \left\vert \varphi(\theta, x) \right\vert \leq C_{ \varphi} (1+ \left\vert \theta \right\vert)$, we have
\begin{align*}
\mathbf{ E} \left[\left\vert \left\langle \nu_{ n, t} - \nu_{ t}\, ,\, \varphi\right\rangle \right\vert^{ 2}\right] &\leq 2 \mathbf{ E} \left[\left\vert \left\langle \nu_{ n, t} - \bar\nu_{ n, t}\, ,\, \varphi\right\rangle \right\vert^{ 2}\right] + 2 \mathbf{ E} \left[\left\vert \left\langle \bar\nu_{ n, t} - \nu_{ t}\, ,\, \varphi\right\rangle \right\vert^{ 2}\right].
\end{align*} 
Introducing the empirical measure of the nonlinear processes $(\bar \theta_{ i})_{ i\in \left[n\right]}$ defined in Section~\ref{sec:general_prop_chaos}:
\begin{equation}
\label{eq:emp_measure_bar}
\bar \nu_{ n, t} = \frac{ 1}{ n}\sum_{ i=1}^{ n} \delta_{ (\bar\theta_{ i, t}, x_{ i}^{ (n)})},\ t\geq0,
\end{equation}
we can estimate the first term above as
\begin{align*}
\mathbf{ E} \left[\left\vert \left\langle \nu_{ n, t} - \bar\nu_{ n, t}\, ,\, \varphi\right\rangle \right\vert^{ 2}\right]&\leq\frac{ 1}{ n} \sum_{ i=1}^{ n} \mathbf{ E} \left[\left\vert \varphi(\theta_{ i, t}, x_{ i}) - \varphi(\bar \theta_{ i, t}, x_{ i})\right\vert^{ 2} \right] \leq \frac{ C}{ n} \sum_{ i=1}^{ n} \mathbf{ E} \left[\left\vert \theta_{ i, t} - \bar \theta_{ i, t}\right\vert^{ 2} \right],\\
&\leq C \sup_{ i\in \left[n\right]}\mathbf{ E} \left[ \sup_{ s\in [0, T]}\left\vert \theta_{ i, s} - \bar \theta_{ i, s}\right\vert^{ 2} \right] \xrightarrow[ n\to\infty]{}0,
\end{align*}
by Theorem~\ref{theo:conv_general}. Concerning the second term
\begin{align}
\mathbf{ E} \left[\left\vert \left\langle \bar\nu_{ n, t} - \nu_{ t}\, ,\, \varphi\right\rangle \right\vert^{ 2}\right]
%&= \mathbf{ E} \left[\left\vert \frac{ 1}{ n} \sum_{ i=1}^{ n} \varphi(\bar \theta_{ i, t}, x_{ i}) - \int \varphi(\theta, x) \nu_{ t}^{ x}({\rm d}\theta) {\rm d}x \right\vert^{ 2}\right], \nonumber\\
&\leq 2\mathbf{ E} \left[\left\vert \frac{ 1}{ n} \sum_{ i=1}^{ n} \left(\varphi(\bar \theta_{ i, t}, x_{ i}) - \int \varphi(\theta, x_{ i}) \nu_{ t}^{ x_{ i}}({\rm d}\theta) \right) \right\vert^{ 2}\right] \label{aux:nunbar1}\\ 
&+ 2\left\vert \frac{ 1}{ n} \sum_{ i=1}^{ n}\int \varphi(\theta, x_{ i}) \nu_{ t}^{ x_{ i}}({\rm d}\theta) - \int \varphi(\theta, x) \nu_{ t}^{ x}({\rm d}\theta) {\rm d}x \right\vert^{ 2}.\label{aux:nunbar2}
\end{align}
The first term \eqref{aux:nunbar1} above can be computed as
\begin{multline*}
\mathbf{ E} \left[\left\vert \frac{ 1}{ n} \sum_{ i=1}^{ n} \left(\varphi(\bar \theta_{ i, t}, x_{ i}) - \int \varphi(\theta, x_{ i}) \nu_{ t}^{ x_{ i}}({\rm d}\theta) \right) \right\vert^{ 2}\right]= \frac{ 1}{ n^{ 2}} \sum_{ i=1}^{ n}\mathbf{ E} \left[ \left(\varphi(\bar \theta_{ i, t}, x_{ i}) - \int \varphi(\theta, x_{ i}) \nu_{ t}^{ x_{ i}}({\rm d}\theta) \right)^{ 2}\right],\\
\leq \frac{ 1}{ n^{ 2}} \sum_{ i=1}^{ n} \int \varphi(\theta, x_{ i})^{ 2} \nu_{ t}^{ x_{ i}}({\rm d}\theta)\leq \frac{ C_{ \varphi}^{ 2}}{ n^{ 2}} \sum_{ i=1}^{ n} \int \left(1+ \left\vert \theta \right\vert\right)^{ 2} \nu_{ t}^{ x_{ i}}({\rm d}\theta),
\end{multline*}
which goes to $0$ as $n\to\infty$, uniformly in $t\in [0, T]$, by \eqref{eq:apriori_bound_nonlin}. For the last term \eqref{aux:nunbar2}, write for simplicity $ u_{ t}(x):= \int \varphi(\theta, x) \nu_{ t}^{ x}({\rm d}\theta)$. Note that $ \left\vert u_{ t}(x) \right\vert \leq C_{ \varphi}\int \left(1+ \left\vert \theta \right\vert\right) \nu_{ t}^{ x}({\rm d}\theta) $ which is bounded uniformly in $t\in [0,T]$ and $x\in I$, by \eqref{eq:apriori_bound_nonlin}. Moreover, (recall the definition of $ \delta \mathcal{ W}$ in \eqref{hyp:Int_W_Lip}), for some constant $C(\varphi, T)>0$ sufficiently large that changes from one line to the other,
\begin{align*}
\left\vert u_{ t}(x) - u_{ t}(y) \right\vert &\leq \int \left\vert \varphi(\theta, x) - \varphi(\theta, y) \right\vert \nu_{ t}^{ x}({\rm d}\theta) +  \left\vert \int\varphi(\theta, y) \nu_{ t}^{ x}({\rm d}\theta)-\int\varphi(\theta, y) \nu_{ t}^{ y}({\rm d}\theta)\right\vert,\\
&\leq C_{ \varphi} \left\vert x-y \right\vert^{ \iota} + C(\varphi, T)\left( \delta \mathcal{ W}(x, y)+ w_{ 1}(\nu_{ 0}^{x}, \nu_{ 0}^{y})\right)\leq C(\varphi, T) \left\vert x-y \right\vert^{ \iota},
\end{align*}
where we used \eqref{eq:regularity_lipschitz_nu_x}, \eqref{hyp:nu_0_Lip} and \eqref{hyp:Int_W_Lip}. Thus, for all $t\in[0, T]$, $ \left\Vert \frac{ u_{ t}}{ C(\varphi, T)} \right\Vert_{ \mathcal{ H}_{ \iota}}\leq 1$. Hence, the term \eqref{aux:nunbar2} can be estimated as
\begin{align*}
\left\vert \frac{ 1}{ n} \sum_{ i=1}^{ n} u_{ t}(x_{ i}) - \int u_{ t}(x) {\rm d}x \right\vert^{ 2}&= \left\vert \left\langle \ell_{ n}- \ell\, ,\, u_{ t}\right\rangle \right\vert^{ 2}\leq C( \varphi, T)^{ 2} d_{ \mathcal{ H}}(\ell_{ n}, \ell)^{ 2}
\end{align*}
which goes to $0$ as $n\to\infty$ by \eqref{hyp:conv_empirical_measure_positions}.
\end{proof}

\subsection{Proof of Proposition~\ref{prop:control_det_Delta3}: regularity of the kernel in the deterministic case}
\label{sec:proof_regularity_graphon_det}
We suppose here that Assumption~\ref{ass:deterministic_positions} holds. Since for all $n\geq1$,
\begin{align*}
\left\vert \frac{ 1}{ n}\sum_{ k=1}^{ n} W(x_{ i}, x_{ k}) - \int W(x_{ i}, y) {\rm d}y \right\vert &\leq \sum_{ k=1}^{ n} \int_{ x_{ k-1}}^{ x_{ k}}  \left\vert W(x_{ i}, x_{ k})  - W(x_{ i}, y) \right\vert  {\rm d} y = s_{ n}(W),
\end{align*}
this inequality together with \eqref{hyp:delta_n_W} and \eqref{hyp:bound_W_L1} implies \eqref{hyp:bound_W_L1_sum}. So we are left with proving \eqref{hyp:Delta3_to_0}. For $i\in \left[n\right]$, we have
\begin{align}
\left\vert \epsilon_{ n, T}^{ (1, i)} \right\vert&\leq \sum_{ k,l=1}^{ n}\int_{ x_{ l-1}}^{x_{ l}} \int_{ x_{ k-1}}^{x_{ k}} \left\vert W(x_{ i}, x_{ k})W(x_{ i}, x_{ l}) \Upsilon_{ T}(x_{ i}, x_{ k}, x_{ l}) -  W(x_{ i}, y) W(x_{ i}, z) \Upsilon_{ T}(x_{ i}, y, z)\right\vert {\rm d} y{\rm d} z \nonumber\\
&\leq \sum_{ k,l=1}^{ n} \int_{ x_{ l-1}}^{x_{ l}}\int_{ x_{ k-1}}^{x_{ k}}  \left\vert W(x_{ i}, x_{ k})W(x_{ i}, x_{ l}) -  W(x_{ i}, y) W(x_{ i}, z) \right\vert \left\vert \Upsilon_{ T}(x_{ i}, x_{ k}, x_{ l}) \right\vert  {\rm d} y{\rm d} z \nonumber\\
&+\sum_{ k,l=1}^{ n} \int_{ x_{ l-1}}^{x_{ l}}\int_{ x_{ k-1}}^{x_{ k}} \left\vert  \Upsilon_{ T}(x_{ i}, x_{ k}, x_{ l})-\Upsilon_{ T}(x_{ i}, y, z) \right\vert W(x_{ i}, y)W(x_{ i}, z) {\rm d} y{\rm d} z := (I) + (II).\label{eq:decomp_eps_1_III}
\end{align}
Concerning the first term above, we have, by Lemma~\ref{lem:regul_Upsilon}:
\begin{align}
(I)&\leq C\sum_{ k,l=1}^{ n} \int_{ x_{ l-1}}^{x_{ l}} \int_{ x_{ k-1}}^{x_{ k}} \left\vert W(x_{ i}, x_{ k})W(x_{ i}, x_{ l}) -  W(x_{ i}, y) W(x_{ i}, z) \right\vert {\rm d} y{\rm d} z, \nonumber\\
&\leq \left(\frac{ 1}{ n}\sum_{ l=1}^{ n}W(x_{ i}, x_{ l})+ \int_{ 0}^{1}  W(x_{ i}, y) {\rm d} y\right) s_{ n}(W). \label{eq:I_eps_1}
\end{align}
By \eqref{hyp:bound_W_L1}, \eqref{hyp:bound_W_L1_sum} and \eqref{hyp:delta_n_W}, this last quantity converges to $0$ as $n\to\infty$, uniformly in $i\in \left[n\right]$. Concerning the second term in \eqref{eq:decomp_eps_1_III}: 
\begin{align*}
(II)& \leq \sum_{ k,l=1}^{ n} \int_{ x_{ l-1}}^{x_{ l}}\int_{ x_{ k-1}}^{x_{ k}} \left\vert  \Upsilon_{ T}(x_{ i}, x_{ k}, x_{ l})-\Upsilon_{ T}(x_{ i}, y, x_{ l}) \right\vert W(x_{ i}, y)W(x_{ i}, z) {\rm d} y{\rm d} z,\\
&+ \sum_{ k,l=1}^{ n} \int_{ x_{ l-1}}^{x_{ l}}\int_{ x_{ k-1}}^{x_{ k}} \left\vert  \Upsilon_{ T}(x_{ i}, y, x_{ l})-\Upsilon_{ T}(x_{ i}, y, z) \right\vert W(x_{ i}, y)W(x_{ i}, z) {\rm d} y{\rm d} z.
\end{align*}
Using Lemma~\ref{lem:regul_Upsilon} again, we have (recall the definition of $ \delta \mathcal{ W}$ in \eqref{hyp:Int_W_Lip})
\begin{align*}
(II)&\leq C \sum_{ k,l=1}^{ n} \int_{ x_{ l-1}}^{x_{ l}}\int_{ x_{ k-1}}^{x_{ k}}  \left( \delta \mathcal{ W}(x_{ k}, y) + w_{ 1} (\nu_{ 0}^{ x_{ k}}, \nu_{ 0}^{ y})\right) W(x_{ i}, y)W(x_{ i}, z) {\rm d} y{\rm d} z,\\
&+  C \sum_{ k,l=1}^{ n} \int_{ x_{ l-1}}^{x_{ l}}\int_{ x_{ k-1}}^{x_{ k}} \left( \delta \mathcal{ W}(x_{ l}, z) + w_{ 1}(\nu_{ 0}^{ x_{ l}}, \nu_{ 0}^{ z})\right) W(x_{ i}, y)W(x_{ i}, z) {\rm d} y{\rm d} z.
\end{align*}
Using now \eqref{hyp:Int_W_Lip} and \eqref{hyp:nu_0_Lip}, we obtain for $L=\max(L_{ 0}, L_{ W})$
\begin{align}
(II)& \leq 2C L\left( \frac{ 1}{ n^{ \iota_{ 2}}} + \frac{ 1}{ n^{ \iota_{ 1}}}\right) \int W(x_{ i}, y)W(x_{ i}, z) {\rm d} y{\rm d} z\leq 2C L \left\Vert \mathcal{ W} \right\Vert_{ \infty}^{ 2} \left( \frac{ 1}{ n^{ \iota_{ 2}}} + \frac{ 1}{ n^{ \iota_{ 1}}}\right) \label{eq:II_eps_1}
\end{align}
Taking $\sup_{ i\in \left[n\right]}$ in \eqref{eq:decomp_eps_1_III}, we conclude by \eqref{eq:I_eps_1} and \eqref{eq:II_eps_1} that $ \epsilon_{ n, T}^{ (1)}( \underline{ x}) \xrightarrow[ n\to \infty]{}0$. The two other terms $\epsilon_{ n, T}^{ (2)}( \underline{ x})$ and $\epsilon_{ n, T}^{ (3)}( \underline{ x})$ can be dealt in a similar way, we leave the proof to the reader. This proves Proposition~\ref{prop:control_det_Delta3}.

\subsection{Proof of Proposition~\ref{prop:control_An6}: regularity of the kernel in the random case}
\label{sec:proof_graphon_regularity_rand}
We assume here that Assumption~\ref{ass:random_positions} holds. For simplicity of notations, we will write $ \epsilon_{ n}^{ (m, i)}$ in place of $ \epsilon_{ n, T}^{ (m, i)}( \underline{ x})$ and $ \epsilon_{ n}^{ (m)}$ in place of $ \epsilon_{ n, T}^{ (m)}( \underline{ x})$ in \eqref{hyp:Delta3_to_0}.
Introduce the following truncation:
\begin{equation}
\label{hyp:WM}
W_{ M}: (x, y) \mapsto W(x, y) \wedge M,\ M>0
\end{equation}
and define $ \epsilon_{ n, M}^{ (m, i)}$ (resp. $ \epsilon_{ n, M}^{ (m)}$) as the truncated version  of $ \epsilon_{n}^{ (m, i)}$ (resp. $ \epsilon_{n}^{ (m)}$), that is, when $W$ is replaced by $W_{ M}$.

\textbf{ Claim 1:} there exists a constant $C>0$ such that for any $M>0$, $i\in \left[n\right]$, $m=1,2,3$,
\begin{multline}
\label{eq:An6_An6M_qu}
 \left\vert \epsilon_{ n}^{ (m, i)} - \epsilon_{ n, M}^{ (m, i)}\right\vert  \leq C\left(\frac{ 1}{ n}\sum_{l=1}^{ n} W(x_{ i}, x_{ l}) +1\right) \\ \Bigg(\frac{ 1}{ n}\sum_{ k=1}^{ n}\left\vert W(x_{ i}, x_{ k})  -   W_{ M}(x_{ i}, x_{ k})\right\vert  + \int  \left\vert W_{ M}(x_{ i}, y) - W(x_{ i}, y) \right\vert \ell({\rm d} y)\Bigg).
\end{multline}
To prove Claim~1, we only consider $m=1$ and leave the other cases to the reader. Fix $M>0$ and $i\in \left[n\right]$. Then, using Lemma~\ref{lem:regul_Upsilon},
\begin{align*}
\Delta_{ n, i}^{ (1)}&:= \left\vert \frac{ 1}{ n^{ 2}} \sum_{ k,l=1}^{ n} \left\lbrace W(x_{ i}, x_{ k})W(x_{ i}, x_{ l}) - W_{ M}(x_{ i}, x_{ k})W_{ M}(x_{ i}, x_{ l})\right\rbrace \Upsilon_{ T}(x_{ i}, x_{ k}, x_{ l}) \right\vert,\\
&\leq \frac{ C}{ n^{ 2}} \sum_{ k,l=1}^{ n} \left\vert W(x_{ i}, x_{ k})W(x_{ i}, x_{ l}) - W_{ M}(x_{ i}, x_{ k})W_{ M}(x_{ i}, x_{ l})\right\vert,\\
&\leq C \left(\frac{ 1}{ n}  \sum_{ k=1}^{ n}W(x_{ i}, x_{ k}) + \frac{ 1}{ n} \sum_{ k=1}^{ n} W_{ M}(x_{ i}, x_{ k})\right) \left(\frac{ 1}{ n}\sum_{ k=1}^{ n}\left\vert W(x_{ i}, x_{ k})- W_{ M}(x_{ i}, x_{ k})\right\vert\right),\\
&\leq 2C \left(\frac{ 1}{ n}  \sum_{ k=1}^{ n}W(x_{ i}, x_{ k})\right) \left(\frac{ 1}{ n}\sum_{ k=1}^{ n}\left\vert W(x_{ i}, x_{ k})- W_{ M}(x_{ i}, x_{ k})\right\vert\right),
\end{align*}
and, in a similar way (recall \eqref{hyp:bound_W_L1})
\begin{align*}
\Delta_{ n, i}^{ (2)}&:= \left\vert \int  \left\lbrace W(x_{ i}, y) W(x_{ i}, z)  - W_{ M}(x_{ i}, y) W_{ M}(x_{ i}, z)\right\rbrace \Upsilon_{ T}(x_{ i}, y, z) {\rm d} y{\rm d} z \right\vert,\\
&\leq C\int  \left\vert W(x_{ i}, y) W(x_{ i}, z)  - W_{ M}(x_{ i}, y) W_{ M}(x_{ i}, z)\right\vert \ell({\rm d} y)\ell({\rm d} z),\\
&\leq 2C \left\Vert \mathcal{ W}_{ 1} \right\Vert_{ \infty} \int\left\vert W(x_{ i}, y) - W_{ M}(x_{ i}, y) \right\vert \ell({\rm d} y).
\end{align*}
Doing the same for $m=2, 3$, this proves Claim~1. The next point is now to show that the quantities in \eqref{eq:An6_An6M_qu} can be almost surely controlled for large $n$, choosing carefully the truncation parameter $M$.

\textbf{ Claim 2:}
 Let us fix parameters $ \delta_{ 1}, \delta_{ 2}>0$ (to be chosen later) and define
\begin{equation}
\label{hyp:M}
M:= n^{ \delta_{ 1}}.
\end{equation}
For this choice of $M$, we have (recall the hypothesis on the $L^{ 2}$-norm of $W$ \eqref{hyp:norm_W_chi}), for some constant $C>0$ independent of $n$,
\begin{align}
\mathbb{ P} \left( \sup_{ i\in \left[n\right]} \frac{ 1}{ n}\sum_{ k=1}^{ n}\left\vert W(x_{ i}, x_{ k})  -   W_{ M}(x_{ i}, x_{ k})\right\vert > n^{ - \delta_{ 2}}\right)
&\leq \frac{ 6 \left\Vert W \right\Vert_{ L^{  \chi}(I^{ 2})}^{ \chi}}{ n^{ \delta_{ 1} ( \chi-1)- \delta_{ 2} -1} },\label{cln:control_W_WM_sum}\\
\mathbb{ P} \left( \sup_{ i\in \left[n\right]} \int  \left\vert W_{ M}(x_{ i}, \tilde x) - W(x_{ i}, \tilde x) \right\vert {\rm d} \tilde{ x} > n^{ - \delta_{ 2}}\right)
&\leq \frac{ 6 \left\Vert W \right\Vert_{ L^{ \chi}(I^{ 2})}^{ \chi}}{ n^{ \delta_{ 1} (\chi-1)- \delta_{ 2} -1} },\label{cln:control_W_WM_int}\\
\mathbb{ P} \left( \sup_{ i\in \left[n\right]} \left\vert \frac{ 1}{ n} \sum_{ l=1}^{ n} W(x_{ i}, x_{ l})  \right\vert > 1+ \left\Vert \mathcal{ W}_{ 1} \right\Vert_{ \infty}\right) &\leq \frac{ C}{ n^{ 2}}.\label{cln:control_W_int}
\end{align}
Let us prove Claim 2: we have
\begin{align*}
\mathbb{ E} \left[\frac{ 1}{ n}\sum_{ k=1}^{ n}\left\vert W(x_{ i}, x_{ k})  -   W_{ M}(x_{ i}, x_{ k})\right\vert\right] &= \frac{ 1}{ n}\sum_{ k=1}^{ n}\mathbb{ E} \left[\left\vert W(x_{ i}, x_{ k})  -   M\right\vert \mathbf{ 1}_{ W(x_{ i}, x_{ k})>M}\right],\\
&\leq \frac{ 2}{ n} \sum_{ k=1}^{ n}\mathbb{ E} \left[W(x_{ i}, x_{ k}) \mathbf{ 1}_{ W(x_{ i}, x_{ k})>M}\right].
\end{align*}
Now, for any independent $X, Y$ on $I$ with law $\ell$,
\begin{align*}
\mathbb{ E} \left[W(X, Y) \mathbf{ 1}_{ W(X, Y)>M}\right] &=\sum_{ l=0}^{ +\infty}\mathbb{ E} \left[W(X, Y) \mathbf{ 1}_{ 2^{ l}M<W(X, Y)\leq 2^{ l+1}M}\right],\\
&\leq M\sum_{ l=0}^{ +\infty} 2^{ l+1}\left(\mathbb{ P} \left(W(X, Y)>2^{ l}M\right) - \mathbb{ P} \left(W(X, Y)>2^{ l+1}M\right)\right),\\
&= M \left(2 \mathbb{ P} \left(W(X, Y)>M\right) + \sum_{ l=1}^{ +\infty} 2^{ l}\mathbb{ P} \left(W(X, Y)>2^{ l}M\right) \right),\\
&\leq \frac{ \mathbb{ E} \left[W(X, Y)^{ \chi}\right]}{ M^{ \chi-1}} \left(2  + \sum_{ l=1}^{ +\infty} 2^{ -(\chi-1)l} \right)\leq \frac{ 3 \left\Vert W \right\Vert_{ L^{ \chi}(I^{ 2})}^{ \chi}}{ M^{ \chi-1}}.
\end{align*}
This gives, for $M$ given by \eqref{hyp:M}:
\begin{align*}
\mathbb{ E} \left[\frac{ 1}{ n}\sum_{ k=1}^{ n}\left\vert W(x_{ i}, x_{ k})  -   W_{ M}(x_{ i}, x_{ k})\right\vert\right] &\leq \frac{ 6 \left\Vert W \right\Vert_{ L^{ \chi}(I^{ 2})}^{ \chi}}{ n^{ \delta_{ 1} (\chi-1)}}.
\end{align*} Hence, \eqref{cln:control_W_WM_sum} follows immediately from Markov inequality and a union bound. In a same way, we have
\begin{align*}
\mathbb{ E} \left[  \int  \left\vert W_{ M}(x_{ i}, \tilde x) - W(x_{ i}, \tilde x) \right\vert \ell({\rm d} \tilde{ x})\right] &\leq  \frac{ 6 \left\Vert W \right\Vert_{ L^{ \chi}(I^{ 2})}^{ \chi}}{ n^{ \delta_{ 1} (\chi-1)}},
\end{align*}
so that inequality \eqref{cln:control_W_WM_int} holds. Inequality 
\begin{equation}
\mathbb{ P} \left( \sup_{ i\in \left[n\right]} \left\vert \frac{ 1}{ n} \sum_{ l=1}^{ n} W(x_{ i}, x_{ l}) - \int W(x_{ i}, z) \ell({\rm d} z) \right\vert > 1\right) \leq \frac{ C}{ n^{ 2}}
\end{equation} follows directly from Markov inequality and the fact that $ \chi>9 \geq 6$ in \eqref{hyp:norm_W_chi} and the independence of the variables $(x_{ k})_{ k\in \left[k\right]}$. Then  \eqref{cln:control_W_int} is a consequence of the inequality $\sup_{ i\in \left[n\right]}\int W(x_{ i}, z) {\rm d} z \leq \left\Vert \mathcal{ W}_{ 1} \right\Vert_{ \infty}<+\infty$. This proves Claim 2.

\textbf{ Claim 3:}
Let $ \delta_{ 3}>0$ be a last constant to be defined later. There exists a constant $C=C_{ T}>0$ such that, for the choice of $M= n^{ \delta_{ 1}}$ defined in \eqref{hyp:M}, 
\begin{equation}
\label{cln:control_An6M}
\mathbb{ P} \left( \epsilon_{ n, M}^{ (1)} + \epsilon_{ n, M}^{ (2)} +\epsilon_{ n, M}^{ (3)}  \geq C \left( \frac{ 1}{ n^{ \delta_{ 3} -2 \delta_{ 1}}} + \frac{ 1}{ n^{ 1- 2 \delta_{ 1}}}\right)\right) \leq Cn\exp \left(- \frac{n^{ 1- 2 \delta_{ 3}} }{ 8}\right).
\end{equation}
Let us prove Claim~3: we only control $ \epsilon_{ n, M}^{ (1)}$ and leave the two other terms to the reader. Since for fixed $x\in I$, $(y, z) \mapsto \Upsilon_{ T}(x, y, z)$ is symmetric, we have, for $i\in \left[n\right]$
\begin{align*}
\epsilon_{ n, M}^{ (1, i)}&:= \frac{ 1}{ n^{ 2}} \sum_{ k=1}^{ n} \left\lbrace W_{ M}(x_{ i}, x_{ k})^{ 2} \Upsilon_{ T}(x_{ i}, x_{ k}, x_{ k}) - \int W_{ M}(x_{ i}, y) W_{ M}(x_{ i}, z) \Upsilon_{ T}(x_{ i}, y, z) \ell({\rm d} y)\ell({\rm d} z)\right\rbrace\\
&+ \frac{ 2}{ n^{ 2}} \sum_{1\leq l< k\leq n} \left\lbrace W_{ M}(x_{ i}, x_{ k})W_{ M}(x_{ i}, x_{ l}) \Upsilon_{ T}(x_{ i}, x_{ k}, x_{ l}) - \int W_{ M}(x_{ i}, y) W_{ M}(x_{ i}, z) \Upsilon_{ T}(x_{ i}, y, z) \ell({\rm d} y)\ell({\rm d} z)\right\rbrace,\\
&=(I)+(II).
\end{align*}
Using Lemma~\ref{lem:regul_Upsilon}, the first term is easily bounded (almost surely for all $t$) by $\frac{ C M^{ 2}}{ n}=\frac{ C}{ n^{ 1- 2 \delta_{ 1}}}$. We now turn to the control of the second term:
\begin{align*}
(II)&= \frac{ 2}{ n^{ 2}}\sum_{1\leq l< k\leq n}^{ n} W_{ M}(x_{ i}, x_{ l}) \left\lbrace W_{ M}(x_{ i}, x_{ k}) \Upsilon_{ T}(x_{ i}, x_{ k}, x_{ l}) -  \int W_{ M}(x_{ i}, y) \Upsilon_{ T}(x_{ i}, y, x_{ l}) \ell({\rm d} y)\right\rbrace\\
&+\frac{ 2}{ n^{ 2}}\sum_{1\leq l< k\leq n}^{ n} \left\lbrace W_{ M}(x_{ i}, x_{ l}) \int W_{ M}(x_{ i}, y) \Upsilon_{ T}(x_{ i}, y, x_{ l})\ell({\rm d}y) - \int W_{ M}(x_{ i}, y) W_{ M}(x_{ i}, z) \Upsilon_{ T}(x_{ i}, y, z) \ell({\rm d} y)\ell({\rm d} z)\right\rbrace\\
&:= \zeta_{ n, M}^{ (1, i)}+ \zeta_{ n, M}^{ (2, i)}.
\end{align*}
We only make the calculations for the $ \zeta_{ n, M}^{ (1, i)}$ and leave the (easier) term $ \zeta_{ n, M}^{ (2, i)}$ to the reader. Denote by
\begin{equation*}
\pi^{ i, M}_{ k,l}=W_{ M}(x_{ i}, x_{ l}) \left\lbrace W_{ M}(x_{ i}, x_{ k}) \Upsilon_{ T}(x_{ i}, x_{ k}, x_{ l}) -  \int W_{ M}(x_{ i}, y) \Upsilon_{ T}(x_{ i}, y, x_{ l}) \ell({\rm d} y)\right\rbrace.
\end{equation*}
By definition of $W_{ M}$ in \eqref{hyp:WM} and using Lemma~\ref{lem:regul_Upsilon}, $\Pi_{ n}^{ M}:= \sup_{ k,l,i}\left\vert \pi^{ i, M}_{ k,l} \right\vert \leq C_{ T}M^{ 2}= C_{ T}n^{ 2 \delta_{ 1}}$.
Writing differently the summation in the definition of $ \zeta_{ n, M}^{ (1, i)}$ leads to
\begin{align}
\zeta_{ n, M}^{ (1, i)}&= \frac{ 2 \Pi_{ n}^{ M}}{ n} \sum_{ r=1}^{ n} Y_{ i, r}^{ (M)} = \frac{ 2 \Pi_{ n}^{ M}}{ n} Y_{ i, i}^{ (M)} + \frac{ 2 \Pi_{n}^{ M}}{ n} \sum_{r\in \left[n\right],\ r\neq i} Y_{ i, r}^{ (M)},\label{aux:An6_sum}
\end{align}
where we have defined $Y_{ i, r}^{ (M)} := \frac{ 1}{ n}\sum_{ p=1}^{ r-1} \frac{ \pi_{ r, p}^{ (i, M)}}{  \Pi_{ n}^{ M}},\ r\in \left[n\right]$.
The first term in \eqref{aux:An6_sum} is easily bounded (almost surely for all $t\leq T$) by $ \frac{ 2 \Pi_{ n}^{ M}}{ n}\leq \frac{ 2C}{ n^{ 1- 2 \delta_{ 1}}}$. We now turn to the second term of \eqref{aux:An6_sum}: for each $r\in \left[n\right]$, $Y_{ i, r}^{ (M)}$ is measurable w.r.t. the $ \sigma$-field $ \mathcal{ F}_{ r}^{ (i)}:= \sigma \left(x_{ i}, x_{ k};\ k=1, \ldots, r\right)$ and such that $ \left\vert Y_{ i, r}^{ (M)} \right\vert \leq 1$. For all $r=2, \ldots , n$, $r\neq i$, 
\begin{align*}
&\mathbb{ E} \left[Y_{ i, r}^{ (M)} \vert Y_{ i, r-1}^{ (M)}\right]= \mathbb{ E} \left[Y_{ i, r}^{ (M)} \vert \mathcal{ F}_{ r-1}^{ (i)}\right] = \frac{ 1}{ n \Pi_{n}^{ M}}\sum_{ p=1}^{ r-1} \mathbb{ E}  \left[\pi_{ r, p}^{ (i, M)}\Big\vert \mathcal{ F}_{ r-1}^{ (i)}\right],\\
&=\frac{ 1}{ n \Pi_{ n}^{ M}}\sum_{ p=1}^{ r-1} \mathbb{ E}  \left[ W_{ M}(x_{ i}, x_{ p}) \left\lbrace W_{ M}(x_{ i}, x_{ r}) \Upsilon_{ T}(x_{ i}, x_{ r}, x_{ p}) -  \int W_{ M}(x_{ i}, y) \Upsilon_{ T}(x_{ i}, y, x_{ p}) \ell({\rm d} y)\right\rbrace \Big\vert \mathcal{ F}_{ r-1}^{ (i)}\right],
\end{align*}
so that by independence of the $(x_{ k})_{ k\in \left[n\right]}$, $\mathbb{ E} \left[Y_{ i, r}^{ (M)} \vert Y_{ i, r-1}^{ (M)}\right]=0$. Note that this calculation only works for $r \neq i$ (this is why we have treated the term $Y_{ i, i}^{ (M)}$ apart in \eqref{aux:An6_sum}). Since $\left\vert Y_{ i, r}^{ (M)} \right\vert \leq 1$, we have obviously $ \mathbb{ E} \left[ \left(Y_{ i, r}^{ (M)}\right)^{ 2} \vert Y_{ i, r-1}^{ (M)}\right]\leq 1$ for all $r=2, \ldots, n$. We are now in position to apply Lemma~\ref{lem:concentration_dembo}: for all $ \varepsilon>0$,
\begin{multline*}
\mathbb{ P} \left( \frac{ 2\Pi_{n}^{ M} }{ n} \sum_{r\in \left[n\right],\ r\neq i} Y_{ i, r}^{ (M)} \geq \varepsilon\right) = \mathbb{ P} \left( \frac{1}{ n} \sum_{r\in \left[n\right],\ r\neq i} Y_{ i, r, t}^{ (M)} \geq \frac{ \varepsilon}{ 2 \Pi_{n}^{ M}}\right),\\
\leq \exp \left(-n H\left( \frac{ \frac{ \varepsilon}{ 2}  \left(\Pi_{n}^{ M}\right)^{ -1} +1}{ 2} \Bigg\vert \frac{ 1}{ 2}\right)\right)\leq \exp \left(- \frac{n \varepsilon^{ 2} }{ 8\left(\Pi_{n}^{ M}\right)^{ 2}}\right).
\end{multline*}
Doing the same for $-Y_{ i, r}$ and by a union bound, we obtain finally, for the choice $ \varepsilon:= n^{ - \delta_{ 3}} \Pi_{ n}^{ M}$,
\begin{align}
\mathbb{ P} \left( \sup_{ i\in \left[n\right]}\left\vert \frac{ 2 \Pi_{n}^{ M}}{ n} \sum_{r\in \left[n\right],\ r\neq i} Y_{ i, r}^{ (M)} \right\vert \geq n^{ - \delta_{ 3}}\Pi_{ n}^{ M}\right) &\leq 2n\exp \left(- \frac{n^{ 1- 2 \delta_{ 3}} }{ 8}\right).
\end{align} 
This proves Claim 3.

\textbf{ Conclusion:} let $ \varepsilon_{ 0}:= \chi-9>0$. Define
\begin{equation}
\delta_{ 1}:= \frac{ 1}{ 4} - \frac{ \varepsilon_{ 0}}{ 8( \varepsilon_{ 0}+6)}>0,\ \delta_{ 2}:= \frac{ \varepsilon_{ 0}}{ 2( \varepsilon_{ 0}+6)}>0 \text{ and }\delta_{ 3}:= \frac{ 1}{ 2} - \frac{ \varepsilon_{ 0}}{ 8( \varepsilon_{ 0}+6)}>0.
\end{equation}
For this choice of parameters, one has obviously $1- 2 \delta_{ 1}>0$, $1- 2 \delta_{ 3}>0$ and $\delta_{ 3}- 2 \delta_{ 1}>0$, and it is easy to verify that $\delta_{ 1} (\chi-1)- \delta_{ 2} -1 >1$.
This means that the probabilities in \eqref{cln:control_W_WM_sum}, \eqref{cln:control_W_WM_int}, \eqref{cln:control_An6M} are summable in $n$ and that both $ n^{ -(\delta_{ 3}- 2 \delta_{ 1})}$ and $ n^{ -(1- 2 \delta_{ 1})}$ go to $0$ as $n\to\infty$. From \eqref{eq:An6_An6M_qu}, \eqref{cln:control_W_WM_sum}, \eqref{cln:control_W_WM_int}, \eqref{cln:control_W_int}, \eqref{cln:control_An6M}, we deduce from Borel Cantelli Lemma that there exists an event of probability $1$, such that on this event,
\begin{equation}
\sum_{ m=1}^{ 3} \epsilon_{ n, T}^{ (m)}( \underline{ x})\leq C \left(n^{ - \delta_{ 2}} + n^{ -(\delta_{ 3} -2 \delta_{ 1})} + n^{ -(1- 2 \delta_{ 1})}\right)
\end{equation}
and 
\begin{equation}
\label{cln:control_sum_W}
\sup_{ i\in \left[n\right]} \frac{ 1}{ n} \sum_{ k=1}^{ n} W(x_{ i}, x_{ k}) \leq 1 + \left\Vert \mathcal{ W}_{ 1} \right\Vert_{ \infty}.
\end{equation}
This concludes the proof of Proposition~\ref{prop:control_An6}.
\section{The spatial field and the nonlinear heat equation}
\subsection{Uniqueness of a solution}
Here, we consider the general case where $(I, \ell)$ is given by Definition~\ref{def:I_ell}. We suppose here that the hypotheses of Proposition~\ref{prop:uniqueness_weak_solution} hold.
\begin{proof}[Proof of Proposition~\ref{prop:uniqueness_weak_solution}]
Let $ \varphi, \psi$ be two weak solutions in $ \mathcal{ C} \left([0, T], L^{ k}(I, \ell)\right)$ with the same initial condition and denote by $ \rho:= \varphi- \psi$ the difference. Since $ \varphi \in \mathcal{ C}([0, T], L^{ k}(I, \ell))$ (resp. $\psi$) is a weak solution, then $ \varphi$ (resp. $ \psi$) belongs to $H^{ 1}([0, T], L^{ k}(I, \ell))$ and we have, for all test function $J$, for almost every $t\in[0, T]$, 
\begin{equation}
\left\langle \frac{ {\rm d} \varphi}{ {\rm d}t }(\cdot, t)\, ,\, J\right\rangle_{ L^{ 2}(I, \ell)}= \left\langle c( \varphi(\cdot, t)) + \int_{I}\Gamma( \varphi(\cdot, t) , \varphi(y, t)) W(\cdot, y) {\rm d}y\, ,\,  J\right\rangle_{ L^{ 2}(I, \ell)}.
\end{equation}
By density, this also true for all test function $(t,x) \mapsto J(x,t)$ in $ L^{ 2}([0,T]\times I)$. Substracting the two equations for $ \varphi$ and $ \psi$ and choosing the test function $J(\cdot, t)= \varphi(\cdot, t)- \psi(\cdot, t)$, we obtain
\begin{align*}
\frac{ {\rm d}}{ {\rm d}t} &\left\Vert \rho(\cdot, t) \right\Vert_{ L^{ 2}(I, \ell)}^{ 2} = \left\langle c(\varphi(\cdot, t) - c(\psi, t)\, ,\, \rho(\cdot, t)\right\rangle_{ L^{ 2}(I, \ell)}\\ &+ \left\langle \int_{I} \left\lbrace\Gamma( \varphi(\cdot, t) , \varphi(y, t)) - \Gamma( \psi(\cdot, t) , \psi(y, t))\right\rbrace W(\cdot, y) \ell({\rm d}y)\, ,\, \rho(\cdot, t)\right\rangle_{ L^{ 2}(I, \ell)}.
\end{align*} 
The first term is bounded by $L_{ c} \left\Vert \rho(\cdot, t) \right\Vert_{ L^{ 2}(I, \ell)}^{ 2}$. The second can be evaluated as, by the properties of $ \Gamma$
\begin{multline*}
\left\vert \left\langle \int_{I} \left\lbrace\Gamma( \varphi(\cdot, t) , \varphi(y, t)) - \Gamma( \psi(\cdot, t) , \psi(y, t))\right\rbrace W(\cdot, y) \ell({\rm d}y)\, ,\, \rho(\cdot, t)\right\rangle_{ L^{ 2}(I, \ell)} \right\vert\\ 
%\leq \int_{ I} \int_{I} \left\vert \Gamma( \varphi(x, t) , \varphi(y, t)) - \Gamma( \psi(x, t) , \psi(y, t)) \right\vert \left\vert \rho(x, t) \right\vert W(x, y) \ell({\rm d}y)  \ell({\rm d}x),\\
%\leq L_{ \Gamma}\int_{ I} \int_{I} \left( \left\vert \rho(x, t) \right\vert + \left\vert \rho(y, t) \right\vert\right) \left\vert \rho(x, t) \right\vert W(x, y) \ell({\rm d}y)  \ell({\rm d}x),\\
\leq L_{ \Gamma}\int_{ I} \int_{I} \left\vert \rho(x, t) \right\vert^{ 2} W(x, y) \ell({\rm d}y)  \ell({\rm d}x) + L_{ \Gamma}\int_{ I}\left\vert \rho(x, t) \right\vert \int_{I} \left\vert \rho(y, t) \right\vert W(x, y) \ell({\rm d}y)  \ell({\rm d}x),
\end{multline*}
The first term is bounded by $ L_{ \Gamma} \left\Vert \mathcal{ W}_{ 1} \right\Vert_{ \infty} \left\Vert \rho(\cdot, t) \right\Vert_{ L^{ 2}(I, \ell)}^{ 2}$. Concerning the second, by Cauchy-Schwarz inequality,
\begin{multline*}
\int_{ I}\left\vert \rho(x, t) \right\vert \int_{I} \left\vert \rho(y, t) \right\vert W(x, y) \ell({\rm d}y)  \ell({\rm d}x)\leq\int_{ I}\left\vert \rho(x, t) \right\vert \left(\int_{I} \left\vert \rho(y, t) \right\vert^{ 2} \ell({\rm d}y)\right)^{ \frac{ 1}{ 2}}\left(\int_{I} W(x, y)^{ 2} \ell({\rm d}y)\right)^{ \frac{ 1}{ 2}}\ell({\rm d}x)\\
\leq \left\Vert \mathcal{ W}_{ 2} \right\Vert_{ \infty}^{ \frac{ 1}{ 2}}\left\Vert \rho(\cdot, t) \right\Vert_{ L^{ 2}(I, \ell)}\left\Vert \rho(\cdot, t) \right\Vert_{ L^{ 1}(I, \ell)} \leq \left\Vert \mathcal{ W}_{ 2} \right\Vert_{ \infty}^{ \frac{ 1}{ 2}}\left\Vert \rho(\cdot, t) \right\Vert_{ L^{ 2}(I, \ell)}^{ 2}.
\end{multline*}
%\begin{align*}
%\int_{ I} \int_{I} \rho(y, t) W(x, y) \ell({\rm d}y)\rho(x, t) \ell({\rm d}x)
%\end{align*}
Using that $ \rho(\cdot, 0)\equiv 0$, we obtain uniqueness by a Gr\"onwall's Lemma.
\end{proof}

\subsection{Convergence of the spatial profile}
\label{sec:proof_conv_spatial_profile}
The point of this paragraph is to prove Theorem~\ref{theo:conv_profile}. Recall that $I=[0,1]$ with regular deterministic positions (Assumption~\ref{ass:deterministic_positions}) and we suppose that the hypotheses of Theorem~\ref{theo:conv_profile} hold. We use the shortcut $ \left\lbrace x\right\rbrace_{ n}:= \left\lfloor nx+1\right\rfloor$ for any $x\in I$. Recall the definition of the spatial field $ \theta_{ n}(x, t) = \theta_{ \left\lbrace x\right\rbrace_{ n}, t}^{ (n)}$ in \eqref{eq:spatial_profile}. Introduce the coupling (where the initial conditions and Brownian motions are the same as for \eqref{eq:odegene}):
\begin{align}
\label{eq:psi_particle_system}
{\rm d}\psi_{i, t}^{ (n)}&= c(\psi_{i, t}^{ (n)}){\rm d} t + \frac{1}{n}\sum_{j=1}^{ n} W(x_{ i}, x_{ j})\Gamma\left(\psi_{i, t}^{ (n)}, \psi_{j, t}^{ (n)}\right) {\rm d} t + \sigma{\rm d} B_{i, t},\ 0\leq t\leq T,\ i\in \left[n\right],
\end{align}
as well as its corresponding spatial field (recall that $x_{ 0}=0$ by definition): 
\begin{align}
\label{eq:psi_spatial_profile}
\psi_{ n}(x, t)&:=\sum_{ i=1}^{ n} \psi_{i, t}^{ (n)}\textbf{ 1}_{ [x_{ i-1}, x_{ i})}(x) = \psi_{ \left\lbrace x\right\rbrace_{ n}, t}^{ (n)}, \ x\in I.
\end{align}
With this notations at hand, we directly see that $(\theta_{ n}(x, t), \psi_{ n}(x, t))$ are solutions to
\begin{align}
{\rm d}\theta_{n}(x, t)&= c( \theta_{n}(x, t)) {\rm d}t + \int_{ I} \Gamma( \theta_{n}(x, t) , \theta_{n}(y, t)) \Xi^{ (n)}_{x, y}{\rm d}y{\rm d}t + \sigma {\rm d}B_{ \left\lbrace x\right\rbrace_{ n}, t},\\
{\rm d}\psi_{n}(x, t)&= c( \psi_{n}(x, t)) {\rm d}t + \int_{ I} \Gamma( \psi_{n}(x, t) , \psi_{n}(y, t)) \hat W_{ n}(x, y) {\rm d}y{\rm d}t + \sigma {\rm d}B_{ \left\lbrace x\right\rbrace_{ n}, t}.
\end{align}
where
\begin{equation}
\Xi^{ (n)}_{x, y}:= \kappa_{\left\lbrace x\right\rbrace_{ n}}^{ (n)} \xi_{ \left\lbrace x\right\rbrace_{ n}, \left\lbrace y\right\rbrace_{ n}}^{ (n)}\text{ and }\hat W_{ n}(x, y):= W( \left\lbrace x\right\rbrace_{ n}, \left\lbrace y\right\rbrace_{ n}).
\end{equation}
\begin{proposition}
Under the hypotheses of Section~\ref{sec:convergence_profile_intro}, the process $ \left(\theta_{ n}(\cdot, t)\right)_{ n\geq1, t\in[0, T]}$ is tight in $\mathcal{ C}([0, T], L^{ k}([0,1]))$.
\end{proposition}
\begin{proof}
Apply Ito's formula:
\begin{align*}
\left\vert \theta_{ n}(x, u) \right\vert^{k} &= \left\vert \theta_{ n}(x, 0) \right\vert^{k} + k\int_{ 0}^{u} \left\vert \theta_{ n}(x, s) \right\vert^{ k-2} \left\langle \theta_{ n}(x, s)\, ,\, c( \theta_{ n}(x,s))\right\rangle {\rm d}s\\ 
&+ k\int_{ 0}^{u} \left\vert \theta_{ n}(x,s) \right\vert^{ k-2} \left\langle \theta_{ n}(x, s)\, ,\, \int \Gamma \left( \theta_{ n}(x, s), \theta_{ n}(y, s)\right) \Xi_{x,y}^{(n)}{\rm d}y\right\rangle{\rm d}s\\
&+ k \int_{ 0}^{u} \left\vert \theta_{ n}(x, s) \right\vert^{ k-2} \left\langle \theta_{ n}(x, s)\, ,\, \sigma {\rm d}B_{ \left\lbrace x\right\rbrace_{ n}, s}\right\rangle\\
&+ \int_{ 0}^{u} k \left(\frac{ k}{ 2}-1\right) \left\vert \theta_{ n}(x,s) \right\vert^{ k-4} \left\langle \sigma \sigma^{ \dagger} \theta_{ n}(x,s)\, ,\, \theta_{ n}(x,s)\right\rangle{\rm d}s + kd\int_{ 0}^{u} \left\vert \theta_{ n}(x,s) \right\vert^{ k-2} {\rm d}s
\end{align*}
Let
\begin{equation*}
\theta_{ n}^{ \ast}(x, t):= \sup_{ s\in[0, t]} \left\vert \theta_{n}(x, s) \right\vert.
\end{equation*}
Using the hypothesis on $c$ and $ \Gamma$ in Section~\ref{sec:general_assumptions_nonlin}, we obtain
\begin{multline*}
\theta_{ n}^{ \ast}(x, t)^{ k} \leq\left\vert \theta_{ n}(x, 0) \right\vert^{ k}+ k \int_{ 0}^{t} \theta_{ n}^{ \ast}(x, s)^{ k-2} \left( \left(L_{ c}+ \frac{ 1}{ 2}\right) \theta_{ n}^{ \ast}(x, s)^{ 2}+ \frac{ \left\vert c(0) \right\vert^{ 2}}{2}\right) {\rm d}s\\ 
+ kL_{ \Gamma}\int_{ 0}^{t} \theta_{ n}^{ \ast}(x,s)^{ k-1} \int_{ I}\left(1+  \theta_{ n}^{ \ast}(x, s) + \theta_{ n}^{ \ast}(y, s) \right)\Xi_{x,y}^{(n)} {\rm d}y {\rm d}s\\
+ k \left(( \frac{ k}{ 2}-1) \left\vert \sigma \sigma^{ \dagger} \right\vert + d\right)\int_{ 0}^{t} \theta_{ n}^{ \ast}(x,s)^{ k-2}{\rm d}s+k\sup_{ u\leq t} \left\vert \int_{ 0}^{u} \left\vert \theta_{ n}(x, s) \right\vert^{ k-2} \left\langle \theta_{ n}(x, s)\, ,\, \sigma {\rm d}B_{ \left\lbrace x\right\rbrace_{ n}, s}\right\rangle \right\vert
\end{multline*}
Taking the square, using Jensen's inequality and taking the expectation, we obtain, for some constant $C>0$ depending on $k, T, c, \Gamma, \sigma, d$:
\begin{align}
\mathbf{ E} \left[ \theta_{ n}^{ \ast}(x, t)^{2 k}\right] &\leq C \mathbf{ E} \left[\left\vert \theta_{ n}(x, 0) \right\vert^{ 2k}\right] \nonumber\\
&+C \left(1+ \left(\sup_{ x\in I}\int_{ I}\Xi_{x,y}^{(n)} {\rm d}y\right)^{ 2} \right)\int_{ 0}^{t} \left(\mathbf{ E}\left[ \theta_{ n}^{ \ast}(x, s)^{ 2k-4}\right]+ \mathbf{ E} \left[\theta_{ n}^{ \ast}(x,s)^{ 2k-2}\right] + \mathbf{ E} \left[\theta_{ n}^{ \ast}(x,s)^{ 2k}\right]\right) {\rm d}s, \nonumber\\
&+C \mathbf{ E}\left[\int_{ 0}^{t} \theta_{ n}^{ \ast}(x,s)^{ 2k-2} \left(\int_{ I}\theta_{ n}^{ \ast}(y, s)\Xi_{x,y}^{(n)} {\rm d}y\right)^{ 2} {\rm d}s\right], \label{aux:thetan_xy_Xin}\\
&+C \mathbf{ E}\left[\sup_{ u\leq t} \left\vert \int_{ 0}^{u} \left\vert \theta_{ n}(x, s) \right\vert^{ k-2} \left\langle \theta_{ n}(x, s)\, ,\, \sigma {\rm d}B_{ \left\lbrace x\right\rbrace_{ n}, s}\right\rangle \right\vert^{ 2}\right] \label{aux:thetan_Bx}
\end{align}
Concentrate on the term \eqref{aux:thetan_xy_Xin}: for any $x\in I$, setting $i= \left\lbrace x\right\rbrace_{ n}$, note that
\begin{multline*}
\left(\int_{ I}\theta_{ n}^{ \ast}(y, s)\Xi_{x,y}^{(n)} {\rm d}y\right)^{ 2}= \left(\sum_{ j=1}^{ n} \kappa_{ i}^{ (n)} \xi_{ i, j}  \int_{ x_{ j-1}}^{ x_{ j}}\theta_{ n}^{ \ast}(y, s)  {\rm d}y\right)^{ 2},\\
\leq \left(\frac{ 1}{ n}\sum_{ j=1}^{ n} \kappa_{ i}^{ (n)} \xi_{ i, j}\right) \sum_{ j=1}^{ n} \kappa_{ i}^{ (n)} \xi_{ i, j}  \int_{ x_{ j-1}}^{x_{ j}} \theta_{ n}^{ \ast}(y, s)^{ 2} {\rm d}y,
\end{multline*}
applying Jensen's inequality for the probability measure on $I$ $\rho_{ n}= \frac{ \sum_{ j=1}^{ n} \kappa_{ i}^{ (n)} \xi_{ i, j}  \mathbf{ 1}_{ [x_{ j-1}, x_{ j})}(y) {\rm d}y}{\frac{ 1}{ n}\sum_{ j=1}^{ n} \kappa_{ i}^{ (n)} \xi_{ i, j}}$. Note also that by H\"older's inequality, \[\mathbf{ E} \left[\theta_{ n}^{ \ast}(x,s)^{ 2k-2} \theta_{ n}^{ \ast}(y, s)^{ 2}\right] \leq \mathbf{ E} \left[\theta_{ n}^{ \ast}(x,s)^{ 2k}\right]^{ \frac{ 2k-2}{ 2k}} \mathbf{ E} \left[\theta_{ n}^{ \ast}(y, s)^{ 2k}\right]^{ \frac{ 1}{ k}}.\]
Hence, if we define  
\begin{equation}
\Theta_{ n}(t):= \sup_{ x\in I} \mathbf{ E} \left[\theta_{ n}^{ \ast}(x, t)^{ 2k}\right]
\end{equation} the term \eqref{aux:thetan_xy_Xin} may be bounded by $\left(\int_{ I}\Xi_{x,y}^{(n)} {\rm d}y\right)^{ 2}  \int_{ 0}^{t}\Theta_{ n}(s) {\rm d}s$. Finally, concerning the last term \eqref{aux:thetan_Bx}, an application of Burkholder-Davis-Gundy inequality gives, for some constant $C>0$ that depends on $ \sigma$,
\begin{align*}
\mathbf{ E}\left[\sup_{ u\leq t} \left\vert \int_{ 0}^{u} \left\vert \theta_{ n}(x, s) \right\vert^{ k-2} \left\langle \theta_{ n}(x, s)\, ,\, \sigma {\rm d}B_{ \left\lbrace x\right\rbrace_{ n}, s}\right\rangle \right\vert^{ 2}\right]&\leq C \int_{ 0}^{t} \mathbf{ E}\left[\left\vert \theta_{ n}^{ \ast}(x, s) \right\vert^{ 2k-2}\right]{\rm d}s.
\end{align*}
Putting everything together and using the fact that there exist a universal constant $C_{ k}>0$ (only depending on $k$) such that $ \left\vert x \right\vert^{ 2k-4} + \left\vert x \right\vert^{ 2k-2} \leq C_{ k} + \left\vert x \right\vert^{ 2k}$, we obtain finally that for some constant $C^{ \prime}>0$ depending on $T, k, \sigma, c, \Gamma$,
\begin{align*}
\Theta_{ n}(t) \leq C^{ \prime} \Theta_{ n}(0) + C^{ \prime} \left(1+ \sup_{ x\in I}  \left(\int_{ I} \Xi_{x, y}^{(n)} {\rm d}y\right)^{ 2} \right)+ C^{ \prime} \left(1+ \sup_{ x\in I}\left(\int_{ I}\Xi_{x, y}^{(n)}  {\rm d}y\right)^{ 2}\right) \int_{ 0}^{t} \Theta_{ n}(s) {\rm d}s.
\end{align*}
Since we have almost surely $ \sup_{ x\in I} \left(\int \Xi_{x, y}^{(n)}{\rm d} y\right)^{ 2} = \sup_{ i\in \left[n\right]}\left(\frac{ 1}{ n}\sum_{ j=1}^{ n}  \kappa_{i}^{ (n)} \xi_{ i, j}\right)^{ 2}< \infty$ (recall \eqref{hyp:bn_xi}, \eqref{eq:control_sum_xi} and \eqref{hyp:bound_W_L1_sum}) as well as the hypothesis on the initial condition \eqref{hyp:second_moment_nu0}, we obtain from Gr\"onwall's Lemma that
\begin{equation}
\sup_{ n\geq1}\sup_{x\in I} \mathbf{ E} \left[\sup_{ t\in[0, T]}\left\vert \theta_{ n}(x, t) \right\vert^{ 2k}\right]  < \infty.
\end{equation}
This implies in particular that, by two successive applications of Jensen's inequality,
\begin{align}
\sup_{ n\geq1}&\mathbf{ E} \left[\sup_{ t\in[0, T]} \left\Vert \theta_{ n}(\cdot, t) \right\Vert_{ L^{ k}(I)}^{ 2}\right] = \sup_{ n\geq1} \mathbf{ E} \left[ \sup_{ t\in[0, T]} \left( \int \left\vert \theta_{ n}(x, t) \right\vert^{ k} {\rm d}x\right)^{ \frac{ 2}{ k}}\right], \nonumber\\
&\leq \sup_{ n\geq1} \mathbf{ E} \left[\left( \int \sup_{ t\in[0, T]}\left\vert \theta_{ n}(x, t) \right\vert^{ 2k} {\rm d}x\right)^{ \frac{ 1}{ k}}\right]\leq \sup_{ n\geq1} \mathbf{ E} \left[ \int \sup_{ t\in[0, T]}\left\vert \theta_{ n}(x, t) \right\vert^{ 2k} {\rm d}x\right]^{ \frac{ 1}{ k}}, \nonumber\\
&=\sup_{ n\geq1} \left(\int \mathbf{ E} \left[\sup_{ t\in[0, T]}\left\vert \theta_{ n}(x, t) \right\vert^{ 2k} \right]{\rm d}x \right)^{ \frac{ 1}{ k}}\leq \left( \sup_{ n\geq1}\sup_{ x\in I}\mathbf{ E} \left[\sup_{ t\in[0, T]}\left\vert \theta_{ n}(x, t) \right\vert^{ 2k} \right]\right)^{ \frac{ 1}{ k}}< \infty.\label{eq:tightness_theta}
\end{align} 
We now turn to a similar estimate concerning the modulus of continuity: set $ \delta>0$ and $s<t\in [0, T]$ such that $ \left\vert t-s \right\vert \leq \delta$. Set
\begin{equation}
\Delta_{ n, s,t}(x):= \left\vert \theta_{ n}(x, t) - \theta_{ n}(x, s) \right\vert
\end{equation}
and
\begin{equation}
\Delta_{ n, \delta}^{ \ast}(x):= \sup_{ s}\Delta_{ n, s,s+ \delta}(x)
\end{equation}
Then, by the same calculations as before, it is straightforward to see that, for another constant $C^{ \prime}>0$,
\begin{align*}
\sup_{ x\in I}\mathbf{ E} \left[\Delta_{ n, \delta}^{ \ast}(x)^{ 2k}\right]&\leq C^{ \prime}\left(1+ \left(\sup_{ x\in I}\int_{ I}\Xi_{x,y}^{(n)} {\rm d}y\right)^{ 2} \right)\int_{ 0}^{ \delta}  \sup_{ x\in I} \mathbf{ E} \left[\Delta_{ n, u}^{ \ast}(x)^{ 2k}\right] {\rm d}u\\ 
&+ C^{ \prime} \mathbf{ E} \left[\sup_{ s, t, \left\vert t-s \right\vert\leq \delta} \left\vert \int_{ s}^{t} \left\vert \theta_{ n}(x, u)- \theta_{ n}(x, s) \right\vert^{ k-2} \left\langle \theta_{ n}(x, u)- \theta_{ n}(x, s)\, ,\, \sigma {\rm d}B_{ \left\lbrace x\right\rbrace_{ n}, u}\right\rangle \right\vert^{ 2}\right]
\end{align*}
Hence, once again by H\"older and Burkholder-Davis-Gundy inequalities, 
\begin{align*}
\sup_{ x}\mathbf{ E} \left[\Delta_{ n, \delta}^{ \ast}(x)^{ 2k}\right]&\leq C^{ \prime\prime} \delta +  C^{ \prime\prime} \left(1+ \sup_{ x\in I}\left(\int \Xi_{x,y}^{(n)}{\rm d}y\right)^{ 2} \right)\int_{ 0}^{ \delta}  \sup_{ x\in I}\mathbf{ E} \left[\Delta_{ n, u}^{ \ast}(x)^{ 2k}\right] {\rm d}u.
\end{align*}
Gr\"onwall's lemma gives directly that
\begin{equation}
\sup_{ n\geq1}\sup_{ x\in I}\mathbf{ E} \left[ \sup_{ \left\vert t-s \right\vert\leq \delta} \left\vert \theta_{ n}(x, t) - \theta_{ n}(x, s) \right\vert^{ 2k}\right] \xrightarrow[\delta \searrow 0]{}0,
\end{equation}
which finally gives
\begin{equation}
\label{eq:tightness_Delta}
\sup_{ n\geq1}\mathbf{ E} \left[ \sup_{ \left\vert t-s \right\vert\leq \delta} \left\Vert \theta_{ n}(\cdot, t) - \theta_{ n}(\cdot, s) \right\Vert^{ 2}_{ L^{ k}(I)}\right] \xrightarrow[\delta \searrow 0]{}0.
\end{equation}
The required tightness result follows directly from \eqref{eq:tightness_theta} and \eqref{eq:tightness_Delta}.
\end{proof}
\begin{proposition}
Under the assumption of Section~\ref{sec:convergence_profile_intro}, any accumulation point of $ \left(\theta_{ n}(\cdot, t)\right)_{ n\geq1, t\in[0, T]}$ in $\mathcal{ C}([0, T], L^{ k}(I))$ is a weak solution to \eqref{eq:heat_equation_weak}.
\end{proposition}
\begin{proof}
Let us consider the field $ \psi_{ n}$ \eqref{eq:psi_spatial_profile} corresponding to the particle system \eqref{eq:psi_particle_system} driven by the same Brownian motions as \eqref{eq:odegene} with the same initial condition. By calculations similar to the previous proof, the difference $ \theta_{ n}- \psi_{ n}$ verifies, using the assumptions on $c$,
\begin{multline*}
 \mathbf{ E} \left[\left\vert \theta_{ n}(x, t)- \psi_{ n}(x,t) \right\vert^{ 2k}\right]\leq C\int_{ 0}^{t} \mathbf{ E} \left[\left\vert \theta_{ n}(x, u)- \psi_{ n}(x,u) \right\vert^{ 2k}\right] {\rm d}u+ C \mathbf{ E}\int_{ 0}^{t} \left\vert \theta_{ n}(x, u)- \psi_{ n}(x,u) \right\vert^{ 2k-2} \\
 \left\langle \theta_{ n}(x, u)- \psi_{ n}(x, u)\, ,\, \int \left(\Gamma \left( \theta_{ n}(x, u), \theta_{ n}(y, u)\right)\Xi_{x, y}^{(n)} - \Gamma \left( \psi_{ n}(x, u), \psi_{ n}(y, u)\right) \hat W_{ n}(x, y)\right) {\rm d}y\right\rangle{\rm d}u
\end{multline*}
Using the assumptions on $ \Gamma$, \eqref{hyp:Gamma_Lip_1}, the difference in the last term above can be bounded by  $C(A_{ 1}+A_{ 2})$ where $A_{ 1}$ and $A_{ 2}$ are given below. First,
\begin{multline*}
A_{ 1}:=\int_{ 0}^{t} \mathbf{ E} \left\vert \theta_{ n}(x, u)- \psi_{ n}(x,u) \right\vert^{ 2k-2}\\
\int_{I}\left\langle \theta_{ n}(x, u) - \psi_{ n}(x, u)\, ,\, \Gamma(\theta_{ n}(x, u), \theta_{ n}(y, u))- \Gamma(\psi_{ n}(x, u), \psi_{ n}(y, u))\right\rangle\Xi^{ (n)}_{x, y} {\rm d}y {\rm d}u,\\
\leq C\int_{ 0}^{t} \mathbf{ E} \left\vert \theta_{ n}(x, u)- \psi_{ n}(x,u) \right\vert^{ 2k-1} \int_{ 0}^{1} \left( \left\vert \theta_{ n}(x, u) - \psi_{ n}(x, u) \right\vert + \left\vert \theta_{ n}(y, u) - \psi_{ n}(y, u) \right\vert\right)\Xi^{ (n)}_{x, y} {\rm d}y {\rm d}u,\\
\leq 2C\int_{ 0}^{1} \Xi^{ (n)}_{x, y} {\rm d}y\int_{ 0}^{t} \mathbf{ E} \left\vert \theta_{ n}(x, u) - \psi_{ n}(x, u) \right\vert^{ 2k}  {\rm d}u \\
+ 2C\int_{ 0}^{t} \int_{ 0}^{1}  \mathbf{ E} \left\vert \theta_{ n}(x, u) - \psi_{ n}(x, u) \right\vert^{ 2k-1} \left\vert \theta_{ n}(y, u) - \psi_{ n}(y, u) \right\vert  \Xi^{ (n)}_{x,y} {\rm d}y {\rm d}u.
\end{multline*}
By H\"older's inequality followed by Young's inequality, 
\begin{multline*}
\mathbf{ E} \left[\left\vert \theta_{ n}(x, u) - \psi_{ n}(x, u) \right\vert^{ 2k-1} \left\vert \theta_{ n}(y, u) - \psi_{ n}(y, u) \right\vert\right] \\ \leq \mathbf{ E} \left[\left\vert \theta_{ n}(x, u) - \psi_{ n}(x, u) \right\vert^{ 2k}\right]^{ \frac{ 2k-1}{ 2k}} \mathbf{ E}\left[\left\vert \theta_{ n}(y, u) - \psi_{ n}(y, u) \right\vert^{ 2k}\right]^{ \frac{ 1}{ 2k}},\\
\leq \frac{ 2k-1}{ 2k} \mathbf{ E} \left[\left\vert \theta_{ n}(x, u) - \psi_{ n}(x, u) \right\vert^{ 2k}\right] + \frac{ 1}{ 2k}\mathbf{ E}\left[\left\vert \theta_{ n}(y, u) - \psi_{ n}(y, u) \right\vert^{ 2k}\right],\\
\leq \sup_{ z\in [0, 1]}\mathbf{ E} \left[\left\vert \theta_{ n}(z, u) - \psi_{ n}(z, u) \right\vert^{ 2k}\right]
\end{multline*}
Hence,
\begin{align*}
A_{ 1}&\leq 4C \left(\sup_{ x\in I}\int_{ I}\Xi^{ (n)}_{x, y} {\rm d}y\right)\int_{ 0}^{t} \sup_{ z\in I}\mathbf{ E} \left[\left\vert \theta_{ n}(z, u) - \psi_{ n}(z, u) \right\vert^{ 2k}\right]  {\rm d}u.
\end{align*}
Secondly,
\begin{align*}
A_{2}&:=\int_{ 0}^{t} \mathbf{ E} \int_{I} \left\langle \theta_{ n}(x, u) - \psi_{ n}(x, u)\, ,\, \Gamma(\psi_{ n}(x, u), \psi_{ n}(y, u))\right\rangle \left(\Xi^{ (n)}_{x, y} - \hat W_{ n}(x, y)\right){\rm d}y {\rm d}s.
\end{align*}
For $i= \left\lbrace x\right\rbrace_{ n}$, we have
 \begin{align*}
A_{2}&:= 2 \int_{0}^{t}\frac{1}{n} \sum_{j=1}^{n} \zeta^{(n)}_{i, j,u} \left(\kappa_{i}^{ (n)} \xi_{i, j} - W(x_{i}, x_{j})\right) {\rm d}u,\\
 &= 2 \int_{0}^{t}\frac{1}{n} \sum_{j=1}^{n} \zeta^{(n)}_{i, j,u} \kappa_{i}^{ (n)}\left(\xi_{i, j} - W_{ n}(x_{i}, x_{j})\right) {\rm d}u + 2 \int_{0}^{t}\frac{1}{n} \sum_{j=1}^{n} \zeta^{(n)}_{i, j,u} \left(\kappa_{i}^{ (n)} W_{ n}(x_{ i}, x_{ j}) - W(x_{i}, x_{j})\right) {\rm d}u,
\end{align*}
where $\zeta^{(n)}_{i, j,s}:=\mathbf{ E} \left[\left\langle \theta_{ i, s}^{(n)} - \psi_{i, s}^{(n)}\, ,\, \Gamma(\psi_{ i, s}^{(n)}, \psi_{ j, s}^{(n)})\right\rangle\right]$. The same reasoning (together with the apriori control \eqref{eq:apriori_bound_odegene}) as in Proposition~\ref{prop:control_Delta_2} shows that the supremum in $i\in \left[n\right]$ of the first term in the sum above goes almost surely to $0$ as $n\to \infty$. The second term in the sum is bounded above by $\sup_{ i\in \left[n\right]}\frac{C}{n} \sum_{j=1}^{n} \left\vert \kappa_{i}^{ (n)} W_{n}(x_{i}, x_{j}) - W(x_{i}, x_{j})\right\vert$, which goes to $0$ uniformly in $i$ as $n\to\infty$, by \eqref{hyp:Delta_n_1}. From the previous estimates and a Gr\"onwall's Lemma, we conclude that $\sup_{x\in I, t\in[0, T]}\mathbf{ E}\left\vert \theta_{ n}(x, t) - \psi_{ n}(x, t)\right\vert^{ 2k}\to_{n\to\infty}0$. In particular,
\begin{equation}
\label{eq:theta_nxt_VS_psi_nxt}
\sup_{t\in[0, T]} \mathbf{E} \int_{I} \left\vert \theta_{ n}(x, t) - \psi_{ n}(x, t)\right\vert^{ 2k} {\rm d}x\to 0,\text{ as }n\to\infty.
\end{equation}
It now remains to identify the limit. By \eqref{eq:theta_nxt_VS_psi_nxt}, any limit point of $( \theta_{ n}(\cdot, t))_{ n\geq1}$ in $ \mathcal{ C}([0, T], L^{ k}(I))$ is also a limit point of $( \psi_{ n}(\cdot, t))_{ n\geq1}$. Recall that, for any $ \mathcal{ C}^{ 1}$-bounded test function $J:I \to \mathbb{ R}^{ d}$ with bounded derivative
\begin{multline}
\label{eq:psin_J}
\int_{I} \left\langle \psi_{n}(x, t)\, ,\, J(x)\right\rangle  {\rm d}x= \int_{ I} \left\langle \psi_{n}(x, 0)\, ,\,  J(x)\right\rangle {\rm d}x + \int_{ 0}^{t}\int_{I} \left\langle c( \psi_{n}(x, s))\, ,\, J(x)\right\rangle  {\rm d}x{\rm d}s \\
+ \int_{ 0}^{t} \int_{ I^{ 2}}\left\langle  \Gamma( \psi_{n}(x, s) , \psi_{n}(y, s)) \hat W_{ n}(x, y)\, ,\,  J(x)\right\rangle{\rm d}x{\rm d}y{\rm d}s +  \int_{I} \left\langle \sigma B_{ \left\lbrace x\right\rbrace_{ n}, t}\, ,\, J(x)\right\rangle {\rm d}x.
\end{multline} 
Concerning the initial condition $ \theta_{ n}(x, 0)= \psi_{ n}(x, 0)$ (recall the definition of $\psi_{ 0}$ in \eqref{hyp:initial_condition_psi}):
\begin{multline*}
\mathbf{E}\left\vert \int_{I} \left\langle \theta_{n}(x, 0)\, ,\, J(x)\right\rangle  {\rm d}x - \int_{I} \left\langle \psi_{ 0}(x)\, ,\, J(x)\right\rangle  {\rm d}x\right\vert^{ 2} \\
%= \mathbf{ E} \left\vert \frac{ 1}{ n}\sum_{ k=1}^{ n} \left\langle \theta_{ k, 0}^{ (n)}\, ,\, n \int_{x_{ k-1}}^{x_{ k}} J(u) {\rm d}u\right\rangle  - \int_{I} \left\langle \psi_{ 0}(x)\, ,\, J(x)\right\rangle {\rm d}x\right\vert^{ 2},\\
\leq 3\mathbf{ E} \left\vert \frac{ 1}{ n}\sum_{ k=1}^{ n} \left\langle \theta_{ k, 0}^{ (n)} - \mathbf{ E} \left[ \theta_{ k, 0}^{ (n)}\right]\, ,\, n \int_{x_{ k-1}}^{x_{ k}} J(u) {\rm d}u\right\rangle \right\vert^{ 2},\\
+3 \left\vert \frac{ 1}{ n}\sum_{ k=1}^{ n} \left\langle \int \theta \nu_{ 0}^{ x_{ k}}({\rm d}\theta)\, ,\, n \int_{x_{ k-1}}^{x_{ k}} J(u) {\rm d}u - J(x_{ k})\right\rangle\right\vert^{ 2},\\
+ 3\left\vert \frac{ 1}{ n}\sum_{ k=1}^{ n} \left\langle \int \theta \nu_{ 0}^{ x_{ k}}({\rm d}\theta)\, ,\, J(x_{ k})\right\rangle - \int_{I} \left\langle \psi_{ 0}(x)\, ,\, J(x)\right\rangle{\rm d}x\right\vert^{ 2},\\
\leq \frac{ 3 \left\Vert J \right\Vert_{ \infty}^{ 2}}{ n} \sup_{ x\in I} \int \left\vert \theta - \int \theta \nu_{ 0}^{ x}({\rm d}\theta) \right\vert^{ 2} \nu_{ 0}^{ x}({\rm d}\theta)+3 \sup_{ x\in I} \left\vert \int \theta \nu_{ 0}^{ x}({\rm d}\theta) \right\vert^{ 2} \frac{ \left\Vert J^{ \prime} \right\Vert_{ \infty}^{ 2}}{ 4n^{ 2}},\\
+ 3\left\vert \frac{ 1}{ n}\sum_{ k=1}^{ n}\left\langle \int \theta \nu_{ 0}^{ x_{ k}}({\rm d}\theta)\, ,\, J(x_{ k})\right\rangle  - \int_{I} \left\langle \psi_{ 0}(x)\, ,\, J(x)\right\rangle {\rm d}x\right\vert^{ 2}.
\end{multline*}
The two first terms above converge to $0$, by \eqref{hyp:second_moment_nu0}. It is straightforward to see that the third term is $3 \left\vert \int_{ I}\left\langle \psi_{ 0}(\cdot)\, ,\, J(\cdot)\right\rangle ({\rm d}\ell_{ n}- {\rm d}\ell)\right\vert$ where $ x \mapsto \left\langle \psi_{ 0}(x)\, ,\, J(x)\right\rangle$ is $ \iota_{ 1}$-H\"older, by \eqref{hyp:nu_0_Lip}. Hence, this term also goes to $0$ as $n\to\infty$. Concerning the noise term in \eqref{eq:psin_J}, we have, for some constant $C>0$
\begin{multline*}
\mathbf{ E} \left[\left\vert \int_{I} \left\langle \sigma B_{ \left\lbrace x\right\rbrace_{ n}, t}\, ,\, J(x)\right\rangle  {\rm d} x \right\vert^{ 2}\right] = \mathbf{ E} \left[\left\vert\sum_{ j=1}^{ n} \int_{x_{ j-1}}^{x_{ j}} \left\langle \sigma B_{ j, t}\, ,\, J(x)\right\rangle    {\rm d} x\right\vert^{ 2}\right],\\
= \sum_{ i, j=1}^{ n} \int_{x_{ i-1}}^{x_{ i}} \int_{x_{ j-1}}^{x_{ j}}  \mathbf{ E} \left[ \left\langle \sigma B_{ i, t}\, ,\, J(x)\right\rangle\left\langle \sigma B_{ j, t}\, ,\, J(y)\right\rangle \right]{\rm d} x {\rm d} y,\\
\leq C t \sum_{ i=1}^{ n} \left(\int_{x_{ i-1}}^{x_{ i}} \left\vert J(x) \right\vert {\rm d} x\right)^{ 2} \leq  \frac{ C T}{ n}\sum_{ i=1}^{ n} \int_{x_{ i-1}}^{x_{ i}} \left\vert J(x) \right\vert^{ 2} {\rm d} x= \frac{ C(\sigma) T}{ n} \int_{I} \left\vert J(x) \right\vert^{ 2}{\rm d} x,
\end{multline*}
which goes to $0$ as $n\to \infty$, uniformly on $t\in[0,T]$.

Let $\left(\psi(x,t)\right)_{ x\in I, t\in[0,T]}\in \mathcal{ C}([0, T], L^{ k}(I))$ an accumulation point of $( \psi_{ n})$: there exists a subsequence $(n_{ k})_{ k\geq1}$ (that we rename $n$ for simplicity of exposition) such that $ (\psi_{ n})_{ n\geq1}$ converges in law to $ \psi$ in $ \mathcal{ C}([0, T], L^{ k}(I))$. By the Skhorokhod representation theorem ($\mathcal{ C}([0, T], L^{ k}(I))$ is separable, \cite{MR2378491}, p.125), one can suppose that $(\psi_{ n})$ converges almost surely in $\mathcal{ C}([0, T], L^{ k}(I))$ to $ \psi$. Since we have convergence in $\mathcal{ C}([0, T], L^{ k}(I))$, we have that $\int_{ I} \left\langle \psi_{n}(x, t)\, ,\, J(x)\right\rangle  {\rm d}x \xrightarrow[n\to\infty]{} \int_{I} \left\langle \psi(x, t)\, ,\, J(x)\right\rangle  {\rm d}x$, uniformly in $t\in [0, T]$. For the same reason,  $\int_{ 0}^{t}\int_{I} \left\langle c( \psi_{n}(x, s))\, ,\, J(x)\right\rangle  {\rm d}x{\rm d}s \xrightarrow[n\to\infty]{} \int_{ 0}^{t}\int_{I} \left\langle c( \psi(x, s))\, ,\, J(x)\right\rangle  {\rm d}x{\rm d}s $ (by an application of dominated convergence theorem when $c$ is bounded or using the convergence in $\mathcal{ C}([0, T], L^{ k}(I))$ when $c$ is polynomial). We now turn the interaction term in \eqref{eq:psin_J}. Assume first that $ \Gamma$ is bounded. Then,
\begin{multline*}
\left\vert \int_{ 0}^{t} \int_{ I^{ 2}} \left\langle \Gamma( \psi_{n}(x, s) , \psi_{n}(y, s)) \left(\hat W_{ n}(x, y) - W(x, y)\right)\, ,\, J(x)\right\rangle {\rm d}x{\rm d}y{\rm d}s \right\vert \leq\\
\left\Vert \Gamma \right\Vert_{ \infty}\left\Vert J \right\Vert_{ \infty} T \int_{ I^{ 2}}\left\vert\hat W_{ n}(x, y) - W(x, y)\right\vert {\rm d}x{\rm d}y,
\end{multline*}
which goes to $0$ as $n\to\infty$, by \eqref{hyp:Int_Wn_W_2}. Moreover,
\begin{multline*}
 \left\vert \int_{ 0}^{t} \int_{I^{ 2}} \left\langle \left(\Gamma( \psi_{n}(x, s) , \psi_{n}(y, s)) - \Gamma( \psi(x, s) , \psi(y, s))\right) W(x, y)\, ,\, J(x)\right\rangle {\rm d}x{\rm d}y{\rm d}s \right\vert\\
\leq \int_{ 0}^{t} \int_{I^{ 2}} \left\vert\Gamma( \psi_{n}(x, s) , \psi_{n}(y, s)) - \Gamma( \psi(x, s) , \psi(y, s))\right\vert W(x, y)\left\vert J(x) \right\vert{\rm d}x{\rm d}y{\rm d}s,
\end{multline*}
which goes to $0$ as $n\to\infty$, by dominated convergence theorem. In the case where $ \Gamma(\theta, \theta^{ \prime})= \Gamma\cdot(\theta- \theta^{ \prime})$ is linear, we have firstly
\begin{multline*}
\left\vert \int_{ 0}^{t} \int_{ I^{ 2}} \left\langle \left(\psi_{n}(x, s)-\psi(x, s)\right) \hat W_{ n}(x, y)\, ,\, J(x)\right\rangle {\rm d}x{\rm d}y{\rm d}s \right\vert\\
\leq \sup_{ z\in I}\left(\int_{I} \hat W_{ n}(z, y) {\rm d}y\right) \left\Vert J \right\Vert_{ \infty}\int_{ 0}^{t} \int_{ I} \left\vert \psi_{n}(x, s)-\psi(x, s)\right\vert {\rm d}x{\rm d}s
\end{multline*}
which goes to $0$ as $n\to\infty$, by \eqref{hyp:bound_W_L1_sum} and since we have convergence in $\mathcal{ C}([0, T], L^{ k}(I))$. Secondly,
\begin{multline*}
\left\vert \int_{ 0}^{t} \int_{I^{ 2}} \left\langle \psi(x, s)\, ,\, J(x)\right\rangle \left(\hat W_{ n}(x, y) - W(x, y)\right) {\rm d}x{\rm d}y{\rm d}s \right\vert,\\
\leq \left\vert \int_{ 0}^{t} \left(\int_{ I} \left\vert \left\langle \psi(x, s)\, ,\, J(x)\right\rangle  \right\vert^{ 2} {\rm d}x\right)^{ \frac{ 1}{ 2}}\left(\int_{ I^{ 2}} \left(\hat W_{ n}(x, y) - W(x, y)\right)^{ 2} {\rm d}y{\rm d}x\right)^{ \frac{ 1}{ 2}}{\rm d}s \right\vert,
\end{multline*}
which goes to $0$ as $n\to\infty$, by \eqref{hyp:Int_Wn_W_2} and since $ \psi\in \mathcal{ C}([0, T], L^{ k}(I))$. Thirdly,
\begin{multline*}
\left\vert \int_{ 0}^{t} \int_{I^{ 2}} \left\langle \left(\psi_{n}(y, s)-\psi(y, s)\right)\, ,\, J(x)\right\rangle \hat W_{ n}(x, y) {\rm d}x{\rm d}y{\rm d}s \right\vert\\
\leq \left\Vert J \right\Vert_{ \infty}\left\vert \int_{ 0}^{t}  \left(\int_{I} \left\vert\psi_{n}(y, s)-\psi(y, s)\right\vert^{ 2} {\rm d}y\right)^{ \frac{ 1}{ 2}} \left(\int_{ I^{ 2}}\hat W_{ n}(x, y)^{ 2}{\rm d}x{\rm d}y \right)^{ \frac{ 1}{ 2}}{\rm d}s \right\vert,
\end{multline*}
which goes to $0$ as $n\to\infty$. Finally,
\begin{multline*}
\left\vert \int_{ 0}^{t} \int_{I^{ 2}} \left(\hat W_{ n}(x, y) - W(x, y)\right) \left\langle \psi(y, s)\, ,\, J(x)\right\rangle {\rm d}x{\rm d}y{\rm d}s \right\vert,\\
\leq \left\Vert J \right\Vert_{ \infty} \int_{ 0}^{t}  \left(\int_{ I} \vert\psi(y, s)\vert^{ 2} {\rm d}y\right)^{ \frac{ 1}{ 2}} \left(\int_{ I^{ 2}} \left\vert\hat W_{ n}(x, y) - W(x, y)\right\vert^{ 2}{\rm d}x{\rm d}y\right)^{ \frac{ 1}{ 2}}{\rm d}s,
\end{multline*}
which also goes to $0$ as $n\to\infty$. These estimates altogether gives the convergence of the interaction term in \eqref{eq:psin_J}. Putting everything together, we obtain that any accumulation point of $( \theta_{ n})_{ n\geq1}$ in $\mathcal{ C}([0, T], L^{ k}(I))$ is a weak solution to \eqref{eq:heat_equation_weak}. 
\end{proof}
\subsection{Identification in the compact case}
\label{sec:ident_compact_case}
We prove Theorem~\ref{theo:identification}. Let $x \mapsto J(x)$ be a regular ($ \mathcal{ C}^{ 1}$) test function on $[0,1]$. Then,
\begin{equation}
\label{eq:thetan_J_nun}
\int_{I} \left\langle \theta_{ n}(x, t)\, ,\, J(x)\right\rangle  {\rm d} x = \frac{ 1}{ n} \sum_{ i=1}^{ n} \left\langle \theta_{ i, t}^{ (n)}\, ,\, \bar J \left( \frac{ i}{ n}\right)\right\rangle ,
\end{equation}
where $\bar J \left( \frac{ i}{ n}\right):= n \int_{ \frac{ i-1}{ n}}^{ \frac{ i}{ n}} J(x) {\rm d} x$. The expression $\bar J \left( \frac{ i}{ n}\right)$ is not an actual function of $ \frac{ i}{ n}$, but one can replace $ \bar J \left(\frac{ i}{ n}\right)$ by $J \left( \frac{ i}{ n}\right)$:
\begin{align*}
\mathbf{ E} \left[ \frac{ 1}{ n} \sum_{ i=1}^{ n} \left\langle \theta_{ i, t}^{ (n)}\, ,\, \bar J \left( \frac{ i}{ n}\right) -  J \left( \frac{ i}{ n}\right)\right\rangle \right]^{ 2}&\leq \frac{ 1}{ n} \sum_{ i=1}^{ n} \mathbf{ E} \left\vert \theta_{ i, t}^{ (n)} \right\vert^{ 2} \left\vert \bar J \left( \frac{ i}{ n}\right) -  J \left( \frac{ i}{ n}\right) \right\vert^{ 2},\\
&\leq \frac{ C}{ n} \sum_{ i=1}^{ n}  \left\vert \bar J \left( \frac{ i}{ n}\right) -  J \left( \frac{ i}{ n}\right) \right\vert^{ 2}\leq \frac{ C \left\Vert J^{ \prime} \right\Vert_{ \infty}^{ 2}}{ 4 n^{ 2}},
\end{align*} using \eqref{eq:apriori_bound_odegene}. So the limit of $\int_{ I} \left\langle \theta_{ n}(x, t)\, ,\, J(x)\right\rangle  {\rm d} x$ as $n \to \infty$ is the same as 
\begin{equation}
U_{ n, t}(J):=\frac{ 1}{ n} \sum_{ i=1}^{ n} \left\langle \theta_{ i, t}^{ (n)} \, ,\, J \left( \frac{ i}{ n}\right)\right\rangle = \int \left\langle \theta\, ,\, J(x)\right\rangle \nu_{ n, t}({\rm d} \theta, {\rm d} x).
\end{equation}
Taking the limit as $ n\to\infty$ in \eqref{eq:thetan_J_nun}, using Theorem~\ref{theo:conv_empirical_measure}, one obtains that, for all $t\in[0, T]$,
\begin{equation}
\int_{I} \left\langle \psi(x, t)\, ,\, J(x)\right\rangle {\rm d} x= \int_{I}\int \left\langle \theta\, ,\, J(x)\right\rangle \nu_{ t}^{ x}({\rm d}\theta) {\rm d}x.
\end{equation}
Theorem~\ref{theo:identification} follows.
\appendix
\section{Well-posedness and regularity results for the nonlinear Fokker-Planck PDE}
\label{sec:wellposed_PDE}
The aim of this section is to prove Proposition~\ref{prop:PDE_wellposed}, as well as some regularity estimates concerning the solution $ \nu$ to \eqref{eq:McKeanVlasov}.
\subsection{Existence of a solution to the nonlinear Fokker-Planck PDE}
We prove here the existence part of Proposition~\ref{prop:PDE_wellposed} and the result of Remark~\ref{rem:apriori_nonlin}. Recall the definition of the set $ \mathcal{ M}$ in Section~\ref{sec:well_posedness_intro}. For any $m\in \mathcal{ M}$, consider the solution to
\begin{equation}
\label{eq:sde_m}
{\rm d} \theta_{ t}^{ x} = c(\theta_{ t}^{ x}) {\rm d}t + \int \Gamma(\theta_{ t}^{ x}, \tilde{ \theta}) W(x, y) m_{ t}({\rm d} \tilde{ \theta}, {\rm d}y) {\rm d}t + \sigma {\rm d}B_{ t},
\end{equation}
with initial condition $ \theta_{ 0}^{ x}\sim \nu_{ 0}^{ x}$. Consider $ \Theta:\mathcal{ M} \to \mathcal{ M}$ the functional which maps any measure $m({\rm d}\theta, {\rm d}x)$ to the law $ \Theta(m)$ of $(\theta^{ x}, x)$, where $ \theta^{ x}$ solves \eqref{eq:sde_m}. The point is to prove that $ \Theta$ admits a fixed-point in $ \mathcal{ M}$, which gives the existence of a solution $ \nu$ to \eqref{eq:McKeanVlasov}. Recall the definition of the Wasserstein metric in \eqref{eq:wasserstein_ptfixe} and the definition of $k$ in \eqref{hyp:polynomial_control_c}. Then, for two driving measures $m_{ 1}$ and $m_{ 2}$ in $ \mathcal{ M}$, for the coupling $( \theta_{ 1}, \theta_{ 2})$ with same initial conditions and Brownian noise, using the properties on $c$ and $ \Gamma$, we get
\begin{align*}
\left\vert \theta_{ 1, t} - \theta_{ 2, t}\right\vert^{ 2} 
&\leq C\int_{ 0}^{t} \left\vert \theta_{ 1, s} - \theta_{ 2, s}\right\vert^{ 2} {\rm d}s\\& + C \int_{ 0}^{t} \left\vert \theta_{ 1, s} - \theta_{ 2, s}\right\vert \int \left\vert\Gamma(\theta_{ 1, s}, \tilde{ \theta}) - \Gamma(\theta_{ 2, s}, \tilde{ \theta})\right\vert W(x, y) m_{ 1, s}({\rm d}\tilde{ \theta}, {\rm d}y){\rm d}s\\
&+ C \int_{ 0}^{t} \left\vert \theta_{ 1, s} - \theta_{ 2, s}\right\vert \left\vert \int \left\lbrace\Gamma(\theta_{ 2, s}, u)-\Gamma(\theta_{ 2, s}, v)\right\rbrace p_{ s}^{ y}({\rm d}u, {\rm d}v) W(x, y) \ell({\rm d}y)\right\vert{\rm d}s,
\end{align*}
where $p^{ y}({\rm d}u, {\rm d}v)$ is \emph{any} coupling of $m_{ 1}^{ y}$ and $m_{ 2}^{ y}$. Hence,
\begin{align*}
\left\vert \theta_{ 1, t} - \theta_{ 2, t}\right\vert^{ 2} &\leq C \left(1+ \left\Vert \mathcal{ W}_{ 1} \right\Vert_{ \infty}\right)\int_{ 0}^{t} \left\vert \theta_{ 1, s} - \theta_{ 2, s}\right\vert^{ 2} {\rm d}s\\
&+ C \int_{ 0}^{t} \left\vert \theta_{ 1, s} - \theta_{ 2, s}\right\vert \int \left\vert u- v\right\vert  p_{ s}^{ y}({\rm d}u, {\rm d}v) W(x, y) \ell({\rm d}y){\rm d}s.
\end{align*}
Elevating everything to the power $ k\geq 1$ and taking the supremum in $s\leq t$ and the expectation, one obtains, for another constant $C=C(T)$
\begin{align*}
\mathbf{ E} \left[\sup_{ s\leq t}\left\vert \theta_{ 1, s} - \theta_{ 2, s}\right\vert^{ 2k}\right] &\leq C \left(1+ \left\Vert \mathcal{ W}_{ 1} \right\Vert_{ \infty}\right)^{k}\int_{ 0}^{t} \mathbf{ E} \left[\sup_{ u\leq s}\left\vert \theta_{ 1, s} - \theta_{ 2, s}\right\vert^{2 k}\right] {\rm d}s\\
&+ C \int_{ 0}^{t} \mathbf{ E} \left[\sup_{ u\leq s}\left\vert \theta_{ 1, u} - \theta_{ 2, u}\right\vert^{2k}\right] \left\lbrace \int \left\vert u- v\right\vert  p_{ s}^{ y}({\rm d}u, {\rm d}v) W(x, y) \ell({\rm d}y)\right\rbrace^{k}{\rm d}s,\\
&\leq C \left(1+ \left\Vert \mathcal{ W}_{ 1} \right\Vert_{ \infty}\right)^{ k}\int_{ 0}^{t} \mathbf{ E} \left[\sup_{ u\leq s}\left\vert \theta_{ 1, s} - \theta_{ 2, s}\right\vert^{ 2k}\right] {\rm d}s\\
&+ C \int_{ 0}^{t} \left\lbrace \sup_{ u\leq s}\int \left\vert u- v\right\vert  p_{ u}^{ y}({\rm d}u, {\rm d}v) W(x, y) \ell({\rm d}y)\right\rbrace^{2k}{\rm d}s.
\end{align*}
Applying Jensen's inequality to the probability measure (for fixed $x\in I$) $ \frac{ W(x, y)}{ \int W(x, z) \ell({\rm d}z)} \ell({\rm d}y)$ (recall \eqref{eq:W_nondegenerate}), we can bound the last term above by
\begin{align*}
C \left\Vert \mathcal{ W}_{ 1} \right\Vert_{ \infty}^{ 2k-1}\int_{ 0}^{t} \sup_{ u\leq s}\int \left\vert u- v\right\vert^{ 2k}  p_{ u}^{ y}({\rm d}u, {\rm d}v) W(x, y) \ell({\rm d}y){\rm d}s.
\end{align*}
Since this is true for all coupling, we obtain
\begin{align*}
\mathbf{ E} \left[\sup_{ s\leq t}\left\vert \theta_{ 1, s} - \theta_{ 2, s}\right\vert^{ 2k}\right] 
&\leq  C \left(1+ \left\Vert \mathcal{ W}_{ 1} \right\Vert_{ \infty}\right)^{k}\int_{ 0}^{t} \mathbf{ E} \left[\sup_{ u\leq s}\left\vert \theta_{ 1, s} - \theta_{ 2, s}\right\vert^{ 2k}\right] {\rm d}s\\
&+ C \left\Vert \mathcal{W}_{ 1} \right\Vert_{ \infty}^{ 2k}\int_{ 0}^{t}  \delta_{ s}(m_{ 1}, m_{ 2})^{ 2k} {\rm d}s.
\end{align*}
Hence, by Gr\"onwall's Lemma, we obtain finally, for constant $C$ depending on $T$ and $ \left\Vert \mathcal{ W}_{ 1} \right\Vert_{ \infty}$,
\begin{align*}
\delta_{ t}( \Theta(m_{ 1}), \Theta(m_{ 2}))^{ 2k} \leq C \int_{ 0}^{t} \delta_{ s}(m_{ 1}, m_{ 2})^{ 2k} {\rm d}s.
\end{align*}
Iterating this estimate gives, for all $l\geq1$, $\delta_{ T}( \Theta^{ l+1}( \nu_{ 0}), \Theta^{ l}( \nu_{ 0}))^{ 2k} \leq C^{ l} \frac{ T^{ l}}{ l!} \delta_{ T}( \Theta(\nu_{ 0}), \nu_{ 0})^{ 2k}$, so that $ \left( \Theta^{ l}(\nu_{ 0})\right)$ is a Cauchy sequence, and hence, converging to $ \nu$, solution to \eqref{eq:McKeanVlasov}.
\subsection{Uniqueness of a solution to the nonlinear Fokker-Planck PDE}
\label{sec:uniqueness_PDE_proof}
We now turn to the uniqueness part of Proposition~\ref{prop:PDE_wellposed}. Let $ \mu = \left\lbrace \mu_{ t}\right\rbrace_{ t\in[0, T]}$ be any other weak solution to \eqref{eq:FP} in $ \mathcal{ M}$ such that $ \nu_{ 0}= \mu_{ 0}$. The point is to prove that $\mu_{ t}= \nu_{ t}$ for $t\in[0, T]$.
\begin{definition}
\label{def:phi_theta_x}
For $m\in \left\lbrace \nu, \mu\right\rbrace$, for $\ell$-almost every $x\in I$, for any $0\leq s\leq t$, any $ \theta\in \mathbb{ R}^{ d}$, denote by $ \left\lbrace \varphi_{ s}^{m, t}( \theta, x)\right\rbrace_{ s\leq t \leq T}$ the unique solution of
\begin{equation}
\label{eq:diff_nu}
{\rm d} \vartheta_t^{m,x}\, =\, \left(c(\vartheta_t^{m,x}) + \int \Gamma \left( \vartheta_{ t}^{m, x}, \tilde{ \theta}\right) W(x, y) m_{ t}({\rm d} \tilde{ \theta}, {\rm d} y)\right){\rm d} t+ \sigma{\rm d} B_t\, , 
\end{equation}
with position $x$ and initial condition $ \varphi_{s}^{m, s}= \theta$ at $t=s$. Define finally, for any test function $f:\theta\mapsto f(\theta)$, $ \theta\in \mathbb{ R}^{ d}$ and $\ell$-almost every $x\in I$, 
\begin{equation}
\label{eq:propagator}
P_{ s, t}f( \theta, x):= \mathbf{ E} f \left( \varphi_{ s}^{ \nu, t}(\theta, x)\right).
\end{equation}
\end{definition}
Let us suppose that $c$ is uniformly Lispchitz on $ \mathbb{ R}^{ d}$ (one can remove this assumption by replacing $c(\cdot)$ by its Yosida approximation, we refer to \cite{LucSta2014}, Section~7, where the same procedure is carried out). Under the assumptions made on the model, the propagator $P$ satisfies the following Backward Kolmogorov equation (see \cite{doi:10.1080/07362999808809576}, Remark 2.3): for $ \theta\in \mathbb{ R}^{ d}$ and $x\in I$,
\begin{multline}
\label{eq:kolmogorov}
\partial_{ s} P_{ s, t}f + \frac{ 1}{ 2} \nabla_{ \theta}\cdot \left( \sigma \sigma^{ \dagger} \nabla_{ \theta} P_{ s, t}f\right)+ \left(  \left( c(\theta) + \int \Gamma(\theta, \tilde{ \theta}) W(x, y)\nu_{ t}({\rm d} \tilde{ \theta}, {\rm d} y)\right)\cdot \nabla_{ \theta}\right)P_{ s, t}f=0.
\end{multline}
For any regular test function $f$, applying Ito formula to $ t \mapsto P_{ t, T}f( \vartheta_{ t}^{ \mu, x}, x)$, where $ \vartheta^{ \mu, x}$ solves \eqref{eq:diff_nu} for the choice of $m= \mu$, gives:
\begin{align}
P_{ s, t}f(\vartheta_{ s}^{ \mu, x}, x) &= P_{ 0, t}f(\vartheta_{ 0}^{ \mu, x}, x) + \int_{ 0}^{s} \partial_{ v} P_{ v, t}f(\vartheta_{ v}^{ \mu, x}, x) {\rm d} v \nonumber\\
&+ \int_{ 0}^{s} \nabla_{ \theta} P_{ v, t}f( \vartheta_{ v}^{ \mu, x}, x) \cdot \left( c(\vartheta_{ v}^{ \mu, x}) + \int \Gamma \left( \vartheta_{ v}^{ \mu, x}, \tilde{ \theta}\right) W(x, y) \mu_{ v}({\rm d} \tilde{ \theta}, {\rm d}y)\right){\rm d} v \nonumber\\
&+ \frac{1}{ 2} \int_{ 0}^{s} \nabla_{ \theta}\cdot \left( \sigma \sigma^{\dagger}\nabla_{ \theta}P_{ v, t}f\right)(\vartheta_{ v}^{ \mu, x}, x) {\rm d} v + \int_{ 0}^{s} \nabla_{ \theta} P_{ v, t}f(\vartheta_{ v}^{ \mu, x}, x)\cdot\sigma{\rm d}B_{ v}.
\end{align}
Using \eqref{eq:kolmogorov}, this simplifies into
\begin{multline}
P_{ s, t}f(\vartheta_{ t}^{ \mu, x}, x) = P_{ 0, t}f( \vartheta_{ 0}^{ \mu, x}, x) +\int_{ 0}^{s} \nabla_{ \theta} P_{ v, t}f(\vartheta_{ v}^{ \mu, x}, x) \cdot\sigma{\rm d}B_{ v}\\
+ \int_{ 0}^{s} \nabla_{ \theta} P_{ v, t}f(\vartheta_{ v}^{ \mu, x}, x) \cdot  \int \Gamma \left( \vartheta_{ v}^{ \mu, x}, \tilde{ \theta}\right)W(x, y) (\mu_{ v}({\rm d} \tilde{ \theta}, {\rm d} y)-  \nu_{ v}({\rm d} \tilde{ \theta}, {\rm d} y)){\rm d} v.
\end{multline}
Since the law of $\vartheta_{ v}^{ \mu, x}$ is $ \mu^{ x}_{ v}({\rm d}\theta)$,  taking the expectation w.r.t. the Brownian motion, we obtain for $s=t$ (recall that $P_{ s, s}f=f$), 
\begin{multline}
\int f(\theta) \mu_{ s}^{ x}({\rm d}\theta) = \int P_{ 0, s}f(\theta, x) \mu_{ 0}^{ x}({\rm d}\theta)\\
+ \int_{ 0}^{s} \int\left\lbrace \int \nabla_{ \theta} P_{ v, s}f(\theta, x) \cdot \Gamma \left(\theta, \tilde{ \theta}\right)W(x, y) \mu_{ v}^{ x}({\rm d}\theta) \right\rbrace(\mu_{ v}({\rm d} \tilde{ \theta}, {\rm d} y)-  \nu_{ v}({\rm d} \tilde{ \theta}, {\rm d} y)){\rm d} v.
\end{multline}
Furthermore, taking in \eqref{eq:McKeanVlasov} test functions of the form $ \varphi(\theta, x)= \phi(\theta) \psi(x)$
shows that for every regular test function $ \theta \mapsto \phi(\theta)$, for $\ell$-almost every $x$,
\begin{align}
\int \phi(\theta)\nu_{s}^{ x}({\rm d}\theta) &= \int \phi(\theta) \nu_{0}^{ x}({\rm d}\theta)+ \int_{ 0}^{s} \int \left\lbrace\frac{ 1}{ 2}\nabla_{ \theta}\left( \sigma \sigma^{ \dagger} \nabla_{ \theta}\phi(\theta) \right) + \nabla_{ \theta} \phi(\theta)\cdot c(\theta)\right\rbrace \nu_{v}^{ x}({\rm d}\theta) {\rm d} v \nonumber\\ &+ \int_{ 0}^{s} \int   \nabla_{ \theta} \phi(\theta) \cdot\int W(x, y)\Gamma(\theta, \tilde \theta) \nu_{ v}({\rm d} \tilde \theta, {\rm d} y) \nu_{ v}^{ x}({\rm d} \theta){\rm d} v.
\end{align}
From this, we get
\begin{equation}
\partial_{ v} \int P_{ v, t}f(\theta, x)\nu_{ v}^{ x}({\rm d}\theta) = \int \partial_{ v}P_{ v, t}f(\theta, x) \nu_{ v}^{ x}({\rm d}\theta) + \int P_{ v, t}f(\theta, x)\partial_{ v}\nu_{ v}^{ x}({\rm d}\theta) =0.
\end{equation}
We obtain finally, for $\ell$-almost every $x$,
\begin{multline}
\label{eq:propagator_mu_nu}
\int f(\theta) \left\lbrace\mu_{ s}^{ x}({\rm d}\theta) - \nu_{ s}^{ x}({\rm d}\theta)\right\rbrace = \int P_{ 0, s}f(\theta, x) \left\lbrace\mu_{ 0}^{ x}({\rm d}\theta)- \nu_{ 0}^{ x}({\rm d}\theta)\right\rbrace\\
+ \int_{ 0}^{s} \int\left\lbrace \int \nabla_{ \theta} P_{ v, s}f(\theta, x) \cdot \Gamma \left(\theta, \tilde{ \theta}\right)W(x, y) \mu_{ v}^{ x}({\rm d}\theta) \right\rbrace(\mu_{ v}^{ y}({\rm d} \tilde{ \theta})-  \nu_{ v}^{ y}({\rm d} \tilde{ \theta}))\ell({\rm d}y){\rm d} v.
\end{multline}
Let us recall the definition of the Wasserstein metric $w_{ 1}(\cdot, \cdot)$ in \eqref{eq:wasserstein}. By the Kantorovich-Rubinstein duality, an equivalent expression of this distance is
\begin{equation}
w_{ 1}( \mu, \nu)= \sup_{ \left\Vert f \right\Vert_{ Lip}\leq 1} \left\vert \int f {\rm d} \mu - \int f {\rm d} \nu \right\vert.
\end{equation}
An important point is to note that there exists a constant $C>0$ such that, uniformly in $(\theta, x)$, $ \left\vert \nabla_{ \theta}P_{ v, s}f(\theta, x) \right\vert\leq C$, where $C$ is uniform on $v, s\in[0, T]$ and $f$ such that $ \left\Vert f \right\Vert_{ Lip}\leq1$ (see \cite{LucSta2014}, Lemma~4.4 for more details). Thus, for fixed $x, y\in I$ and $v,s\in [0, T]$, the function $ \tilde{ \theta} \mapsto I_{ x, y}(\tilde{ \theta}):=\int \nabla_{ \theta} P_{ v, s}f(\theta, x) \cdot \Gamma \left(\theta, \tilde{ \theta}\right)W(x, y) \mu_{ v}^{ x}({\rm d}\theta)$ is Lipschitz:
\begin{align*}
\left\vert I_{ x, y}(\theta_{ 1}) - I_{ x, y}(\theta_{ 2})\right\vert&= \left\vert \int \nabla_{ \theta} P_{ v, s}f(\theta, x) \cdot \left(\Gamma \left(\theta, \theta_{ 1}\right) - \Gamma \left(\theta, \theta_{ 2}\right)\right)W(x, y) \mu_{ v}^{ x}({\rm d}\theta)\right\vert\\&\leq C L_{ \Gamma} W(x, y)\left\vert \theta_{ 1}- \theta_{ 2} \right\vert
\end{align*}
Hence, we obtain
\begin{multline}
\left\vert \int f(\theta) \left\lbrace\mu_{ s}^{ x}({\rm d}\theta) - \nu_{ s}^{ x}({\rm d}\theta)\right\rbrace \right\vert \leq \left\vert \int P_{ 0, s}f(\theta, x) \left\lbrace\mu_{ 0}^{ x}({\rm d}\theta)- \nu_{ 0}^{ x}({\rm d}\theta)\right\rbrace \right\vert \\
+ C L_{ \Gamma}\int_{ 0}^{s} \int \sup_{ u\leq v}w_{ 1}(\mu_{ u}^{ y}, \nu_{ u}^{ y})W(x, y)\ell({\rm d}y){\rm d} v.
\end{multline}
Taking the supremum on $f$ with $ \left\Vert f \right\Vert_{ Lip}\leq 1$ and using the fact that $ \mu_{ 0}= \nu_{ 0}$, we obtain
\begin{equation*}
w_{ 1} \left(\mu_{ s}^{ x}, \nu_{ s}^{ x}\right) \leq C L_{ \Gamma}\int_{ 0}^{s} \int \sup_{ u\leq v}w_{ 1}(\mu_{ u}^{ y}, \nu_{ u}^{ y})W(x, y)\ell({\rm d}y){\rm d} v.
\end{equation*}
By Cauchy-Schwarz and Jensen inequalities, 
\begin{align*}
w_{ 1} \left(\mu_{ s}^{ x}, \nu_{ s}^{ x}\right)^{ 2} &\leq C^{ 2} L_{ \Gamma}^{ 2} \left(\int_{ 0}^{s}  \left(\int \sup_{ u\leq v}w_{ 1}(\mu_{ u}^{ y}, \nu_{ u}^{ y})^{ 2}\ell({\rm d}y)\right)^{ \frac{ 1}{ 2}} \left(\int W(x, y)^{ 2}\ell({\rm d}y)\right)^{ \frac{ 1}{ 2}}{\rm d} v\right)^{ 2} \\
%& \leq C^{ 2} L_{ \Gamma}^{ 2} \left(\sup_{ z}\int W(z, y)^{ 2}\ell({\rm d}y)\right) \left(\int_{ 0}^{T}  \left(\int w_{ 1}(\mu_{ s}^{ y}, \nu_{ s}^{ y})^{ 2}\ell({\rm d}y)\right)^{ \frac{ 1}{ 2}} {\rm d} s\right)^{ 2}\\
& \leq C^{ 2} L_{ \Gamma}^{ 2} T \left\Vert \mathcal{ W}_{ 2} \right\Vert_{ \infty} \int_{ 0}^{s}  \int \sup_{ u\leq v}w_{ 1}(\mu_{ u}^{ y}, \nu_{ u}^{ y})^{ 2}\ell({\rm d}y) {\rm d} v.
\end{align*}
Taking the supremum in $s\leq t$ and integrating w.r.t. $x$ gives
\begin{align*}
\int \sup_{ s\leq t}w_{ 1} \left(\mu_{ s}^{ x}, \nu_{ s}^{ x}\right)^{ 2}\ell({\rm d}x)\leq C^{ 2} L_{ \Gamma}^{ 2} T \left\Vert \mathcal{ W}_{ 2} \right\Vert_{ \infty} \int_{ 0}^{t}  \int \sup_{ u\leq v}w_{ 1}(\mu_{ u}^{ y}, \nu_{ u}^{ y})^{ 2}\ell({\rm d}y) {\rm d} v
\end{align*}
so that Gr\"onwall's lemma gives uniqueness.
\subsection{A priori estimates and spatial regularity}
We now gather some estimates concerning the solution $ \nu$ to \eqref{eq:McKeanVlasov}. Recall Definition~\ref{def:phi_theta_x}: in the following, for $ \theta\in \mathbb{ R}^{ d}$, $x\in I$, we set $ \varphi_{ s}^{ t}(\theta, x) := \varphi_{ s}^{\nu, t}( \theta, x)$ (now $m= \nu$).
\begin{lemma}
\label{lem:regul_Phi}
For all $ \theta_{ 1}, \theta_{ 2}\in \mathbb{ R}^{ d}$, $x,y\in I$, $t\leq T$, consider $ \varphi_{ 0}^{ t}(\theta_{ 1}, x)$ and $ \varphi_{ 0}^{ t}(\theta_{ 2}, y)$ coupled in such a way that they are driven by the same Brownian motion $B$ in \eqref{eq:diff_nu}. Under the assumptions of Section~\ref{sec:general_assumptions_nonlin}, there exists a constant $C_{ 1}$ depending only on $ \Gamma, \sigma, c, W, \nu_{ 0}$ such that for all $t\geq0$,
\begin{equation}
\label{eq:regul_Phi}
 \mathbf{ E} \left[\sup_{ u\in[0, t]} \left\vert  \varphi_{ 0}^{ u}(\theta_{ 1}, x) - \varphi_{ 0}^{ u}(\theta_{ 2}, y)\right\vert^{ 2}\right]\leq C_{ 1} e^{ C_{ 1}t} \left( \delta \mathcal{ W}(x, y)^{ 2}+ \left\vert \theta_{ 1}- \theta_{ 2} \right\vert^{ 2}\right),
\end{equation}
where we recall the definition of $ \delta W$ in \eqref{hyp:Int_W_Lip}.
\end{lemma}
\begin{proof}[Proof of Lemma~\ref{lem:regul_Phi}]
Set for simplicity $ \Phi_{u}:= \varphi_{ 0}^{ u}(\theta_{ 1}, x)$ and $ \Psi_{ u}:= \varphi_{ 0}^{ u}(\theta_{ 2}, y)$. Then, for all $s\leq t$,
\begin{multline}
\left\vert \Phi_{ s} - \Psi_{ s}\right\vert^{ 2} = \left\vert \theta_{ 1} - \theta_{ 2} \right\vert^{ 2} + 2 \int_{ 0}^{s} \left\langle \Phi_{ u} - \Psi_{ u}\, ,\, c( \Phi_{ u}) - c( \Psi_{ u})\right\rangle {\rm d} u\\ + 2\int_{ 0}^{s}  \left\langle \Phi_{ u} - \Psi_{ u}\, ,\, \int  \left(W(x, z)\Gamma( \Phi_{ u}, \tilde\theta)  - W(y,z)\Gamma( \Psi_{ u}, \tilde\theta)\right)\nu_{ u}({\rm d} \tilde\theta, {\rm d} z)\right\rangle{\rm d} u.
\end{multline}
Since $c$ is one-sided Lipschitz \eqref{hyp:c_onesidedLip}, taking the supremum in $s\in [0, t]$
\begin{multline}
\sup_{s\in[0, t]}\left\vert  \Phi_{ s} - \Psi_{ s}\right\vert^{ 2} \leq \left\vert \theta_{ 1} - \theta_{ 2} \right\vert^{ 2}+ (2L_{ c}+1) \int_{ 0}^{t} \sup_{ v\in [0, u]}\left\vert \Phi_{ v} - \Psi_{ v} \right\vert^{ 2} {\rm d} u\\ + \int_{ 0}^{t}  \left\vert \int  \left(W(x, z)\Gamma( \Phi_{ u},  \tilde\theta)  - W(y,z)\Gamma( \Psi_{ u}, \tilde\theta)\right)\nu_{ u}({\rm d} \tilde\theta, {\rm d} z) \right\vert^{ 2}{\rm d} u.
\end{multline}
The last term within the last integral can be estimated above by $2(a^{ (1)}_{ u} + a^{ (2)}_{ u})$ where
\begin{align*}
a^{ (1)}_{ u}&= \left\vert \int  W(x, z)\left(\Gamma( \Phi_{ u}, \tilde\theta)  - \Gamma(\Psi_{ u}, \tilde\theta)\right)\nu_{ u}({\rm d} \tilde\theta,  {\rm d} z) \right\vert^{ 2}\leq  L_{ \Gamma}^{ 2} \left\Vert \mathcal{ W} \right\Vert_{ \infty}^{ 2} \sup_{ v\in[0, u]}\left\vert \Phi_{ v} - \Psi_{ v}\right\vert^{ 2},
\end{align*}
and where
\begin{align*}
a^{ (2)}_{ u}&= \left( \int \left\lbrace\int  \Gamma( \Psi_{ u}, \tilde\theta) \nu_{ u}^{z}({\rm d} \tilde\theta)\right\rbrace  \left\vert W(x, z)  - W(y,z) \right\vert \ell({\rm d} z)\right)^{ 2}\\
&\leq  L_{ \Gamma}^{ 2} \left(1 + \left\vert \Psi_{ u} \right\vert^{ 2} + \sup_{z\in I}\mathbf{ E} \left[\sup_{ v\leq u}\left\vert \bar \theta_{ v}^{z} \right\vert^{ 2}\right]\right)\delta \mathcal{ W}(x, y)^{ 2}.
\end{align*}
Taking the expectation, using \eqref{eq:apriori_bound_nonlin} and Gr\"onwall's lemma, one obtains \eqref{eq:regul_Phi}.
\end{proof}
\begin{lemma}
\label{lem:regul_nu}
Let $f: \theta \mapsto f(\theta)$ be such that $\left\vert f(\theta) - f(\theta^{ \prime}) \right\vert \leq c^{ (1)}_{ f} \left\vert \theta- \theta^{ \prime} \right\vert$ and $\left\vert f(\theta) \right\vert \leq c^{ (2)}_{ f}\left(1 + \left\vert \theta \right\vert\right)$.
Then, under the assumptions of Section~\ref{sec:general_assumptions_nonlin}, for the same constant $C_{ 1}$ as in \eqref{eq:regul_Phi}, for all $t\geq0$, for $x, y\in[0,1]$,
\begin{equation}
\label{eq:regularity_lipschitz_nu_x}
\left\vert \int f(\theta) \nu_{ t}^{x}({\rm d} \theta) - \int f(\theta) \nu_{ t}^{y}({\rm d} \theta)\right\vert \leq c_{ f}^{ (1)}C_{ 1}e^{ C_{ 1}t} \left( \delta \mathcal{ W}(x, y)+ w_{ 1}(\nu_{ 0}^{x}, \nu_{ 0}^{y})\right).
\end{equation}
\end{lemma}
\begin{proof}[Proof of Lemma~\ref{lem:regul_nu}]
Recall Definition~\ref{def:phi_theta_x} (in particular the definition of $P$ in \eqref{eq:propagator}) and the calculations made in the proof of uniqueness in Section~\ref{sec:uniqueness_PDE_proof}. Apply \eqref{eq:propagator_mu_nu} to the case $ \mu= \nu$:
for any regular test function $f$ satisfying the hypotheses of Lemma~\ref{lem:regul_nu},
\begin{equation}
\int f(\theta) \nu_{ t}^{x}({\rm d} \theta) = \int P_{ 0, t}f(\theta, x) \nu_{ 0}^{x}({\rm d} \theta).
\end{equation}
In particular, for $x, y\in I$,
\begin{align}
 \left\vert \int f(\theta) \nu_{ t}^{x}({\rm d} \theta) - \int f(\theta) \nu_{ t}^{ y}({\rm d} \theta)\right\vert 
% &= \left\vert \int P_{ 0, t}f(\theta, x) \nu_{ 0}^{x}({\rm d} \theta) - \int P_{ 0, t}f(\theta, y) \nu_{ 0}^{y}({\rm d} \theta) \right\vert, \nonumber\\
&\leq\int  \left\vert P_{ 0, t}f(\theta, x) -  P_{ 0, t}f(\theta, y) \right\vert \nu_{ 0}^{x}({\rm d} \theta) \label{aux:Pt1}\\
&+\left\vert \int P_{ 0, t}f(\theta, y) \nu_{ 0}^{x}({\rm d} \theta) - \int P_{ 0, t}f(\theta, y) \nu_{ 0}^{y}({\rm d} \theta) \right\vert. \label{aux:Pt2}
\end{align} 
For the first term \eqref{aux:Pt1}, we have
\begin{align}
\int  \left\vert P_{ 0, t}f(\theta, x) -  P_{ 0, t}f(\theta, y) \right\vert \nu_{ 0}^{x}({\rm d} \theta) &= \int  \left\vert \mathbf{ E} \left[f( \varphi_{ 0}^{ t}(\theta, x))\right] - \mathbf{ E} \left[f( \varphi_{ 0}^{ t}(\theta, y))\right]\right\vert \nu_{ 0}^{x}({\rm d} \theta),\nonumber\\
&\leq c_{ f}^{ (1)}\int  \mathbf{ E}\left\vert \varphi_{ 0}^{ t}(\theta, x)-\varphi_{ 0}^{ t}(\theta, y)\right\vert \nu_{ 0}^{x}({\rm d} \theta)\label{aux:Pt0f}
\end{align} 
where the coupling $(\varphi_{ 0}^{ t}(\theta, x), \varphi_{ 0}^{ t}(\theta, y))$ is given by Lemma~\ref{lem:regul_Phi}. By \eqref{eq:regul_Phi}, one obtains that
\begin{equation}
\int  \left\vert P_{ 0, t}f(\theta, x) -  P_{ 0, t}f(\theta, y) \right\vert \nu_{ 0}^{x}({\rm d} \theta)\leq c_{ f}^{ (1)} C_{ 1} e^{ C_{ 1} t} \delta \mathcal{ W}(x, y).
\end{equation}
As far as the second term \eqref{aux:Pt2} is concerned, for any fixed $(\theta_{ 1}, \theta_{ 2})$, applying once again Lemma~\ref{lem:regul_Phi} gives
 $\left\vert P_{ 0, t}f(\theta_{ 1}, y) - P_{ 0, t}f(\theta_{ 2}, y) \right\vert \leq c_{ f}^{ (1)} C_{ 1} e^{ C_{ 1}t} \left\vert \theta_{ 1} - \theta_{ 2} \right\vert$.
In particular, for any coupling $(\bar \theta_{ 0}^{ x}, \bar \theta_{ 0}^{ y})$,
\begin{multline*}
\left\vert \int P_{ 0, t}f(\theta, y) \nu_{ 0}^{x}({\rm d} \theta) - \int P_{ 0, t}f(\theta, y) \nu_{ 0}^{y}({\rm d} \theta) \right\vert = \left\vert \mathbf{ E} \left[P_{ 0, t}f( \bar \theta_{ 0}^{ x}, y) - P_{ 0, t}f( \bar \theta_{ 0}^{ y}, y)\right] \right\vert,\\
\leq  \mathbf{ E} \left\vert P_{ 0, t}f( \bar \theta_{ 0}^{ x}, y) - P_{ 0, t}f( \bar \theta_{ 0}^{ y}, y)\right\vert \leq c_{ f}^{ (1)} C_{ 1} e^{ C_{ 1}t} \mathbf{ E} \left\vert \bar \theta_{ 0}^{ x} - \bar \theta_{ 0}^{ y} \right\vert.
\end{multline*}
Since this true for all coupling of the initial conditions $(\bar \theta_{ 0}^{ x}, \bar \theta_{ 0}^{ y})$, one obtains finally that
\begin{equation}
\left\vert \int P_{ 0, t}f(\theta, y) \nu_{ 0}^{x}({\rm d} \theta) - \int P_{ 0, t}f(\theta, y) \nu_{ 0}^{y}({\rm d} \theta) \right\vert \leq c_{ f}^{ (1)} C_{ 1} e^{ C_{ 1}t} w_{ 1}( \nu_{ 0}^{x}, \nu_{ 0}^{y})
\end{equation}
This concludes the proof of Lemma~\ref{lem:regul_nu}.
\end{proof}
\begin{lemma}
\label{lem:regul_Upsilon}
Recall the definition of $ \left[ \Gamma\right]_{ u}$ in \eqref{eq:G} and of $ \Upsilon_{ u}$ in \eqref{eq:upsilon_t} ($u\geq0$). Under the assumptions of Section~\ref{sec:general_assumptions_nonlin}, there exists a constant $C>0$ (that depends only on $T$, $c$, $ \Gamma$ and $ \sigma$) such that for all $u\in [0, T]$, $ \theta\in \mathbb{ R}^{ d}$, $x,y, y^{ \prime}, z, z^{ \prime} \in I$, 
\begin{align}
\sup_{ u\in [0, T]} \sup_{x\in I} \left\vert \left[ \Gamma\right]_{ u}(\theta, x) \right\vert&\leq C (1+ \left\vert \theta \right\vert),\label{eq:bound_Gamma_u}\\
\sup_{ u\in [0, T]} \sup_{ x,y,z\in I} \left\vert \Upsilon_{ u}(x,y,z) \right\vert &\leq C
\end{align}
as well as
\begin{align}
\sup_{ u\in[0, T]}\left\vert [\Gamma]_{u}(\theta, x) - [\Gamma]_{u}(\theta, y) \right\vert &\leq C\left(\delta \mathcal{ W}(x, y)+ w_{ 1}(\nu_{ 0}^{x}, \nu_{ 0}^{y})\right),\label{eq:bound_Gamma_diff}\\
\sup_{ u\in [0, T]}\sup_{ x, z\in I}\left\vert  \Upsilon_{ u}(x, y, z)-\Upsilon_{ u}(x, y^{ \prime}, z) \right\vert &\leq C \left( \delta \mathcal{ W}(y, y^{ \prime}) + w_{ 1}(\nu_{ 0}^{ y},\nu_{ 0}^{ y^{ \prime}})\right),\\
\sup_{ u\in[0, T]}\sup_{ x, y\in I}\left\vert  \Upsilon_{ u}(x, y, z)-\Upsilon_{ u}(x, y, z^{ \prime}) \right\vert &\leq C \left( \delta \mathcal{ W}(z, z^{ \prime}) + w_{ 1}(\nu_{ 0}^{ z},\nu_{ 0}^{ z^{ \prime}})\right).
\end{align}
\end{lemma}
\begin{proof}[Proof of Lemma~\ref{lem:regul_Upsilon}]
The estimates on $ \Upsilon$ are an easy consequence of the estimates on $ \left[ \Gamma\right]$. The bound \eqref{eq:bound_Gamma_u} on $ \left[ \Gamma\right]$ is a direct consequence of \eqref{hyp:Gamma_bound} and the uniform estimates we have on $ \bar \theta$ \eqref{eq:apriori_bound_nonlin}. Let us now prove \eqref{eq:bound_Gamma_diff}: apply the results of Lemma~\ref{lem:regul_nu} to the test function $f:= \tilde{ \theta} \mapsto \Gamma(\theta, \tilde{ \theta})$ (for fixed $\theta\in \mathbb{ R}^{ d}$). The test function $f$ satisfies the hypothesis of Lemma~\ref{lem:regul_nu} for $c_{ f}^{ (1)}:= L_{ \Gamma}$ (and some $c_{ f}^{ (2)}$ that depends on $ \theta$, but note that $c_{ f}^{ (2)}$ does not enter into account in the estimates of Lemma~\ref{lem:regul_nu}). In particular, uniformly in $\theta$,
\begin{align*}
\left\vert [\Gamma]_{u}(\theta, x) - [\Gamma]_{u}(\theta, y) \right\vert &=  \left\vert \int \Gamma(\theta, \tilde{ \theta}) \nu_{ u}^{x}({\rm d} \tilde{ \theta}) - \int \Gamma(\theta, \tilde{ \theta}) \nu_{ u}^{y}({\rm d} \tilde{ \theta}) \right\vert \leq C \left( \delta \mathcal{ W}(x, y)+ w_{ 1}(\nu_{ 0}^{x}, \nu_{ 0}^{y})\right),
\end{align*}
which gives the result.
\end{proof}
\section{Identification in the general case: proof of Theorem~\ref{theo:identification_non_compact}}
\label{sec:identification_noncompact_case}
The point of this section is to prove Theorem~\ref{theo:identification_non_compact}. Recall that $I= \mathbb{ R}^{ p}$ endowed with a probability measure $\ell({\rm d}x) := \ell(x) {\rm d}x$ with a $ \mathcal{ C}^{ 1}$ density $\ell$. We fix a $ \mathcal{ C}^{ 1}$-kernel $W(x, y)$ on $ \mathbb{ R}^{ p}\times \mathbb{ R}^{ p}$. We proceed by truncation from the compact case (Theorem~\ref{theo:identification}): fix $M>0$, define
\begin{equation}
B_{ M}:=[-M, M]^{ p} \text{ and }\Lambda_{ M}:= B_{ M}\times B_{ M},
\end{equation}
and introduce the following probability measure, whose support is $B_{ M}$:
\begin{equation}
\ell^{(M)}({\rm d}x):= \ell^{ (M)}(x) {\rm d}x:= \ell(x)\frac{ \mathbf{ 1}_{ B_{ M}}(x)}{ \ell(B_{ M})} {\rm d}x.
\end{equation}
In what follows, we choose $M$ sufficiently large so that 
\begin{equation}
\label{eq:ell_BM}
\ell(B_{ M}) \geq \frac{ 1}{ 2}
\end{equation}
What has been done in Section~\ref{sec:ident_compact_case} for $I=[0,1]$ with deterministic regular positions can be transposed without difficulties (up to obvious notational changes) to $I=B_{ M}$, endowed with its renormalized Lebesgue measure $ \frac{ {\rm d}x}{ \left\vert B_{ M} \right\vert}$, where $ \left\vert B_{ M} \right\vert= (2M)^{ p}$. Hence, we can apply the result of Section~\ref{sec:ident_compact_case} for $I=B_{ M}$ and deterministic positions, for the choice of kernel on $B_{ M}^{ 2}$
\begin{equation}
\label{eq:tilde_WM}
\tilde{W}(x, y):=\tilde{W}_{ M}(x, y):= W(x, y)\ell^{ (M)}(y) \left\vert B_{ M} \right\vert,\ x, y\in B_{ M}.
\end{equation}
Indeed, the kernel $(x, y) \mapsto \tilde{W}_{ M}(x, y)$ is bounded and $ \mathcal{ C}^{ 1}$ on $B_{ M}^{ 2}$ so that $\tilde{W}_{ M}$ satisfies the assumptions of Section~\ref{sec:general_assumptions_nonlin}. Moreover, we see from Section~\ref{sec:dense_bounded_graphons} that $\tilde{W}_{ M}$ can be realized as the macroscopic limit of a graph $(\mathcal{G}^{ (n)}, \kappa^{(n)})$ constructed as in \eqref{eq:generic_Wn_P_bounded} and \eqref{eq:kappan_uniform_P}. Since $W$ and $\ell$ are regular, Definitions~\ref{def:convergence_graph} and~\ref{def:graphon_regular} are satisfied. Hence, both solutions $ \nu^{ (M)}$ to \eqref{eq:McKeanVlasov} (with initial condition $ \nu_{ 0}^{ (M)}({\rm d}\theta, {\rm d}x)= \nu_{ 0}^{ x}({\rm d}\theta) \ell^{ (M)}({\rm d}x)$) and $ \psi^{ (M)}$ to \eqref{eq:heat_equation_weak} (with initial condition $ \psi_{ 0}^{ (M)}(x)= \int \theta \nu_{ 0}^{ x, (M)}({\rm d}\theta)$) in the case $I=B_{ M}$ endowed with $ \frac{ {\rm d}x}{ \left\vert B_{ M} \right\vert}$, for the kernel $\tilde{W}_{ M}$ are well posed and satisfy the identification 
\begin{equation}
\label{eq:ident_psiM_nuM}
\int_{B_{ M}} \left\langle \psi^{ (M)}(x, t)\, ,\, J(x)\right\rangle \frac{ {\rm d}x}{ \left\vert B_{ M} \right\vert}= \int_{B_{ M}}\int \left\langle \theta\, ,\, J(x)\right\rangle \nu_{t}^{(M), x}({\rm d}\theta) \frac{ {\rm d}x}{ \left\vert B_{ M} \right\vert},
\end{equation}
for all regular test functions $J$ on $B_{ M}$. The point of the remaining is to make $M\to\infty$ in \eqref{eq:ident_psiM_nuM}. We treat the two sides of \eqref{eq:ident_psiM_nuM} separately. Concerning the lefthand, $ \psi^{ (M)}$ is the unique weak solution to
\begin{multline}
\label{eq:psiM}
\int_{B_{ M}} \left\langle \psi^{ (M)}(x, t)\, ,\, J(x)\right\rangle \frac{ {\rm d}x}{ \left\vert B_{ M} \right\vert}= \int_{B_{ M}} \left\langle \psi_{ 0}^{ (M)}(x)\, ,\, J(x)\right\rangle \frac{ {\rm d}x}{ \left\vert B_{ M} \right\vert} + \int_{ 0}^{t} \int_{B_{ M}} \left\langle c( \psi^{ (M)}(x, s))\, ,\, J(x)\right\rangle \frac{ {\rm d}x}{ \left\vert B_{ M} \right\vert} {\rm d}s \\+ \int_{ 0}^{t}\int_{B_{ M}^{ 2}} \left\langle \Gamma( \psi^{ (M)}(x, s) , \psi^{ (M)}(y, s))\, ,\, J(x)\right\rangle \tilde{W}_{ M}(x, y)\frac{ {\rm d}y}{ \left\vert B_{ M} \right\vert} \frac{ {\rm d}x}{ \left\vert B_{ M} \right\vert}{\rm d}s.
\end{multline}
Multiplying everything by $ \left\vert B_{ M} \right\vert$ and choosing test functions of the form $J(x) \ell^{ (M)}(x)$ gives, by definition of $\tilde{W}_{ M}$,
\begin{multline}
\int_{ \mathbb{ R}^{ p}} \frac{ {\rm d}}{ {\rm d}t} \left\langle \psi^{ (M)}(x, t)\, ,\,  J(x)\right\rangle\ell^{ (M)}({\rm d} x)=  \int_{ \mathbb{ R}^{ p}} \left\langle c( \psi^{ (M)}(x, t))\, ,\, J(x)\right\rangle\ell^{ (M)}({\rm d}x)\\
+ \int_{ \mathbb{ R}^{ p}}\int_{ \mathbb{ R}^{ p}} \left\langle \Gamma( \psi^{ (M)}(x, t) , \psi^{ (M)}(y, t))\, ,\, J(x)\right\rangle  W(x, y)\ell^{ (M)}({\rm d}y)\ell^{ (M)}({\rm d}x). \label{eq:psiM_2}
\end{multline}
We first give some a priori bound on $ \psi^{ (M)}$: by density, \eqref{eq:psiM_2} is also true for all test functions $J(x,t)$ and for $J(x,t)= \left\vert \psi^{ (M)}(x, t) \right\vert^{ 2k-2} \psi^{ (M)}(x,t)$, we obtain, using the properties on $c$ and $ \Gamma$
\begin{multline}
\frac{ 1}{ 2k}\frac{ {\rm d}}{ {\rm d}t}\int_{ \mathbb{ R}^{ p}} \left\vert \psi^{ (M)}(x, t) \right\vert^{ 2k}\ell^{ (M)}({\rm d} x)=  \int_{ \mathbb{ R}^{ p}} \left\langle c( \psi^{ (M)}(x, t))\, ,\, \psi^{ (M)}(x, t)\right\rangle \left\vert \psi^{ (M)}(x, t) \right\vert^{ 2k-2} \ell^{ (M)}({\rm d}x)\\
+ \int_{ \mathbb{ R}^{ 2p}} \left\langle \Gamma( \psi^{ (M)}(x, t) , \psi^{ (M)}(y, t))\, ,\, \psi^{ (M)}(x,t)\right\rangle  \left\vert \psi^{ (M)}(x, t) \right\vert^{ 2k-2} W(x, y)\ell^{ (M)}({\rm d}y)\ell^{ (M)}({\rm d}x),\\
\leq \left(L_{ c} + \frac{1}{ 2}\right) \int_{ \mathbb{ R}^{ p}} \left\vert \psi^{ (M)}(x, t)) \right\vert^{ 2k}\ell^{ (M)}({\rm d}x) + \frac{ \left\vert c(0) \right\vert^{ 2}}{ 2}\int_{ \mathbb{ R}^{ p}} \left\vert \psi^{ (M)}(x, t)) \right\vert^{ 2k-2}\ell^{ (M)}({\rm d}x) \\
+ L_{ \Gamma}\int_{ \mathbb{ R}^{ 2p}} \left(\left\vert \psi^{ (M)}(x, t) \right\vert^{ 2k-1} + \left\vert \psi^{ (M)}(x, t) \right\vert^{ 2k} \right) W(x, y)\ell^{ (M)}({\rm d}y)\ell^{ (M)}({\rm d}x)\\
+ L_{ \Gamma} \int_{ \mathbb{ R}^{ 2p}} \left\vert \psi^{ (M)}(y, t) \right\vert  \left\vert \psi^{ (M)}(x, t) \right\vert^{ 2k-1} W(x, y)\ell^{ (M)}({\rm d}y)\ell^{ (M)}({\rm d}x). \label{aux:bound_psi_M}
\end{multline}
The last term in \eqref{aux:bound_psi_M} is bounded by
\begin{multline*}
%\int_{ \mathbb{ R}^{ 2p}} \left\vert \psi^{ (M)}(y, t) \right\vert  \left\vert \psi^{ (M)}(x, t) \right\vert^{ 2k-1} W(x, y)\ell^{ (M)}({\rm d}y)\ell^{ (M)}({\rm d}x)\\
\int_{ \mathbb{ R}^{p}} \left\vert \psi^{ (M)}(x, t) \right\vert^{ 2k-1} \left(\int_{ \mathbb{ R}^{ p}} \left\vert \psi^{ (M)}(y, t) \right\vert^{ 2} \ell^{ (M)}({\rm d}y)\right)^{ \frac{ 1}{ 2}} \left(\int_{ \mathbb{ R}^{ p}} W(x, y)^{ 2}\ell^{ (M)}({\rm d}y)\right)^{ \frac{ 1}{ 2}}\ell^{ (M)}({\rm d}x),\\
\leq \sqrt{ 2}\left\Vert \mathcal{ W}_{ 2} \right\Vert_{ \infty}^{ \frac{ 1}{ 2}} \left(\int_{ \mathbb{ R}^{ p}} \left\vert \psi^{ (M)}(y, t) \right\vert^{ 2} \ell^{ (M)}({\rm d}y)\right)^{ \frac{ 1}{ 2}} \int_{ \mathbb{ R}^{p}} \left\vert \psi^{ (M)}(x, t) \right\vert^{ 2k-1}  \ell^{ (M)}({\rm d}x),
\end{multline*}
where we used \eqref{eq:ell_BM} and \eqref{hyp:bound_W_L1}. Applying Young's inequality for $p=2k$ and $q=p^{ \ast}= \frac{ 2k}{ 2k-1}$, the last quantity is smaller than
\begin{multline*}
%\int_{ \mathbb{ R}^{ 2p}} \left\vert \psi^{ (M)}(y, t) \right\vert  \left\vert \psi^{ (M)}(x, t) \right\vert^{ 2k-1} W(x, y)\ell^{ (M)}({\rm d}y)\ell^{ (M)}({\rm d}x)\\
\sqrt{2}\left\Vert \mathcal{ W}_{ 2} \right\Vert_{ \infty}^{ \frac{ 1}{ 2}} \left\lbrace \frac{ 1}{ 2k}\left(\int_{ \mathbb{ R}^{ p}} \left\vert \psi^{ (M)}(y, t) \right\vert^{ 2} \ell^{ (M)}({\rm d}y)\right)^{ k} + \frac{ 2k-1}{ 2k}\left(\int_{ \mathbb{ R}^{p}} \left\vert \psi^{ (M)}(x, t) \right\vert^{ 2k-1}  \ell^{ (M)}({\rm d}x)\right)^{ \frac{ 2k}{ 2k-1}}\right\rbrace,\\
\leq \sqrt{2}\left\Vert \mathcal{ W}_{ 2} \right\Vert_{ \infty}^{ \frac{ 1}{ 2}} \left\lbrace \frac{ 1}{ 2k}\int_{ \mathbb{ R}^{ p}} \left\vert \psi^{ (M)}(y, t) \right\vert^{ 2k} \ell^{ (M)}({\rm d}y) + \frac{ 2k-1}{ 2k}\int_{ \mathbb{ R}^{p}} \left\vert \psi^{ (M)}(x, t) \right\vert^{2k}  \ell^{ (M)}({\rm d}x)\right\rbrace,\\
= \sqrt{2}\left\Vert \mathcal{ W}_{ 2} \right\Vert_{ \infty}^{ \frac{ 1}{ 2}} \int_{ \mathbb{ R}^{ p}} \left\vert \psi^{ (M)}(y, t) \right\vert^{ 2k} \ell^{ (M)}({\rm d}y),
\end{multline*}
by Jensen's inequality. Using this bound in \eqref{aux:bound_psi_M}, the fact that there are constants $c_{ 0}, c_{ 1}>0$ such that $ \left\vert x \right\vert^{ 2k-2} + \left\vert x \right\vert^{ 2k-1} + \left\vert x \right\vert^{ 2k}\leq c_{ 0} + c_{ 1} \left\vert x \right\vert^{ 2k}$ and Gr\"onwall's lemma gives 
\begin{equation}
\label{eq:uniform_bound_psiM}
\sup_{ t\in [0, T]} \sup_{ M\geq1} \int_{ \mathbb{ R}^{ p}} \left\vert \psi^{ (M)}(x, t) \right\vert^{ 2k} \ell^{ (M)}({\rm d}x)<+\infty.
\end{equation}
We now prove that $( \psi^{ (M)} \mathbf{ 1}_{ B_{ M}})$ is Cauchy: for $N> M$, set 
\begin{equation}
\rho(x,t)=\rho_{ N, M}(x, t):= \psi^{ (N)}(x, t)\mathbf{ 1}_{ B_{ N}}(x) - \psi^{ (M)}(x, t) \mathbf{ 1}_{ B_{ M}}(x).
\end{equation}
Multiplying by $\ell(B_{ M})$ in \eqref{eq:psiM_2}, we have,
\begin{multline*}
\int_{ \mathbb{ R}^{ p}} \frac{ {\rm d}}{ {\rm d}t} \left\langle \psi^{ (M)}(x, t)  \mathbf{ 1}_{ B_{ M}}(x)\, ,\,  J(x, t)\right\rangle\ell({\rm d} x)=  \int_{ \mathbb{ R}^{ p}} \left\langle c( \psi^{ (M)}(x, t))\mathbf{ 1}_{ B_{ M}}(x)\, ,\, J(x, t)\right\rangle \ell({\rm d}x)\\
+ \frac{ 1}{ \ell(B_{ M})}\int_{ \mathbb{ R}^{ 2p}} \left\langle \Gamma( \psi^{ (M)}(x, t) , \psi^{ (M)}(y, t))\, ,\, J(x, t)\right\rangle  W(x, y)\mathbf{ 1}_{ \Lambda_{ M}}(x, y)\ell({\rm d}y)\ell({\rm d}x).
\end{multline*}
Use the notation $\tilde{c}(\theta)= c(\theta) -c(0)$. Note that $\tilde{c}$ also satisfies \eqref{hyp:c_onesidedLip} for the same constant $L_{ c}$. Since $\tilde{c}(0)=0$, one has that for all indicator function $ x \mapsto \mathbf{ 1}_{ B}(x)$, $\tilde{c}(\theta) \mathbf{ 1}_{ B}(x) = \tilde{c}(\theta \mathbf{ 1}_{ B}(x))$. In a same way, one has for all $ \theta_{ 1}, \theta_{ 2}$, $ \Gamma \left(\theta_{ 1}, \theta_{ 2}\right) \mathbf{ 1}_{ B^{ 2}}(x, y)= \Gamma \left(\theta_{ 1}\mathbf{ 1}_{ B}(x), \theta_{ 2}\mathbf{ 1}_{ B}(y)\right) \mathbf{ 1}_{ B^{ 2}}(x, y)$. Hence, we can write
\begin{multline*}
\int_{ \mathbb{ R}^{ p}} \frac{ {\rm d}}{ {\rm d}t} \left\langle \psi^{ (M)}(x, t)  \mathbf{ 1}_{ B_{ M}}(x)\, ,\,  J(x, t)\right\rangle\ell({\rm d} x)=  \int_{ \mathbb{ R}^{ p}} \left\langle \tilde{c}( \psi^{ (M)}(x, t)\mathbf{ 1}_{ B_{ M}}(x))\, ,\, J(x, t)\right\rangle \ell({\rm d}x)\\
+ \int_{ \mathbb{ R}^{ p}} \left\langle c( 0)\, ,\, J(x, t)\right\rangle \mathbf{ 1}_{ B_{ M}}(x) \ell({\rm d}x)\\
+ \frac{ 1}{ \ell(B_{ M})}\int_{\mathbb{ R}^{ 2p}} \left\langle \Gamma( \psi^{ (M)}(x, t)\mathbf{ 1}_{ B_{ M}}(x) , \psi^{ (M)}(y, t)\mathbf{ 1}_{ B_{ M}}(y))\, ,\, J(x, t)\right\rangle  W(x, y)\mathbf{ 1}_{ \Lambda_{ M}}(x, y)\ell({\rm d}y)\ell({\rm d}x).
\end{multline*}
Taking the difference between $ \psi^{ (M)}$ and $ \psi^{ (N)}$ gives for $J(x, t)= \rho(x, t) \left\vert \rho(x, t) \right\vert^{ k-2}$:
\begin{multline}
\frac{ {\rm d}}{ {\rm d}t}\int_{ \mathbb{ R}^{ p}} \left\vert \rho(x, t) \right\vert^{ k}\ell({\rm d} x)=   \int_{ \mathbb{ R}^{ p}} \left\langle \tilde{c}( \psi^{ (N)}(x, t)\mathbf{ 1}_{ B_{ N}}(x)) - \tilde{c}( \psi^{ (M)}(x, t)\mathbf{ 1}_{ B_{ M}}(x)) \, ,\, J(x,t)\right\rangle   \ell({\rm d}x)\\
+\int_{ \mathbb{ R}^{ p}} \left\langle c(0)\, ,\, J(x,t)\right\rangle \mathbf{ 1}_{ B_{ N}\setminus B_{ M}}(x)\ell({\rm d}x)\\
+\frac{ 1}{ \ell(B_{ M})}\int_{\mathbb{ R}^{ 2p}} \left\langle \Gamma( \psi^{ (N)}(x, t)\mathbf{ 1}_{ B_{ N}}(x) , \psi^{ (N)}(y, t)\mathbf{ 1}_{ B_{ N}}(y))\, ,\, J(x, t)\right\rangle  W(x, y)\mathbf{ 1}_{ \Lambda_{ N}\setminus \Lambda_{ M}}(x, y)\ell({\rm d}y)\ell({\rm d}x)\\
+ \left(\frac{ 1}{ \ell(B_{ N})} - \frac{ 1}{ \ell(B_{ M})}\right)\int_{\mathbb{ R}^{ 2p}} \left\langle \Gamma( \psi^{ (N)}(x, t)\mathbf{ 1}_{ B_{ N}}(x) , \psi^{ (N)}(y, t)\mathbf{ 1}_{ B_{ N}}(y))\, ,\, J(x,t)\right\rangle  W(x, y)\mathbf{ 1}_{ \Lambda_{ N}}(x, y)\ell({\rm d}y)\ell({\rm d}x)\\
+\frac{ 1}{ \ell(B_{ M})}\int_{\mathbb{ R}^{ 2p}} \left\langle \Gamma( \psi^{ (N)}(x, t)\mathbf{ 1}_{ B_{ N}}(x) , \psi^{ (N)}(y, t)\mathbf{ 1}_{ B_{ N}}(y))- \Gamma( \psi^{ (M)}(x, t)\mathbf{ 1}_{ B_{ M}}(x) , \psi^{ (M)}(y, t)\mathbf{ 1}_{ B_{ M}}(y))\, ,\, J(x, t)\right\rangle  W(x, y)\mathbf{ 1}_{ \Lambda_{ M}}(x, y)\ell({\rm d}y)\ell({\rm d}x)\\
:= (I) + (II) + (III) + (IV)+(V). \label{aux:eq_rhoNM}
\end{multline}
Concerning the first term in \eqref{aux:eq_rhoNM}, we have, by the property of $c$, $ \left\vert (I) \right\vert \leq L_{ c} \int_{ \mathbb{ R}^{ p}}\left\vert \rho(x, t) \right\vert^{ k}\ell({\rm d}x)$. The second term is controlled as 
\begin{multline*}
\left\vert (II) \right\vert \leq \left\vert c(0) \right\vert \left(\int_{ \mathbb{ R}^{ p}} \left\vert \rho(x, t) \right\vert^{ k} \ell({\rm d}x)\right)^{ 1-\frac{ 1}{k}} \ell(B_{ N}\setminus B_{ M})^{ \frac{ 1}{ k}} \leq \left\vert c(0) \right\vert \left\lbrace \frac{ k-1}{ k}\int_{ \mathbb{ R}^{ p}} \left\vert \rho(x, t) \right\vert^{ k} \ell({\rm d}x) + \frac{ 1}{ k} \ell(B_{ N}\setminus B_{ M})\right\rbrace.\end{multline*}
We now turn to $(III)$: by \eqref{eq:ell_BM} and \eqref{hyp:Gamma_bound},
\begin{multline}
\left\vert (III) \right\vert 
%\leq 2 L_{ \Gamma}\int_{\mathbb{ R}^{ 2p}} \left(1+ \left\vert \psi^{ (N)}(x, t)\mathbf{ 1}_{ B_{ N}}(x) \right\vert + \left\vert \psi^{ (N)}(y, t)\mathbf{ 1}_{ B_{ N}}(y) \right\vert\right) \left\vert \rho_{ N, M}(x, t) \right\vert^{ k-1} W(x, y)\mathbf{ 1}_{ \Lambda_{ N}\setminus \Lambda_{ M}}(x, y)\ell({\rm d}y)\ell({\rm d}x),\\
\leq 2 L_{ \Gamma}\int_{\mathbb{ R}^{ 2p}}  \left\vert \rho(x, t) \right\vert^{ k-1} W(x, y)\mathbf{ 1}_{ \Lambda_{ N}\setminus \Lambda_{ M}}(x, y)\ell({\rm d}y)\ell({\rm d}x)\\
+2 L_{ \Gamma}\int_{\mathbb{ R}^{ 2p}} \left\vert \psi^{ (N)}(x, t)\mathbf{ 1}_{ B_{ N}}(x) \right\vert \left\vert \rho(x, t) \right\vert^{ k-1} W(x, y)\mathbf{ 1}_{ \Lambda_{ N}\setminus \Lambda_{ M}}(x, y)\ell({\rm d}x)\ell({\rm d}y)\\
+2 L_{ \Gamma}\int_{\mathbb{ R}^{ 2p}}  \left\vert \psi^{ (N)}(y, t)\mathbf{ 1}_{ B_{ N}}(y) \right\vert \left\vert \rho(x, t) \right\vert^{ k-1} W(x, y)\mathbf{ 1}_{ \Lambda_{ N}\setminus \Lambda_{ M}}(x, y)\ell({\rm d}y)\ell({\rm d}x)\\:=(i)+(ii)+(iii). \label{aux:III_psiM}
\end{multline}
The term $(i)$ is bounded by
\begin{align*}
 (i) \leq 2L_{ \Gamma} \left(\int_{ \mathbb{ R}^{ p}}\left\vert \rho(x, t) \right\vert^{ k} \ell({\rm d}x)\right)^{ 1- \frac{ 1}{ k}} \left(\int_{ \mathbb{ R}^{ p}} \left(\int_{ \mathbb{ R}^{ p}}  W(x, y) \mathbf{ 1}_{ \Lambda_{ N}\setminus \Lambda_{ M}}(x, y)\ell({\rm d}y)\right)^{ k} \ell({\rm d}x)\right)^{ \frac{ 1}{ k}}\\ 
\leq \frac{2(k-1)L_{ \Gamma}}{ k} \int_{ \mathbb{ R}^{ p}}\left\vert \rho(x, t) \right\vert^{ k} \ell({\rm d}x) + \frac{ 2L_{ \Gamma}}{ k}\left\Vert \mathcal{ W}_{ 2} \right\Vert_{ \infty}^{ \frac{ k}{ 2}}\ell(\Lambda_{ N}\setminus \Lambda_{ M})^{ \frac{ k}{ 2}}.
\end{align*}
Secondly,
\begin{multline*} 
 (ii) \leq 
% 2L_{ \Gamma}\int_{ B_{ N}} \left\vert \psi^{ (N)}(x, t) \right\vert \left\vert \rho_{ N, M}(x, t) \right\vert^{ k-1} \left(\int_{ \mathbb{ R}^{p}} W(x, y) \mathbf{ 1}_{ \Lambda_{ N}\setminus \Lambda_{ M}}(x, y)\ell({\rm d}y)\right) \ell({\rm d}x),\\
\leq  2^{ k-1}L_{ \Gamma}\int_{ B_{ N}} \left\vert \psi^{ (N)}(x, t) \right\vert^{ k} \left(\int_{ \mathbb{ R}^{p}} W(x, y) \mathbf{ 1}_{ \Lambda_{ N}\setminus \Lambda_{ M}}(x, y)\ell({\rm d}y)\right) \ell({\rm d}x)\\
+2^{ k-1}L_{ \Gamma}\int_{ B_{ M}} \left\vert \psi^{ (N)}(x, t) \right\vert \left\vert \psi^{ (M)}(x, t) \right\vert^{ k-1} \left(\int_{ \mathbb{ R}^{p}} W(x, y) \mathbf{ 1}_{ \Lambda_{ N}\setminus \Lambda_{ M}}(x, y)\ell({\rm d}y)\right) \ell({\rm d}x),\\
\leq  2^{ k-1}L_{ \Gamma} \left(\int_{ B_{ N}} \left\vert \psi^{ (N)}(x, t) \right\vert^{ 2k} \ell({\rm d}x)\right)^{ \frac{ 1}{ 2}} \left\Vert \mathcal{ W}_{ 2} \right\Vert_{ \infty}^{ \frac{ 1}{ 2}} \ell(\Lambda_{ N}\setminus \Lambda_{ M})^{ \frac{ 1}{ 2}}\\
+2^{ k-1}L_{ \Gamma}\left(\int_{ B_{ N}} \left\vert \psi^{ (N)}(x, t) \right\vert^{ qr} \ell({\rm d}x)\right)^{ \frac{ 1}{ qr}} \left(\int_{ B_{ M}} \left\vert \psi^{ (M)}(x, t) \right\vert^{ qr^{ \ast}(k-1)}\ell({\rm d}x)\right)^{ \frac{ 1}{ qr^{ \ast}}} \\\left(\int_{ \mathbb{ R}^{p}} \left(\int_{ \mathbb{ R}^{p}} W(x, y) \mathbf{ 1}_{ \Lambda_{ N}\setminus \Lambda_{ M}}(x, y)\ell({\rm d}y)\right)^{ q^{ \ast}} \ell({\rm d}x)\right)^{ \frac{ 1}{ q^{ \ast}}},
\end{multline*}
by two successive applications of H\"older's inequality for the last term. Choosing $q=q^{ \ast}=2$ and $(r, r^{ \ast})=(k, \frac{ k}{ k-1})$, using \eqref{eq:uniform_bound_psiM}, we see that the term $(ii)$ is finally bounded by $C  \left\Vert \mathcal{ W}_{ 2} \right\Vert_{ \infty}^{ \frac{ 1}{ 2}} \ell(\Lambda_{ N}\setminus \Lambda_{ M})^{ \frac{ 1}{ 2}}$. Finally, the last term in \eqref{aux:III_psiM} is controlled as
\begin{align*}
(iii)
%=2L_{ \Gamma}\int_{ \mathbb{ R}^{p}} \int_{ B_{ N}} \left\vert \psi^{ (N)}(y, t)\right\vert \left\vert \rho_{ N, M}(x, t) \right\vert^{ k-1} W(x, y) \mathbf{ 1}_{ \Lambda_{ N}\setminus \Lambda_{ M}}(x, y)\ell({\rm d}y) \ell({\rm d}x),\\
&\leq 2^{ k-1}L_{ \Gamma} \int_{ B_{ N}^{ 2}} \left\vert \psi^{ (N)}(y, t) \right\vert \left\vert \psi^{ (N)}(x, t) \right\vert^{ k-1} W(x, y) \mathbf{ 1}_{ \Lambda_{ N}\setminus \Lambda_{ M}}(x, y)\ell({\rm d}y) \ell({\rm d}x)\\
&+ 2^{ k-1}L_{ \Gamma}\int_{ B_{ M}} \int_{ B_{ N}} \left\vert \psi^{ (N)}(y, t) \right\vert \left\vert \psi^{ (M)}(x, t) \right\vert^{ k-1} W(x, y) \mathbf{ 1}_{ \Lambda_{ N}\setminus \Lambda_{ M}}(x, y)\ell({\rm d}y) \ell({\rm d}x),\\
&:= (a)+(b).
\end{align*}
By H\"older's inequality, the term $(a)$ is controlled as
\begin{multline*}
(a)\leq 2^{ k-1}L_{ \Gamma} \left(\int_{ B_{ N}} \left\vert \psi^{ (N)}(y, t) \right\vert^{ q}\ell({\rm d}y)\right)^{ \frac{ 1}{ q}}  \left(\int_{ B_{ N}} \left\vert \psi^{ (N)}(x, t) \right\vert^{ r(k-1)}  \ell({\rm d}x)\right)^{ \frac{ 1}{ r}}\\ \left(\int_{ \mathbb{ R}^{p}}\left(\int_{ \mathbb{ R}^{p}}W(x, y)^{ q^{ \ast}} \mathbf{ 1}_{ \Lambda_{ N}\setminus \Lambda_{ M}}(x, y)\ell({\rm d}y)\right)^{ \frac{ r^{ \ast}}{ q^{ \ast}}} \ell({\rm d}x)\right)^{ \frac{ 1}{ r^{ \ast}}}.
%+ 2^{ k-1}L_{ \Gamma}\int_{ \mathbb{ R}^{p}} \left\vert \psi^{ (M)}(x, t) \right\vert^{ k-1} \int_{ \mathbb{ R}^{p}} \left\vert \psi^{ (N)}(y, t) \right\vert  W(x, y) \mathbf{ 1}_{ \Lambda_{ N}\setminus \Lambda_{ M}}(x, y)\ell({\rm d}y) \ell({\rm d}x)
\end{multline*}
For the choice $(r, r^{ \ast})=( \frac{ 2k}{ k-1}, \frac{ 2k}{ k+1})$, $(q, q^{ \ast})=( \frac{ 2- \delta}{ 1- \delta}, 2- \delta)$ for $ \delta\in (0,1)$,
\begin{multline*}
(a)\leq 2^{ k-1}L_{ \Gamma} \left(\int_{ B_{ N}} \left\vert \psi^{ (N)}(y, t) \right\vert^{ \frac{ 2- \delta}{ 1- \delta}}\ell({\rm d}y)\right)^{ \frac{ 1- \delta}{ 2- \delta}}  \left(\int_{ B_{ N}} \left\vert \psi^{ (N)}(x, t) \right\vert^{2k}  \ell({\rm d}x)\right)^{ \frac{ k-1}{ 2k}}\\ \left(\int_{ \mathbb{ R}^{p}}\left(\int_{ \mathbb{ R}^{p}}W(x, y)^{ 2- \delta} \mathbf{ 1}_{ \Lambda_{ N}\setminus \Lambda_{ M}}(x, y)\ell({\rm d}y)\right)^{ \frac{ 2k}{ (2- \delta)(k+1)}} \ell({\rm d}x)\right)^{ \frac{ k+1}{ 2k}}
%+ 2^{ k-1}L_{ \Gamma}\int_{ \mathbb{ R}^{p}} \left\vert \psi^{ (M)}(x, t) \right\vert^{ k-1} \int_{ \mathbb{ R}^{p}} \left\vert \psi^{ (N)}(y, t) \right\vert  W(x, y) \mathbf{ 1}_{ \Lambda_{ N}\setminus \Lambda_{ M}}(x, y)\ell({\rm d}y) \ell({\rm d}x)
\end{multline*}
Choosing $ \delta$ sufficiently small so that $ \frac{ 2- \delta}{ 1- \delta}<2k$ (possible since $k>1$), we obtain
\begin{equation*}
(a)\leq 2^{ k-1}L_{ \Gamma} \sup_{ N\geq1}\sup_{ t\in[0, T]}\left(\int_{ B_{ N}} \left\vert \psi^{ (N)}(y, t) \right\vert^{ 2k}\ell({\rm d}y)\right)^{ \frac{ 1}{2}} \left\Vert \mathcal{ W}_{ 2} \right\Vert_{ \infty}^{ \frac{ 1}{ 2}}  \ell(\Lambda_{ N}\setminus \Lambda_{ M})^{\frac{ \delta }{ 2(2- \delta)}}.
\end{equation*}
The same calculation gives the same estimate for $(b)$. Concerning the term $(IV)$, we have obviously $\left\vert \frac{ 1}{ \ell(B_{ N})} - \frac{ 1}{ \ell(B_{ M})} \right\vert \leq 4 \ell(B_{ N}\setminus B_{ M})$ so that it only suffices to have a uniform bound on the integral term in $(IV)$. This is indeed true, by the same calculations as for $(III)$ (with $ \mathbf{ 1}_{ \Lambda_{ N}\setminus \Lambda_{ M}}(x,y)$ replaced by $ \mathbf{ 1}_{ \Lambda_{ N}}(x,y)\leq1$). It remains to control $(V)$: by \eqref{hyp:Gamma_Lip_1},
\begin{multline*}
\left\vert (V) \right\vert \leq 
%2L_{ \Gamma} \int_{ \left(\mathbb{ R}^{ p}\right)^{ 2}} \left( \left\vert \rho_{ N, M}(x, t) \right\vert + \left\vert \rho_{ N, M}(y, t) \right\vert\right)  \left\vert \rho_{ N, M}(x, t) \right\vert^{ k-1}W(x, y)\mathbf{ 1}_{ \Lambda_{ M}}(x, y)\ell({\rm d}y)\ell({\rm d}x),\\
2L_{ \Gamma} \int_{ \left(\mathbb{ R}^{ p}\right)^{ 2}} \left\vert \rho(x, t) \right\vert^{ k}W(x, y)\mathbf{ 1}_{ \Lambda_{ M}}(x, y)\ell({\rm d}y)\ell({\rm d}x)\\
+2L_{ \Gamma} \int_{ \left(\mathbb{ R}^{ p}\right)^{ 2}} \left\vert \rho(y, t) \right\vert  \left\vert \rho(x, t) \right\vert^{ k-1}W(x, y)\mathbf{ 1}_{ \Lambda_{ M}}(x, y)\ell({\rm d}y)\ell({\rm d}x).
\end{multline*}
The first term in the sum above si easily bounded by $2L_{ \Gamma} \left\Vert \mathcal{ W}_{ 1} \right\Vert_{ \infty}\int_{\mathbb{ R}^{ p}} \left\vert \rho(x, t) \right\vert^{ k}\ell({\rm d}x)$. The second is controlled by
\begin{align*}
2L_{ \Gamma} \left\Vert \mathcal{ W}_{ 2} \right\Vert_{ \infty}^{ \frac{ 1}{ 2}}\left(\int_{ \mathbb{ R}^{p}} \left\vert \rho(y, t) \right\vert^{ 2}\ell({\rm d}y)\right)^{ \frac{ 1}{ 2}} \int_{ \mathbb{ R}^{p}}\left\vert \rho(x, t) \right\vert^{ k-1} \ell({\rm d}x)\\ 
\leq 2L_{ \Gamma} \left\Vert \mathcal{ W}_{ 2} \right\Vert_{ \infty}^{ \frac{ 1}{ 2}}\left\lbrace  \frac{ 1}{ k}\left(\int_{ \mathbb{ R}^{p}} \left\vert \rho(y, t) \right\vert^{ 2}\ell({\rm d}y)\right)^{ \frac{ k}{ 2}} + \frac{ k-1}{ k}\left(\int_{ \mathbb{ R}^{p}}\left\vert \rho(x, t) \right\vert^{ k-1} \ell({\rm d}x)\right)^{ \frac{ k}{ k-1}}\right\rbrace \\
\leq 2L_{ \Gamma}  \left\Vert \mathcal{ W}_{ 2} \right\Vert_{ \infty}^{ \frac{ 1}{ 2}}\int_{ \mathbb{ R}^{p}}\left\vert \rho(x, t) \right\vert^{k} \ell({\rm d}x),
\end{align*}
by Jensen's inequality (since $k\geq2$).

Gathering all these estimates into \eqref{aux:eq_rhoNM}, by a Gr\"onwall's lemma and the fact that $\ell( \Lambda_{ N}\setminus \Lambda_{ M})\leq \ell( \Lambda_{ M}^{ c}) \xrightarrow[ M\to\infty]{}0$, we see that the sequence $ \left(\psi^{ (M)}\mathbf{ 1}_{ B_{ M}}\right)_{ M\geq1}$ is Cauchy in $ \mathcal{ C}([0, T], L^{ k}(I, \ell))$ and hence convergent to some $ \psi(\cdot, t)$. By the same argument as before, it is easy to show that $ \psi$ is a weak solution (and hence, by Proposition~\ref{prop:uniqueness_weak_solution} the only solution) to \eqref{eq:heat_equation_weak}. Moreover, the convergence in $ \mathcal{ C}([0, T], L^{ k}(I, \ell))$ implies that $ \int_{ \mathbb{ R}^{ p}} \left\langle \psi^{ (M)}(x, t)\, ,\, J(x)\right\rangle {\rm d} x \xrightarrow[ M\to\infty]{} \int_{ \mathbb{ R}^{ p}} \left\langle \psi(x, t)\, ,\, J(x)\right\rangle {\rm d} x $ for all $t\in [0, T]$ and every test functions $J$ with compact support.

\medskip

We now turn to the righthand side of \eqref{eq:ident_psiM_nuM}. By the same procedure as in Section~\ref{sec:uniqueness_PDE_proof}, it is possible to prove that, for all test functions $ \theta \mapsto f(\theta)$,
for $\ell^{ (M)}$-almost every $x$,
\begin{multline}
\int f(\theta) \left\lbrace\nu_{s}^{(M), x}({\rm d}\theta) - \nu_{s}^{ x}({\rm d}\theta)\right\rbrace = \int P_{ 0, s}f(\theta, x) \left\lbrace\nu_{0}^{(M), x}({\rm d}\theta) - \nu_{0}^{ x}({\rm d}\theta)\right\rbrace\\
+ \int_{ 0}^{s}  \int_{ \mathbb{ R}^{ d}\times \mathbb{ R}^{ p}}  \left\lbrace\int_{ \mathbb{ R}^{ d}}\nabla_{ \theta} P_{ v, t}f(\theta, x) \cdot  \Gamma \left(\theta, \tilde{ \theta}\right) \nu_{ v}^{(M),x}({\rm d} \theta) \right\rbrace W(x, y) \left\lbrace \nu_{ v}^{(M), y}({\rm d} \tilde{ \theta}) - \nu_{ v}^{y}({\rm d} \tilde{ \theta})\right\rbrace\ell^{ (M)}({\rm d} y) {\rm d} v\\
+ \int_{ 0}^{s}  \int_{ \mathbb{ R}^{ d}\times \mathbb{ R}^{ p}}  \left\lbrace\int_{ \mathbb{ R}^{ d}}\nabla_{ \theta} P_{ v, t}f(\theta, x) \cdot  \Gamma \left(\theta, \tilde{ \theta}\right) \nu_{ v}^{(M), x}({\rm d} \theta) \right\rbrace W(x, y)\nu_{ v}^{y}({\rm d} \tilde{ \theta}) \left\lbrace \frac{ \mathbf{ 1}_{ B_{ M}}(x)}{ \ell(B_{ M})} -1\right\rbrace \ell({\rm d}y){\rm d} v\\
:= (A)+(B).
\end{multline}
Concerning the first term: it is bounded by
\begin{align*}
\left\vert (A) \right\vert \leq \frac{ C L_{ \Gamma}}{ \ell(B_{ M})}\int_{ 0}^{s} \int \sup_{ u\leq v} w_{ 1} \left(\nu_{ u}^{(M), y}, \nu_{ u}^{ y}\right) W(x, y)  \ell({\rm d} y){\rm d} v\\\leq \frac{ C L_{ \Gamma} \left\Vert \mathcal{ W}_{ 2} \right\Vert_{ \infty}^{ \frac{ 1}{ 2}}}{ \ell(B_{ M})}\int_{ 0}^{s} \left(\int \sup_{ u\leq v} w_{ 1} \left(\nu_{ u}^{(M), y}, \nu_{ u}^{ y}\right)^{ 2} \ell({\rm d} y)\right)^{ \frac{ 1}{ 2}}{\rm d} v.
\end{align*}
Concerning the second term, using the fact that $\sup_{ s\leq T, M\geq1, y\in I} \int \left\vert \theta \right\vert \nu_{ s}^{ M, y} ({\rm d}\theta)<\infty$, for some constant $C>0$
\begin{align*}
\left\vert (B) \right\vert 
&\leq C \int_{ 0}^{s} \left\vert \frac{ \mathbf{ 1}_{B_{ M}}(x)}{ \ell(B_{ M})} -1 \right\vert \int \int \left(1+ \left\vert \theta \right\vert + \left\vert \tilde{ \theta} \right\vert\right) \nu_{ v}^{(M), x}({\rm d} \theta) \nu_{ v}^{ y}({\rm d} \tilde{ \theta}) W(x, y)\ell({\rm d} y){\rm d} v\\
&\leq  C \left\Vert \mathcal{ W}_{ 1} \right\Vert_{ \infty} \left\vert \frac{ \mathbf{ 1}_{B_{ M}}(x)}{ \ell(B_{ M})} -1 \right\vert s.
\end{align*}
These estimates and Gr\"onwall Lemma gives that $\sup_{ x}\sup_{ s\in[0,T]}\int f(\theta) \left\lbrace\nu_{s}^{(M), x}({\rm d}\theta) - \nu_{s}^{ x}({\rm d}\theta)\right\rbrace \xrightarrow[M\to\infty]{}0$.

\section{Convergence of graphs: proof of Proposition~\ref{prop:conv_graph_bar_Gn}}
\label{sec:link_graph_conv}
The point of this section is to prove Proposition~\ref{prop:conv_graph_bar_Gn}. To do so, define first two other auxiliary (directed and weighted) graphs:
\begin{enumerate}
\item $ \mathcal{ H}^{ (n)}_{ 1}$, with vertex set $[n]$: in $ \mathcal{ H}^{ (n)}_{ 1}$, for all $i\neq j\in [n]$, both directed edges $i\to j$ and $j\to i$ are present and associated with the respective weights $ \kappa_{i}^{ (n)} W_{ n}(x_{ i}^{ (n)}, x_{ j}^{ (n)})$ and $ \kappa_{j}^{ (n)} W_{ n}(x_{ i}^{ (n)}, x_{ j}^{ (n)})$,
\item $ \mathcal{ H}^{ (n)}_{ 2}$, with vertex set $[n]$: in $ \mathcal{ H}^{ (n)}_{ 2}$, for all $i\neq j\in [n]$, the edge $i\to j$ (resp. $j\to i$) is present and associated with the weight $ W(x_{ i}^{ (n)}, x_{ j}^{ (n)})$ (resp. $W(x_{ i}^{ (n)}, x_{ j}^{ (n)})$).
\end{enumerate} 
\subsection{Some distances and norms on graphs and kernels}
Before proving Proposition~\ref{prop:conv_graph_bar_Gn}, we need to introduce the necessary definitions coming from graph convergence theory (see \cite{LOVASZ2006933,MR2455626,MR2825531} and references therein). Concerning the notion of cut-off distance considered in Proposition~\ref{prop:conv_graph_bar_Gn} and other related definitions, we follow here closely \cite{borgs2018,Borgs:2014aa}. In particular, we generalize here the definitions of \cite{Borgs:2014aa}, \S~2.3 to the case of directed graphs and non-symmetric kernels: let $G= (V(G), E(G))$ be a possibly directed weighted graph, where each vertex $i\in V(G)$ is associated to a weight $ \alpha_{ i}$ and each edge $i\to j$ is associated to a weight $ \beta_{ i,j}$ (where possibly $ \beta_{ i,j} \neq \beta_{ j, i}$). We define $ \alpha_{ G}= \sum_{ i\in V(G)} \alpha_{ i}$ and the kernel $W^{G}$ on $[0,1]^{ 2}$ in the following way: divide $[0,1]$ into intervals $I_{ 1}, \ldots, I_{ \left\vert V(G) \right\vert}$ of length $ \left\vert I_{ i} \right\vert = \frac{ \alpha_{ i}}{ \alpha_{ G}}$ and define
\begin{equation}
W^{G}(x, y) = \sum_{ i,j \in V(G)} \beta_{ i, j} \mathbf{ 1}_{ (x, y)\in I_{ i}\times I_{ j}}.
\end{equation}

For one kernel $W$ define
\begin{equation}
\label{eq:norm_W_L_infty_1}
\left\Vert W \right\Vert_{ \infty\to1} := \sup_{ \left\Vert f \right\Vert_{ \infty}, \left\Vert g \right\Vert_{ \infty}\leq1} \left\vert \int_{ [0, 1]^{ 2}} W(x, y) f(x) g(y) {\rm d}x{\rm d}y\right\vert
\end{equation}
as well as the cut norm
\begin{equation}
\label{eq:cut_norm_W}
\left\Vert W \right\Vert_{ \Box} := \sup_{ S, T\subset [0,1]} \left\vert \int_{ S\times T} W(x,y) {\rm d}x {\rm d}y\right\vert.
\end{equation}
Note that the norms $ \left\Vert \cdot \right\Vert_{ \infty\to 1}$ and $ \left\Vert \cdot \right\Vert_{ \Box}$ are equivalent (see \cite{Borgs:2014aa}, Eq. (2.3)): for some $C>0$, for any kernel $W$, 
\begin{equation}
\label{eq:equivalence_norms_infty1_cut}
 \left\Vert W \right\Vert_{ \Box}\leq \left\Vert W \right\Vert_{ \infty\to 1} \leq C \left\Vert W \right\Vert_{ \Box}.
\end{equation}
For any weighted directed graphs $ \mathcal{ G}$ and $ \mathcal{ G}^{ \prime}$ with vertex set $[n]$, the same nodeweights $ \alpha_{ i}$ and with respective weights $( \beta_{ i,j})_{ i,j\in \left[n\right]}$ and $( \beta_{ i,j}^{ \prime})_{ i,j\in \left[n\right]}$, define the cut-off distance
\begin{equation}
\label{eq:cut_distance_G}
d_{ \Box}(G, G^{ \prime}):= \left\Vert W^{ G} - W^{ G^{ \prime}} \right\Vert_{ \Box}=\max_{ S, T\subset [n]} \left\vert \sum_{ i\in S, j\in T} \frac{ \alpha_{ i} \alpha_{ j}}{ \alpha_{ G}^{ 2}}(\beta_{ i, j}- \beta_{ i, j}^{ \prime})\right\vert.
\end{equation}
Finally, we define the following $L^{ 1}$-distance between two kernels $W$ and $W^{ \prime}$ (not necessarily symmetric):
\begin{equation}
d_{ 1}(W, W^{ \prime}):= \left\Vert W- W^{ \prime} \right\Vert_{ 1}:= \int_{ [0,1]^{ 2}} \left\vert W(x,y) -W^{ \prime}(x,y)\right\vert {\rm d}x {\rm d}y.
\end{equation}
\subsection{Proof of Proposition~\ref{prop:conv_graph_bar_Gn}}
Proposition~\ref{prop:conv_graph_bar_Gn} is a direct consequence of Propositions~\ref{prop:distance_Gn_H1},~\ref{prop:distance_H1_H2} and~\ref{prop:distance_H2_W} below.
\begin{proposition}
\label{prop:distance_Gn_H1}
Under the hypotheses of Section~\ref{sec:comment_graph_convergence}, we have
\begin{equation}
\label{eq:distance_Gn_H1}
d_{ \Box}\left( \bar{ \mathcal{ G}}^{ (n)}, \mathcal{ H}^{ (n)}_{ 1}\right) \xrightarrow[ n\to\infty]{}0,\ \text{ with probability $1$}.
\end{equation}
\end{proposition}
\begin{proof}[Proof of Proposition~\ref{prop:distance_Gn_H1}]
Let $S, T\subseteq [n]$: we have
\begin{align*}
\mathbb{ P} \left( \left\vert \frac{ 1}{ n^{ 2}} \sum_{ i\in S, j\in T} \kappa_{i}^{ (n)} \bar \xi_{ i, j}\right\vert > \varepsilon\right) &\leq \mathbb{ P}\left( \left\vert \frac{ 1}{ n^{ 2}}\sum_{ i\in S, j\in T, i<j} \kappa_{i}^{ (n)}\bar\xi_{ i, j} \right\vert> \frac{ \varepsilon}{ 2}\right)+ \mathbb{ P}\left( \left\vert \frac{ 1}{ n^{ 2}}\sum_{ i\in S, j\in T, i>j} \kappa_{i}^{ (n)}\bar\xi_{ i, j} \right\vert> \frac{ \varepsilon}{ 2}\right).
\end{align*}
Let $U_{ +}:= \left\lbrace (i, j)\in [n]^{ 2},\ i\in S,\ j\in T, i<j\right\rbrace$. We have $ \left\vert U_{ +} \right\vert \leq \left\vert S \right\vert \left\vert T \right\vert$. We use the following result (\cite{chung2006}, Th 3.3):
\begin{proposition}
\label{prop:concentration_CL}
Let $X_{ 1}, \ldots, X_{ n}$ be independent random variables with $X_{ i}\sim \mathcal{ B}(p_{ i})$. For $X= \sum_{i=1}^{ n} a_{ i}X_{ i}$ with $a_{ i}>0$, define $ \nu= \sum_{ i=1}^{ n}a_{ i}^{ 2} p_{ i}$. Then, for $a=\max \left\lbrace a_{ 1}, \ldots, a_{ n}\right\rbrace$,
\begin{align}
\mathbb{ P} \left(X- \mathbb{ E}[X]< -\lambda\right)&\leq e^{ - \frac{ \lambda^{ 2}}{ 2 \nu}}, \label{eq:concentration_CL_1}\\
\mathbb{ P}\left(X- \mathbb{ E}[X]> \lambda\right)&\leq e^{ - \frac{ \lambda^{ 2}}{ 2(\nu + a \lambda/3)}},\label{eq:concentration_CL_2}
\end{align}
\end{proposition}
Since, by definition the variables $( \xi_{ i, j})_{ (i, j)\in U_{ +}}$ are independent (recall that $ \bar \xi_{ i,j} = \xi_{ i, j} - W_{ n}(x_{ i}, x_{ j})$, apply \eqref{eq:concentration_CL_1} with $a_{ i}= \kappa_{ i}^{ (n)}$, so that
\[\mathbb{ P}\left( \frac{ 1}{ n^{ 2}}\sum_{ (i, j)\in U_{ +}} \kappa_{i}^{ (n)}\bar\xi_{ i, j} < -\frac{ \varepsilon}{ 2}\right)
%=\mathbb{ P}\left(\sum_{ (i, j)\in U_{ +}} \kappa_{i}^{ (n)}\bar\xi_{ i, j} < -\frac{ \varepsilon n^{ 2}}{ 2}\right)
\leq \exp \left(- \frac{ \varepsilon^{ 2} n^{ 4}}{ 8 \sum_{ (i, j)\in U_{ +}} \left(\kappa_{i}^{ (n)}\right)^{ 2}W_{ n}(x_{ i}, x_{ j}) }\right).\]
Since 
\begin{align*}
\frac{ 1}{ n^{ 2}\kappa_{ n}}\sum_{ (i, j)\in U_{ +}}  \left(\kappa_{i}^{ (n)}\right)^{ 2}W_{ n}(x_{ i}, x_{ j})&\leq \frac{ 1}{ n^{ 2}}\sum_{ i,j=1}^{ n} \kappa_{i}^{ (n)} W_{ n}(x_{ i}, x_{ j})
%&\leq  \frac{ 1}{ n} \sum_{ i=1}^{ n} \left(\frac{ 1}{ n}\sum_{ j=1}^{ n}  \left\vert \kappa_{i}^{ (n)} W_{ n}(x_{ i}, x_{ j}) -W(x_{ i}, x_{ j}) \right\vert\right) + \frac{ 1}{ n^{ 2}}\sum_{ i,j=1}^{ n} W(x_{ i}, x_{ j}),\\
\leq  \delta_{ n}(\underline{ x}) + \frac{ 1}{ n^{ 2}}\sum_{ i,j=1}^{ n} W(x_{ i}, x_{ j}),
\end{align*}
using the fact that $ \frac{ 1}{ n^{ 2}} \sum_{ i, j=1}^{ n} W(x_{ i}, x_{ j}) \xrightarrow[ n\to\infty]{} \int W(x, y) {\rm d}x {\rm d}y = \left\Vert W \right\Vert_{ L^{ 1}}>0$, we have for $n$ sufficiently large,
\begin{align*}
\frac{ 1}{ n^{ 2}\kappa_{ n}}\sum_{ (i, j)\in U_{ +}} \left(\kappa_{i}^{ (n)}\right)^{ 2} W_{ n}(x_{ i}, x_{ j}) &\leq 1 + \frac{ 3 \left\Vert W \right\Vert_{ L^{ 1}}}{ 2}.
\end{align*}
For such $n$,
\begin{align*}
\mathbb{ P}\left( \frac{ 1}{ n^{ 2}}\sum_{ (i, j)\in U_{ +}} \kappa_{i}^{ (n)}\bar\xi_{ i, j} < -\frac{ \varepsilon}{ 2}\right)&\leq \exp \left(- \frac{ \varepsilon^{ 2} n^{ 2}}{ 8\kappa_{ n} \left(1 + \frac{ 3 \left\Vert W \right\Vert_{ L^{ 1}}}{ 2}\right)}\right).
\end{align*}
Moreover, by similar arguments, for $n$ sufficiently large and $ \varepsilon< 6$, we have
\begin{align*}
\mathbb{ P}\left( \frac{ 1}{ n^{ 2}}\sum_{ (i, j)\in U_{ +}} \kappa_{i}^{ (n)}\bar\xi_{ i, j} >\frac{ \varepsilon}{ 2}\right)&=\mathbb{ P}\left(\sum_{ (i, j)\in U_{ +}} \kappa_{i}^{ (n)}\bar\xi_{ i, j} >\frac{ \varepsilon n^{ 2}}{ 2}\right)\leq \exp \left(- \frac{ \varepsilon^{ 2} n^{ 2}}{ 8\kappa_{ n} \left( 2 + \frac{ 3 \left\Vert W \right\Vert_{ L^{ 1}}}{ 2}\right)}\right).
\end{align*}
Putting things together, we obtain
\begin{align*}
\mathbb{ P}\left( \left\vert \frac{ 1}{ n^{ 2}}\sum_{ (i, j)\in U_{ +}} \kappa_{i}^{ (n)}\bar\xi_{ i, j} \right\vert >\frac{ \varepsilon}{ 2} \right)\leq 2\exp \left(- \frac{ \varepsilon^{ 2} n^{ 2}}{ 8 \kappa_{ n} \left( 2 + \frac{ 3 \left\Vert W \right\Vert_{ L^{ 1}}}{ 2}\right)}\right).
\end{align*}
By the same argument on $U_{ -}$, we have
\begin{align*}
\mathbb{ P} \left( \left\vert \frac{ 1}{ n^{ 2}} \sum_{ i\in S, j\in T} \kappa_{i}^{ (n)} \bar \xi_{ i, j}\right\vert > \varepsilon\right)\leq 4 \exp \left(- \frac{ \varepsilon^{ 2} n^{ 2}}{ 8\kappa_{ n} \left( 2 + \frac{ 3 \left\Vert W \right\Vert_{ L^{ 1}}}{ 2}\right)}\right).
\end{align*}
Writing $w_{ n}\kappa_{ n}^{ 2}= u_{ n} \frac{ n}{ \log(n)}$  with $u_{ n} \xrightarrow[ n\to\infty]{}0$ (recall \eqref{hyp:alpha_n_infty}), we have from $ \frac{ 1}{ \kappa_{ n}} \leq w_{ n}$ (recall \eqref{hyp:compare_wn_alphan}) that $\kappa_{ n} \leq u_{ n} \frac{ n}{ \log(n)}$. Hence, for $ \varepsilon:= \varepsilon_{ n}= u_{ n}^{ \frac{ 1}{ 2}} \xrightarrow[ n\to\infty]{}0$, 
%\begin{align*}
%\mathbb{ P} \left( \left\vert \frac{ 1}{ n^{ 2}} \sum_{ i\in S, j\in T} \kappa_{i}^{ (n)} \bar \xi_{ i, j}\right\vert > \varepsilon_{ n}\right)\leq 4 \exp \left(- \frac{n\log(n) }{ 8\left( 2 + \frac{ 3 \left\Vert W \right\Vert_{ L^{ 1}}}{ 2}\right)}\right).
%\end{align*}
a union bound on $S, T\subseteq [n]$ gives:
\begin{align*}
\mathbb{ P} \left(\max_{ S, T\subseteq[n]}\left\vert \frac{ 1}{ n^{ 2}} \sum_{ i\in S, j\in T} \kappa_{i}^{ (n)} \bar \xi_{ i, j}\right\vert > \varepsilon_{ n}\right)\leq 4^{ n+1} \exp \left(- \frac{n\log(n) }{ 8\left( 2 + \frac{ 3 \left\Vert W \right\Vert_{ L^{ 1}}}{ 2}\right)}\right).
\end{align*}
Borel-Cantelli Lemma gives the convergence \eqref{eq:distance_Gn_H1}.
\end{proof}
\begin{proposition}
\label{prop:distance_H1_H2}
With the previous definitions, assuming that \eqref{hyp:Delta_n_1} holds, we have
\begin{equation}
\label{eq:distance_H1_H2}
d_{ \Box} \left(\mathcal{ H}_{ 1}^{ (n)}, \mathcal{ H}_{ 2}^{ (n)}\right) \xrightarrow[ n\to\infty]{}0.
\end{equation}
\end{proposition}
\begin{proof}[Proof of Proposition~\ref{prop:distance_H1_H2}]
For the norm \eqref{eq:norm_W_L_infty_1}, we have
\begin{align*}
\left\Vert W^{ \mathcal{ H}_{ 1}^{ (n)}} - W^{ \mathcal{ H}_{ 2}^{ (n)}}\right\Vert_{ \infty \to 1} &= \sup_{ \left\Vert f \right\Vert_{ \infty}, \left\Vert g \right\Vert_{ \infty}\leq 1} \left\vert \int_{ [0, 1]^{ 2}} \left(W^{ \mathcal{ H}_{ 1}^{ (n)}}(x, y) - W^{ \mathcal{ H}_{ 2}^{ (n)}}(x, y)\right)f(x) g(y) {\rm d}x {\rm d}y\right\vert,\\
&\leq \sum_{ k,l=1}^{ n}\int_{I_{ k} \times I_{ l}} \left\vert W^{ \mathcal{ H}_{ 1}^{ (n)}}(x, y) - W^{ \mathcal{ H}_{ 2}^{ (n)}}(x, y) \right\vert {\rm d}x {\rm d}y \leq \delta_{ n}(\underline{ x}) \xrightarrow[ n\to\infty]{}0.
%&= \frac{ 1}{ n^{ 2}}\sum_{ k,l=1}^{ n} \left\vert \kappa_{k}^{ (n)} W_{ n}(x_{ k}, x_{ l}) - W(x_{ k}, x_{ l}) \right\vert
\end{align*}
Hence, we obtain from \eqref{eq:equivalence_norms_infty1_cut} and \eqref{eq:cut_distance_G} the required convergence \eqref{eq:distance_H1_H2}.
\end{proof}
\begin{proposition}
\label{prop:distance_H2_W}
Under \eqref{hyp:Int_complete_W_converge}, we have
\begin{equation}
d_{ \Box} \left( \mathcal{ H}^{ (n)}_{ 2}, W\right) \xrightarrow[n \to \infty]{} 0.
\end{equation}
\end{proposition}
\begin{proof}[Proof of Proposition~\ref{prop:distance_H2_W}]
Note that by \eqref{eq:norm_W_L_infty_1}, we have that $ \left\Vert W \right\Vert_{ \infty\to1} \leq \left\Vert W \right\Vert_{ 1}$, so that, by \eqref{eq:equivalence_norms_infty1_cut}, the convergence w.r.t. the distance $d_{ 1}(\cdot, \cdot)$ implies the convergence w.r.t. the distance $d_{ \Box}(\cdot, \cdot)$. This, it suffices to prove that
\begin{equation}
d_{ 1} \left( \mathcal{ H}^{ (n)}_{ 2}, W\right) \xrightarrow[n \to \infty]{} 0.
\end{equation}
But this is exactly equivalent to \eqref{hyp:Int_complete_W_converge}. 
\end{proof}
\def\cprime{$'$}

%\bibliographystyle{abbrv}
%\bibliography{/Users/ericlucon/Seafile/doc/biblio.bib}

\begin{thebibliography}{10}

\bibitem{doi:10.1142/S0218127406014551}
D.~M. Abrams and S.~H. Strogatz.
\newblock Chimera states in a ring of nonlocally coupled oscillators.
\newblock {\em International Journal of Bifurcation and Chaos}, 16(01):21--37,
  2006.

\bibitem{MR2378491}
C.~D. Aliprantis and K.~C. Border.
\newblock {\em Infinite dimensional analysis}.
\newblock Springer, Berlin, third edition, 2006.
\newblock A hitchhiker's guide.

\bibitem{Amari:1977:DPF:2731211.2731248}
S.-I. Amari.
\newblock Dynamics of pattern formation in lateral-inhibition type neural
  fields.
\newblock {\em Biol. Cybern.}, 27(2):77--87, June 1977.

\bibitem{Bertini:2013aa}
L.~Bertini, G.~Giacomin, and C.~Poquet.
\newblock Synchronization and random long time dynamics for mean-field plane
  rotators.
\newblock {\em Probability Theory and Related Fields}, pages 1--61, 2013.

\bibitem{BHAMIDI2018}
S.~Bhamidi, A.~Budhiraja, and R.~Wu.
\newblock Weakly interacting particle systems on inhomogeneous random graphs.
\newblock {\em Stochastic Processes and their Applications}, 2018.

\bibitem{MR2825531}
C.~Borgs, J.~Chayes, L.~Lov\'asz, V.~S\'os, and K.~Vesztergombi.
\newblock Limits of randomly grown graph sequences.
\newblock {\em European J. Combin.}, 32(7):985--999, 2011.

\bibitem{Borgs:2014aa}
C.~Borgs, J.~T. Chayes, H.~Cohn, and Y.~Zhao.
\newblock {An $L^p$ theory of sparse graph convergence I: limits, sparse random
  graph models, and power law distributions}.
\newblock 01 2014.

\bibitem{borgs2018}
C.~Borgs, J.~T. Chayes, H.~Cohn, and Y.~Zhao.
\newblock An $l^{p}$ theory of sparse graph convergence ii: Ld convergence,
  quotients and right convergence.
\newblock {\em Ann. Probab.}, 46(1):337--396, 01 2018.

\bibitem{MR2455626}
C.~Borgs, J.~T. Chayes, L.~Lov{{\'a}}sz, V.~T. S{{\'o}}s, and K.~Vesztergombi.
\newblock Convergent sequences of dense graphs. {I}. {S}ubgraph frequencies,
  metric properties and testing.
\newblock {\em Adv. Math.}, 219(6):1801--1851, 2008.

\bibitem{MR2925382}
C.~Borgs, J.~T. Chayes, L.~Lov{{\'a}}sz, V.~T. S{{\'o}}s, and K.~Vesztergombi.
\newblock Convergent sequences of dense graphs {II}. {M}ultiway cuts and
  statistical physics.
\newblock {\em Ann. of Math. (2)}, 176(1):151--219, 2012.

\bibitem{MR3392551}
M.~Bossy, O.~Faugeras, and D.~Talay.
\newblock Clarification and complement to ``{M}ean-field description and
  propagation of chaos in networks of {H}odgkin-{H}uxley and
  {F}itz{H}ugh-{N}agumo neurons''.
\newblock {\em J. Math. Neurosci.}, 5:Art. 19, 23, 2015.

\bibitem{MR2871421}
P.~C. Bressloff.
\newblock Spatiotemporal dynamics of continuum neural fields.
\newblock {\em J. Phys. A}, 45(3):033001, 109, 2012.

\bibitem{MR3136844}
P.~C. Bressloff.
\newblock {\em Waves in neural media}.
\newblock Lecture Notes on Mathematical Modelling in the Life Sciences.
  Springer, New York, 2014.
\newblock From single neurons to neural fields.

\bibitem{Cabana:2015aa}
T.~Cabana and J.~Touboul.
\newblock Large deviations for spatially extended random neural networks.
\newblock 10 2015.

\bibitem{CHEVALLIER2018}
J.~Chevallier, A.~Duarte, E.~L{\"o}cherbach, and G.~Ost.
\newblock Mean field limits for nonlinear spatially extended hawkes processes
  with exponential memory kernels.
\newblock {\em Stochastic Processes and their Applications}, 2018.

\bibitem{Chiba:2016aa}
H.~Chiba and G.~S. Medvedev.
\newblock The mean field analysis for the kuramoto model on graphs i. the mean
  field equation and transition point formulas.
\newblock 12 2016.

\bibitem{chung2006}
F.~Chung and L.~Lu.
\newblock Concentration inequalities and martingale inequalities: a survey.
\newblock {\em Internet Math.}, 3(1):79--127, 2006.

\bibitem{2018arXiv180710921C}
F.~{Coppini}, H.~{Dietert}, and G.~{Giacomin}.
\newblock {A Law of Large Numbers and Large Deviations for interacting
  diffusions on Erd$\backslash$H$\{$o$\}$s-R$\backslash$'enyi graphs}.
\newblock {\em ArXiv e-prints}, July 2018.

\bibitem{2018arXiv180401263C}
J.~{Crevat}, G.~{Faye}, and F.~{Filbet}.
\newblock {Rigorous derivation of the nonlocal reaction-diffusion
  FitzHugh-Nagumo system}.
\newblock {\em arXiv e-prints}, page arXiv:1804.01263, Apr. 2018.

\bibitem{doi:10.1080/07362999808809576}
G.~Da~Prato and L.~Tubaro.
\newblock Some remarks about backward {I}t{\^o} formula and applications.
\newblock {\em Stochastic Analysis and Applications}, 16(6):993--1003, 1998.

\bibitem{MR3568168}
S.~Delattre, G.~Giacomin, and E.~Lu{\c{c}}on.
\newblock A {N}ote on {D}ynamical {M}odels on {R}andom {G}raphs and
  {F}okker--{P}lanck {E}quations.
\newblock {\em J. Stat. Phys.}, 165(4):785--798, 2016.

\bibitem{2018arXiv180709989D}
J.-F. {Delmas}, J.-S. {Dhersin}, and M.~{Sciauveau}.
\newblock {Asymptotic for the cumulative distribution function of the degrees
  and homomorphism densities for random graphs sampled from a graphon}.
\newblock {\em ArXiv e-prints}, July 2018.

\bibitem{Dembo1998}
A.~Dembo and O.~Zeitouni.
\newblock {\em Large deviations techniques and applications}, volume~38 of {\em
  Applications of Mathematics (New York)}.
\newblock Springer-Verlag, New York, second edition, 1998.

\bibitem{MR1932358}
R.~M. Dudley.
\newblock {\em Real analysis and probability}, volume~74 of {\em Cambridge
  Studies in Advanced Mathematics}.
\newblock Cambridge University Press, Cambridge, 2002.
\newblock Revised reprint of the 1989 original.

\bibitem{MR3367676}
O.~Faugeras and J.~Inglis.
\newblock Stochastic neural field equations: a rigorous footing.
\newblock {\em J. Math. Biol.}, 71(2):259--300, 2015.

\bibitem{doi:10.1137/18M1165797}
G.~Faye and Z.~Kilpatrick.
\newblock Threshold of front propagation in neural fields: An interface
  dynamics approach.
\newblock {\em SIAM Journal on Applied Mathematics}, 78(5):2575--2596, 2018.

\bibitem{1742-5468-2014-8-R08001}
S.~Gupta, A.~Campa, and S.~Ruffo.
\newblock Kuramoto model of synchronization: equilibrium and nonequilibrium
  aspects.
\newblock {\em Journal of Statistical Mechanics: Theory and Experiment},
  2014(8):R08001, 2014.

\bibitem{PhysRevE.85.066201}
S.~Gupta, M.~Potters, and S.~Ruffo.
\newblock One-dimensional lattice of oscillators coupled through power-law
  interactions: Continuum limit and dynamics of spatial fourier modes.
\newblock {\em Phys. Rev. E}, 85:066201, Jun 2012.

\bibitem{10113715M102856X}
J.~Inglis and J.~MacLaurin.
\newblock A general framework for stochastic traveling waves and patterns, with
  application to neural field equations.
\newblock {\em SIAM Journal on Applied Dynamical Systems}, 15(1):195--234,
  2016.

\bibitem{10113716M1075831}
D.~Kaliuzhnyi-Verbovetskyi and G.~Medvedev.
\newblock The semilinear heat equation on sparse random graphs.
\newblock {\em SIAM Journal on Mathematical Analysis}, 49(2):1333--1355, 2017.

\bibitem{10113717M1134007}
D.~Kaliuzhnyi-Verbovetskyi and G.~Medvedev.
\newblock The mean field equation for the kuramoto model on graph sequences
  with non-lipschitz limit.
\newblock {\em SIAM Journal on Mathematical Analysis}, 50(3):2441--2465, 2018.

\bibitem{10113713095094X}
J.~Kr{\"u}ger and W.~Stannat.
\newblock Front propagation in stochastic neural fields: A rigorous
  mathematical framework.
\newblock {\em SIAM Journal on Applied Dynamical Systems}, 13(3):1293--1310,
  2014.

\bibitem{10113715M1033927}
E.~Lang.
\newblock A multiscale analysis of traveling waves in stochastic neural fields.
\newblock {\em SIAM Journal on Applied Dynamical Systems}, 15(3):1581--1614,
  2016.

\bibitem{LOVASZ2006933}
L.~Lov{\'a}sz and B.~Szegedy.
\newblock Limits of dense graph sequences.
\newblock {\em Journal of Combinatorial Theory, Series B}, 96(6):933 -- 957,
  2006.

\bibitem{MR3689966}
E.~Lu\c{c}on and C.~Poquet.
\newblock Long time dynamics and disorder-induced traveling waves in the
  stochastic {K}uramoto model.
\newblock {\em Ann. Inst. Henri Poincar{\'e} Probab. Stat.}, 53(3):1196--1240,
  2017.

\bibitem{LucSta2014}
E.~Lu\c{c}on and W.~Stannat.
\newblock Mean field limit for disordered diffusions with singular
  interactions.
\newblock {\em Ann. Appl. Probab.}, 24(5):1946--1993, 2014.

\bibitem{LucSta2016}
E.~Lu\c{c}on and W.~Stannat.
\newblock Transition from {G}aussian to non-{G}aussian fluctuations for
  mean-field diffusions in spatial interaction.
\newblock {\em Ann. Appl. Probab.}, 26(6):3840--3909, 2016.

\bibitem{maclaurin2017mean}
J.~MacLaurin, J.~Salhi, and S.~Toumi.
\newblock Mean field dynamics of a wilson--cowan neuronal network with
  nonlinear coupling term.
\newblock {\em Stochastics and Dynamics}, page 1850046, 2017.

\bibitem{MR3238494}
G.~S. Medvedev.
\newblock The nonlinear heat equation on dense graphs and graph limits.
\newblock {\em SIAM J. Math. Anal.}, 46(4):2743--2766, 2014.

\bibitem{MR3187677}
G.~S. Medvedev.
\newblock The nonlinear heat equation on {$W$}-random graphs.
\newblock {\em Arch. Ration. Mech. Anal.}, 212(3):781--803, 2014.

\bibitem{Muller:2015aa}
P.~E. M{\"u}ller.
\newblock Path large deviations for interacting diffusions with local
  mean-field interactions.
\newblock {\em Arxiv e-print 1512.05323}, 2015.

\bibitem{Neunzert1984}
H.~Neunzert.
\newblock An introduction to the nonlinear boltzmann-vlasov equation.
\newblock In {\em Kinetic Theories and the Boltzmann Equation}, pages 60--110.
  Springer Berlin Heidelberg, 1984.

\bibitem{2018arXiv180706898O}
R.~I. {Oliveira} and G.~{Reis}.
\newblock {Interacting diffusions on random graphs with diverging degrees:
  hydrodynamics and large deviations}.
\newblock {\em arXiv e-prints}, page arXiv:1807.06898, July 2018.

\bibitem{PhysRevLett.106.234102}
I.~Omelchenko, Y.~Maistrenko, P.~H\"ovel, and E.~Sch\"oll.
\newblock Loss of coherence in dynamical networks: Spatial chaos and chimera
  states.
\newblock {\em Phys. Rev. Lett.}, 106:234102, Jun 2011.

\bibitem{101137130918721}
J.~Rankin, D.~Avitabile, J.~Baladron, G.~Faye, and D.~Lloyd.
\newblock Continuation of localized coherent structures in nonlocal neural
  field equations.
\newblock {\em SIAM Journal on Scientific Computing}, 36(1):B70--B93, 2014.

\bibitem{SznitSflour}
A.-S. Sznitman.
\newblock Topics in propagation of chaos.
\newblock In {\em {\'E}cole d'{\'E}t{\'e} de {P}robabilit{\'e}s de
  {S}aint-{F}lour {XIX}---1989}, volume 1464 of {\em Lecture Notes in Math.},
  pages 165--251. Springer, Berlin, 1991.

\bibitem{MR2998591}
J.~Touboul.
\newblock Limits and dynamics of stochastic neuronal networks with random
  heterogeneous delays.
\newblock {\em J. Stat. Phys.}, 149(4):569--597, 2012.

\bibitem{MR2459454}
C.~Villani.
\newblock {\em Optimal transport}, volume 338 of {\em Grundlehren der
  Mathematischen Wissenschaften [Fundamental Principles of Mathematical
  Sciences]}.
\newblock Springer-Verlag, Berlin, 2009.
\newblock Old and new.

\bibitem{Wilson:1972aa}
H.~R. Wilson and J.~D. Cowan.
\newblock Excitatory and inhibitory interactions in localized populations of
  model neurons.
\newblock {\em Biophysical journal}, 12(1):1--24, 01 1972.

\end{thebibliography}

\end{document}